\documentclass[11pt]{article}
\usepackage{latexsym,amsfonts,amssymb,amsmath,amsthm}
\usepackage{graphicx}
\usepackage{cite}
\usepackage[usenames,dvipsnames]{color}
\usepackage{ulem}
\usepackage[colorlinks=true]{hyperref}
\hypersetup{urlcolor=blue, citecolor=red}

\newcommand{ \bl}{\color{blue}}

\newcommand{\norm}[1]{||#1||}

\newcommand{\ucite}[1]{\cite{#1}}
\begin{document}
\setlength{\baselineskip}{16pt}

\parindent 0.5cm
\evensidemargin 0cm \oddsidemargin 0cm \topmargin 0cm \textheight
22cm \textwidth 16cm \footskip 2cm \headsep 0cm

\newtheorem{theorem}{Theorem}[section]
\newtheorem{lemma}[theorem]{Lemma}
\newtheorem{proposition}[theorem]{Proposition}
\newtheorem{definition}{Definition}[section]
\newtheorem{example}{Example}[section]
\newtheorem{corollary}[theorem]{Corollary}

\newtheorem{remark}{Remark}[section]
\newtheorem{property}[theorem]{Property}
\numberwithin{equation}{section}
\newtheorem{mainthm}{Theorem}
\newtheorem{mainlem}{Lemma}

\numberwithin{equation}{section}

\def\p{\partial}
\def\I{\textit}
\def\R{\mathbb R}
\def\C{\mathbb C}
\def\u{\underline}
\def\l{\lambda}
\def\a{\alpha}
\def\O{\Omega}
\def\e{\epsilon}
\def\ls{\lambda^*}
\def\D{\displaystyle}
\def\wyx{ \frac{w(y,t)}{w(x,t)}}
\def\imp{\Rightarrow}
\def\tE{\tilde E}
\def\tX{\tilde X}
\def\tH{\tilde H}
\def\tu{\tilde u}
\def\d{\mathcal D}
\def\aa{\mathcal A}
\def\DH{\mathcal D(\tH)}
\def\bE{\bar E}
\def\bH{\bar H}
\def\M{\mathcal M}
\renewcommand{\labelenumi}{(\arabic{enumi})}

\def\disp{\displaystyle}
\def\undertex#1{$\underline{\hbox{#1}}$}
\def\card{\mathop{\hbox{card}}}
\def\sgn{\mathop{\hbox{sgn}}}
\def\exp{\mathop{\hbox{exp}}}
\def\OFP{(\Omega,{\cal F},\PP)}
\newcommand\JM{Mierczy\'nski}
\newcommand\RR{\ensuremath{\mathbb{R}}}
\newcommand\CC{\ensuremath{\mathbb{C}}}
\newcommand\QQ{\ensuremath{\mathbb{Q}}}
\newcommand\ZZ{\ensuremath{\mathbb{Z}}}
\newcommand\NN{\ensuremath{\mathbb{N}}}
\newcommand\PP{\ensuremath{\mathbb{P}}}
\newcommand\abs[1]{\ensuremath{\lvert#1\rvert}}

\newcommand\normf[1]{\ensuremath{\lVert#1\rVert_{f}}}
\newcommand\normfRb[1]{\ensuremath{\lVert#1\rVert_{f,R_b}}}
\newcommand\normfRbone[1]{\ensuremath{\lVert#1\rVert_{f, R_{b_1}}}}
\newcommand\normfRbtwo[1]{\ensuremath{\lVert#1\rVert_{f,R_{b_2}}}}
\newcommand\normtwo[1]{\ensuremath{\lVert#1\rVert_{2}}}
\newcommand\norminfty[1]{\ensuremath{\lVert#1\rVert_{\infty}}}

\title{Almost automorphically  and almost periodically forced circle flows of almost periodic parabolic equations on $S^1$}

\author {
\\
Wenxian Shen\\
Department of Mathematics and Statistics\\
 Auburn University, Auburn, AL 36849, USA
\\
\\
Yi Wang\thanks{Partially supported by NSF of China No.11771414, 11471305, Wu Wen-Tsun Key Laboratory and the Fundamental
Research Funds for the Central Universities.} \\
School of Mathematical Science\\
 University of Science and Technology of China
\\ Hefei, Anhui, 230026, P. R. China
\\
\\
 Dun Zhou\thanks{Corresponding author, E-mail address: zhd1986@mail.ustc.edu.cn. On leave from School of Mathematical Science,
 University of Science and Technology of China, Hefei, Anhui, 230026, P. R. China and partially supported by NSF of China No.11601498.}\\
 Department of Mathematics
 \\ Nanjing University of Science and Technology
 \\ Nanjing, Jiangsu, 210094, P. R. China
 \\
\\
}
\date{}

\maketitle

\begin{abstract}
We consider the skew-product semiflow which is generated by a scalar reaction-diffusion equation
\begin{equation*}
u_{t}=u_{xx}+f(t,u,u_{x}),\,\,t>0,\,x\in S^{1}=\mathbb{R}/2\pi \mathbb{Z},
\end{equation*}
where $f$ is uniformly almost periodic in $t$. The structure of the minimal set $M$ is thoroughly investigated under the assumption that the center space $V^c(M)$ associated with $M$ is no more than $2$-dimensional. Such situation naturally occurs while, for instance, $M$ is hyperbolic or uniquely ergodic. It is shown in this paper that $M$ is a $1$-cover of the hull $H(f)$ provided that $M$ is hyperbolic (equivalently, ${\rm dim}V^c(M)=0$). If ${\rm dim}V^c(M)=1$ (resp.
${\rm dim}V^c(M)=2$ with ${\rm dim}V^u(M)$ being odd), then either $M$ is an almost $1$-cover of $H(f)$ and topologically conjugate to a minimal flow in $\mathbb{R}\times H(f)$; or $M$ can be (resp. residually) embedded into an almost periodically
(resp. almost automorphically)  forced circle-flow $S^1\times H(f)$.

When $f(t,u,u_x)=f(t,u,-u_x)$ (which includes the case $f=f(t,u)$), it is proved that any minimal set $M$ is an almost $1$-cover of $H(f)$. In particular, any hyperbolic minimal set $M$ is a $1$-cover of $H(f)$. Furthermore, if ${\rm dim}V^c(M)=1$, then $M$ is either a $1$-cover of $H(f)$ or is topologically conjugate to a minimal flow in $\mathbb{R}\times H(f)$. For the general spatially-dependent nonlinearity $f=f(t,x,u,u_{x})$, we show that any stable or linearly stable minimal invariant set $M$ is residually embedded into $\mathbb{R}^2\times H(f)$.
\end{abstract}

\section{Introduction}
In this paper we investigate the large time behavior of bounded solutions for the scalar reaction-diffusion equations on the circle
\begin{equation}\label{equation-1}
u_{t}=u_{xx}+f(t,x, u,u_{x}),\,\,t>0,\,x\in S^{1}=\mathbb{R}/2\pi \mathbb{Z},
\end{equation}
where $f:\RR\times  \RR\times \RR\times\RR\to \RR$ is $C^{3}$, $f(t,x+2\pi,u,u_x)=f(t,x,u,u_x)$, and $f(t,x,u,p)$ together with all its derivatives (up to order $3$) are almost periodic in $t$ uniformly for $(x,u,p)$ in compact subsets.

To carry out our study for the non-autonomous equation \eqref{equation-1}, we will embed it into a skew-product semiflow in the following way. The function $f$ generates a family
$\{f_{\tau}:\tau \in \mathbb{R}\}$ in the space of continuous functions $C(\mathbb{R}\times S^{1} \times \mathbb{R} \times \mathbb{R},\mathbb{R})$ equipped with the compact open topology. Here $f_{\tau}(t,x,u,p)=f(t+\tau,x,u,p)(\tau \in \RR)$ denotes the time-translation of $f$. Let $H(f)$, called the hull of $f$, be the closure of $\{f_{\tau}:\tau\in \mathbb{R}\}$ in
the compact open topology. By the Ascoli--Arzela theorem, $H(f)$ is a compact metric space and every $g\in H(f)$ is uniformly almost periodic and has the same regularity as $f$. The action of time-translation $g\cdot t\equiv g_{t}\,(g\in H(f))$ defines a compact minimal flow on $H(f)$ (\cite{Sell,Shen1998}). This means that $H(f)$ is the only nonempty compact subset of itself
that is invariant under the flow $g\cdot t$. By introducing the hull $H(f)$, equation \eqref{equation-1}
gives rise to a family of equations associated to each $g\in H(f)$,
\begin{equation}\label{equation-lim1}
u_{t}=u_{xx}+g(t,x,u,u_{x}),\,\,\quad t>0,\quad x\in S^{1}.
\end{equation}

Let $X=X^{\alpha}$ ($\frac{1}{2}<\alpha<1$)  be the fractional power space associated with the operator $u\rightarrow
-u_{xx}:H^{2}(S^{1})\rightarrow L^{2}(S^{1})$, then the embedding
relation $X\hookrightarrow C^{1}(S^{1})$ is satisfied (that is, $X$ is compactly embedded in $C^1(S^1)$). For any $u\in X$, \eqref{equation-lim1} admits (locally) a unique solution $\varphi(t,\cdot;u,g)$ in $X$ with $\varphi(0,\cdot;u,g)=u(\cdot)$. This solution also continuously depends on $g\in H(f)$ and $u\in
X$. Therefore, \eqref{equation-lim1} defines a (local) skew product semiflow $\Pi^{t}$ on $X\times
H(f)$:
\begin{equation}\label{equation-lim2}
\Pi^{t}(u,g)=(\varphi(t,\cdot;u,g),g\cdot t),\quad t>0.
\end{equation}
Following from the work in \cite{Hen} and the standard a priori estimates for parabolic equations, it is known that if
$\varphi(t,\cdot;u,g) (u\in X)$ is bounded in $X$ in the existence interval of the solution, then $u$ is a globally defined classical solution. In the terminology of the skew-product semiflow \eqref{equation-lim2}, the study of dynamics of \eqref{equation-lim1} gives rise to the problem of understanding the $\omega$-limit set $\omega(u,g)$ of the bounded semi-orbit $\Pi^t(u,g)$ in $X\times H(f)$. Note that, for any $\delta>0$, $\{\varphi(t,\cdot;u,g):t\ge \delta\}$ is relatively compact in $X$. As a consequence, $\omega(u,g)$ is a nonempty connected compact subset of $X\times H(f)$. It is further known that $\Pi^{t}$ on the $\omega$-limit set $\omega(u,g)$ has a unique continuous backward time extension (see, e.g. \cite{hale1988asymptotic}).

In the case where $f$ is independent of $t$ (i.e., the autonomous case) or, equivalently, if $H(f)=\{f\}$, Fiedler and Mallet-Paret \cite{Fiedler} have shown the well-known Poincar\'{e}-Bendixson type Theorem for system \eqref{equation-1}. It states that any $\omega$-limit set $\omega(u)$ is either a single periodic orbit or it consists of equilibria and connecting (homoclinic and heteroclinic) orbits. In particular, if $f$ does not depend on $x$ (also called spatially-homogeneous), i.e., $f=f(u,u_x)$, the solution semiflow commutes with the natural action of shifting $x\in S^1$. Due to such $S^1$-equivariance,  Massatt \cite{Massatt1986} and Matano \cite{Matano} showed independently that any periodic orbit is a rotating wave $u=\phi(x-ct)$ for some $2\pi$-periodic function $\phi$ and constant $c$; and hence, $\omega(u)$ is either itself a single rotating wave, or a set of equilibria differing only by phase shift in $x$. Recently, transversality of the stable and unstable manifolds of hyperbolic equilibria and periodic orbits for autonomous system \eqref{equation-lim1} has been established in \cite{CR,JR1}. Based on this, Joly and Raugel \cite{JR2} have proved the generic Morse-Smale property for the system.

In the case that $f$ is time-periodic with period $1$ (equivalently, $H(f)$ is homeomorphic to the circle $\mathcal{T}^1=\mathbb{R}/\mathbb{Z}$), Chen and Matano \cite{Chen1989160} proved that for $f=f(t,u)$ independent of $x$ and $u_x$, the $\omega$-limit set $\omega(u)$ of any bounded solution consists of a unique time-periodic orbit with period $1$. If $f=f(t,u,u_x)$ is independent of $x$, Sandstede and Fiedler \cite{SF1} showed that the  $\omega$-limit set $\omega(u)$ can be viewed as a subset of the two-dimensional torus $\mathcal{T}^1\times S^1$ carrying a linear flow. However, one can not expect a simple asymptotic behavior of solutions for the general nonlinearity $f=f(t,x,u,u_x)$. As a matter of fact, Sandstede and Fiedler \cite{SF1} have further pointed out that chaotic behavior exhibited by any time-periodic planar vector field can also be found in certain time-periodic equation with the nonlinearity $f=f(t,x,u,u_x)$. On the other hand, Tere\v{s}\v{c}\'{a}k \cite{Te} proved that any $\omega$-limit set of the Poincar\'{e} map generated by the time-periodic system with $f=f(t,x,u,u_x)$ can be imbedded into a $2$-dimensional plane.

In the language of skew-product semiflows, Tere\v{s}\v{c}\'{a}k's result \cite{Te} implies that
each $\omega$-limit set $\omega(u,g)$ (with $g\in H(f)\sim\mathcal{T}^1$) of \eqref{equation-lim2} can be imbedded into $\RR^2\times H(f)$ (see Definition \ref{residual-imbed}). In particular, in the spatially-homogeneous case when $f=f(t,u,u_x)$ is time-periodic,  the result by  Sandstede and Fiedler \cite{SF1} entails that each $\omega$-limit set $\omega(u,g)$ can be imbedded into the periodically-forced circle flow $S^1\times H(f)$.

In nature, large quantities of systems evolve influenced by external effects which are roughly, but not
exactly periodic, or under environmental forcing which exhibits different, noncommensurate
periods. As a consequence, models with such time dependence are characterized more
appropriately by quasi-periodic or almost periodic equations or even by certain nonautonomous
equations rather than by periodic ones. Consequently, time non-periodic equations have been attracting more attention recently.

The current paper is devoted to the study of dynamics of time almost-periodic scalar parabolic equations with periodic boundary conditions.
For separated boundary conditions, one can refer to a series of work by Shen and Yi \ucite{Shen1995114,ShenYi-2,ShenYi-TAMS95,ShenYi-JDDE96,Shen1998}.
Among others, they \cite{Shen1995114,ShenYi-TAMS95} have proved that any minimal invariant set $M$ of the skew-product semiflow is an almost $1$-$1$ cover of $H(f)$; and hence, $M$ is an almost automorphic minimal set. In particular, if $M$ is hyperbolic, then $M$ is a $1$-$1$ cover of $H(f)$ (see \cite{ShenYi-2}). As in \ucite{Shen1995114,ShenYi-2,ShenYi-TAMS95,ShenYi-JDDE96,Shen1998}, the zero number properties developed in \cite{2038390,H.MATANO:1982} play important roles in their studies. With periodic boundary condition the zero number can still be applied. It does not yield the almost automorphy of the minimal sets in general, however. To the best of our knowledge, the structures of the minimal sets have been hardly studied for scalar parabolic equations, even for the spatially-homogeneous system $f=f(t,u,u_x)$, with periodic boundary condition (see \cite{Chow1995,Wang} for some partial related results). As a consequence, we will try to initiate our research on this aspect.

We first  consider the spatially-homogeneous case of $f=f(t,u,u_x)$. The structure of the minimal set $M$ for \eqref{equation-lim2}  will be thoroughly investigated  under the assumptions
 that the center space $V^c(M)$ associated with $M$ is no more than $2$-dimensional.  Such situation naturally occurs, for instance, while $M$ is uniquely ergodic (see Theorem \ref{ergodic-thm}(1)) or $M$ is hyperbolic. We denote by $V^u(M)$ the unstable space associated  to $M$.
 Among others, we will prove

\medskip

\noindent $\bullet$ {\it  If ${\rm dim}V^c(M)=2$ with ${\rm dim}V^u(M)$ being odd, then either

\begin{itemize}
\item[{\rm (i)}]  $M$ is  an almost $1$-cover of $H(f)$ and topologically conjugate to a minimal flow in $\mathbb{R}\times H(f)$ (see Theorem \ref{almost-automorphic-thm}(1)); or
\vspace{-0.1in}
\item[{\rm (ii)}]  $M$ can be residually embedded into an almost automorphically  forced circle-flow $S^1\times H(f)$ (see Theorem \ref{sysm-embed}(1)).
\end{itemize}
\vspace{-0.1in}
}

\medskip
\smallskip

\noindent $\bullet$ {\it  If ${\rm dim}V^c(M)=1$,  then either
\begin{itemize}
\item[{\rm (i)}] $M$ is an almost $1$-cover of $H(f)$ and the dynamics on $M$ is topologically conjugate to a   minimal flow in $\mathbb{R}\times H(f)$ (see Theorem \ref{almost-automorphic-thm}(1)); or
\vspace{-0.1in}
\item[{\rm (ii)}] $M$ is normally hyperbolic and can be embedded into an almost periodically forced circle-flow $S^1\times H(f)$ (see Theorem \ref{sysm-embed}(2)).
\end{itemize}
\vspace{-0.1in}
}

\medskip
\smallskip

\noindent $\bullet$ {\it  If ${\rm dim}V^c(M)=0$ (equivalently, $M$ is called {\it hyperbolic}), then $M$ is a $1$-cover of $H(f)$
and the dynamics  on $M$ is topologically conjugate to an almost periodic  minimal flow in $\mathbb{R}\times H(f)$
(see Theorem \ref{almost-automorphic-thm}(4)).}

\medskip
\smallskip

\noindent $\bullet$ {\it Any spatially homogeneous minimal set $M$ is an almost $1$-cover of $H(f)$ and the dynamics on $M$ is topologically conjugate to an almost automorphic  minimal flow in $\mathbb{R}\times H(f)$
(see Theorem \ref{almost-automorphic-thm}(1)).}

\bigskip

\noindent $\bullet$ {\it If $M$ is linearly stable, then $M$ is spatially homogeneous and hence is an almost $1$-cover of $H(f)$
(see Theorem \ref{almost-automorphic-thm}(2))}.

\vskip 2mm
We also remark that, for $f=f(t,u,u_x)$, if $M$ is a spatially homogeneous minimal set or uniquely ergodic, then one may obtain the oddness of ${\rm dim}V^u(M)$ (see Theorem \ref{ergodic-thm}). In general, it remains open whether ${\rm dim}V^u(M)$ is odd provided that
${\rm dim}V^u(M)\not =0$.
\vskip 2mm

Comparing with the results in\cite{ShenYi-2,ShenYi-JDDE96}  for separated boundary conditions, one can still observe here the $1$-cover property of the hyperbolic minimal sets; while for the case ${\rm dim}V^c(M)=1$, we obtained the new phenomena for periodic boundary conditions that $M$  can be embedded into an almost periodically forced circle flow $S^1\times H(f)$, which is a natural generalization of the rotating waves in \cite{Massatt1986,Matano} (autonomous cases) and the two-dimensional torus flow in \cite{SF1} (time-periodic cases) to time almost-periodic systems. It also deserves to point out that an almost periodically forced circle flow could still be very complicated (see Huang and Yi \cite{HuYi} and the references therein). The new phenomena we discovered here reinforce the appearance of the almost periodically forced circle flow in the infinite-dimensional dynamical systems generated by certain evolutionary equations.

\vskip 1mm
When $f(t,u,p)=f(t,u,-p)$ in \eqref{equation-1} (which includes the case $f=f(t,u)$), more information of the structure of a minimal set for \eqref{equation-lim2} can be obtained. We will show, in this case,

\medskip

\noindent $\bullet$ {\it  Any minimal set $M$ is an almost $1$-cover of $H(f)$. Moreover,  $M$ is a $1$-cover of $H(f)$
 provided that $M$ is hyperbolic or  ${\rm dim}V^c(M)=1$  (see Theorem \ref{almost-periodic-thm}).}

\medskip

\noindent  Thus, we have also generalized the convergence results in \cite{Chen1989160} from time-periodic systems to time almost-periodic systems.

\vskip 1mm
Finally, we will consider the general nonlinearity $f=f(t,x,u,u_x)$. We will show

\medskip

\noindent $\bullet$ {\it
Any linearly stable or stable minimal set $M$ is
{\it residually embedded} (see Definition \ref{residual-imbed}) into $\mathbb{R}^2\times H(f)$. In particular, the $\omega$-limit set of any uniformly stable bounded trajectory can be embedded into $\mathbb{R}^2\times H(f)$ (see Theorem \ref{main-result-thm}).}

\medskip

\noindent The above embedding property for the minimal sets partially extends the
results of Fiedler and Mallet-Paret \cite{Fiedler} (autonomous cases) and Tere\v{s}\v{c}\'{a}k \cite{Te}  (time-periodic cases) to time almost-periodic systems.

The paper is organized as follows. In Section 2, we agree on some notations, give relevant
definitions and preliminary results, including the Floquet bundles theory and the invariant manifolds theory for skew-product semiflows, which will be important to our proofs.
 In Section 3, we will investigate the skew-product semiflow \eqref{equation-lim2} generated by \eqref{equation-lim1} with $f=f(t,u,u_x)$. The structure of the minimal set $M$ is investigated under the assumption that ${\rm dim}V^c(M)=1$, or ${\rm dim}V^c(M)=2$ with ${\rm dim}V^u(M)$ being odd (Theorem \ref{sysm-embed}). The new phenomena is found in this section that $M$  can be embedded into an almost periodically forced circle flow $S^1\times H(f)$. We also obtain in this section that the unique ergodicity of $M$ implies that ${\rm dim}V^c(M)\le 2$. In Section 4, we focus on the (almost) $1$-cover property of minimal sets for  \eqref{equation-lim1} with $f=f(t,u,u_x)$. In particular, when $f=f(t,u)$, any minimal set $M$ is an almost $1$-cover. In Section 5, we study the embedding properties of linearly stable and stable minimal sets of \eqref{equation-lim2} in the general case $f=f(t,x,u,u_x)$.

\section{Notations and preliminaries results}
 In this section, we summarize some preliminary materials to be used in later sections.
We start by summarizing some lifting properties of compact dynamical systems.  Next,  we
give a brief review about almost periodic (automorphic) functions. We then present some
basic properties of zero numbers of solutions for linear parabolic equations. Finally, we present some basic properties about  Floquet bundles and invariant subspaces of linear parabolic  equations on $S^1$ and some basic properties about invariant manifolds of nonlinear parabolic equations on $S^1$.

\subsection{Lifting properties of compact dynamical systems}

 Let $Y$ be a compact metric space with metric $d_{Y}$, and
$\sigma:Y\times \RR\to Y, (y,t)\mapsto y\cdot t$ be a continuous
flow on $Y$, denoted by $(Y,\sigma)$ or $(Y,\RR)$. A subset
$S\subset Y$ is {\it invariant} if $\sigma_t(S)=S$ for every $t\in
\RR$. A subset $S\subset Y$ is called {\it minimal} if it is
compact, invariant and the only non-empty compact invariant subset
of it is itself.  Every compact and $\sigma$-invariant set contains a minimal subset and a subset $S$ is minimal if and only if every trajectory is
 dense. The continuous flow $(Y,\sigma)$ is called to be
{\it recurrent} or {\it minimal} if $Y$ is minimal. We
say that the flow $(Y,\sigma)$ is {\it distal} when, for each pair
$y_1,y_2$ of different elements of $Y$, there is a $\delta>0$ such
that $d_Y(y_1\cdot t,y_2\cdot t)>\delta$ for every $t\in
\mathbb{R}$.

If $(Z,\RR)$ is another continuous flow, {\it
a flow homomorphism} from $(Z,\RR)$ to $(Y,\sigma)$ is a continuous
mapping $p$ from $Z$ to $Y$ such that $p(z\cdot t)=p(z)\cdot t$ for
all $z\in Z$ and $t\in \RR$. An onto flow homomorphism is called a {\it flow epimorphism}.
It is easy to see that if $(Y,\sigma)$ is minimal then the flow homomorphism $p$ from $Z$ to $Y$ is a flow epimorphism. The following lemma is adopted from \cite{Shen1998} and will play important roles in our
forthcoming sections.
\begin{lemma}\label{epimorphism-thm}
Let $p:(Z,\mathbb{R})\rightarrow(Y,\mathbb{R})$ be an epimorphism of flows, where $Z,Y$ are compact
metric spaces. Then the set
\begin{center}
$Y'=\{y_{0}\in Y: for \ any\ z_{0}\in p^{-1}(y_{0}),\ y\in Y \ and\ any\ sequence\
\{t_{i}\}\subset \RR \ with\ y\cdot t_{i}\rightarrow y_{0},\ there\ is\ a\ sequence\
\{z_{i}\}\in p^{-1}(y)\ such\ that\ z_{i}\cdot t_{i}\rightarrow z_{0} \}$
\end{center}
\vskip -2mm
is residual and invariant. In particular, if $(Z,\mathbb{R})$ is minimal and distal, then $Y'=Y$.
\end{lemma}
\begin{proof}
See \cite[Lemma I.2.16]{Shen1998} and the remarks below it.
\end{proof}

 Let $X,Y$ be metric spaces and $(Y,\sigma)$ be a compact flow (called the base flow). Let also
 $\RR^+=\{t\in \RR:t\ge 0\}$. A
 skew-product semiflow $\Pi^t:X\times Y\rightarrow X\times Y$ is a semiflow
 of the following form
 \begin{equation}\label{skew-product-semiflow}
 \Pi^{t}(u,y)=(\varphi(t,u,y),y\cdot t),\quad t\geq 0,\, (u,y)\in X\times Y,
 \end{equation}
satisfying (i) $\Pi^{0}={\rm Id}_X$ and (ii) the co-cycle property:
$\varphi(t+s,u,y)=\varphi(s,\varphi(t,u,y),y\cdot t)$ for each $(u,y)\in X\times Y$ and $s,t\in \RR^+$.
 A subset $A\subset
X\times Y$ is {\it positively invariant} if $\Pi^t(A)\subset A$ for
all $t\in \RR^+$. The {\it forward orbit} of any $(u,y)\in X\times
Y$ is defined by $\mathcal{O}^+(u,y)=\{\Pi^t(u,y):t\ge 0\}$, and the
{\it $\omega$-limit set} of $(u,y)$ is defined by
$\omega(u,y)=\{(\hat{u},\hat{y})\in X\times Y:\Pi^{t_n}(u,y)\to
(\hat{u},\hat{y}) (n\to \infty) \textnormal{ for some sequence
}t_n\to \infty\}$.

A {\it flow extension} of a skew-product semiflow $\Pi^t$
 is a continuous skew-product flow $\hat{\Pi}^t$ such that $\hat{\Pi}^t(u,y)=\Pi^t(u,y)$ for each
$(u,y)\in X\times Y$ and $t\in \RR^+$. A compact positively
invariant subset is said to admit {\it a flow extension} if the
semiflow restricted to it does. Actually, a compact positively
invariant set $K\subset X\times Y$ admits a flow extension if every
point in $K$ admits a unique backward orbit which remains inside the
set $K$ (see \cite[part II]{Shen1998}). A compact positively
invariant set $K\subset X\times Y$ for $\Pi^t$ is called {\it minimal} if it
does not contain any other nonempty compact positively invariant set
than itself.

Let $K\subset X\times Y$ be a
positively invariant set for $\Pi^t$ which admits a flow extension.
Let also $p:X\times Y\to Y$ be the natural projection. Then $p$ is a flow
homomorphism for the flows $(K,\RR)$ and $(Y,\sigma)$.  Moreover, $K\subset X\times Y$ is called an
{\it almost $1$-cover} ({\it $1$-cover}) of $Y$ if card$(p^{-1}(y)\cap K)=1$
 for at least one $y\in Y$ (for any $y\in Y$).

 Now let us recall some definitions concerning the stability of the trajectories of the
semiflows.

\begin{definition}\label{stability}
{\rm  Let $d_X$ be the metric on $X$. \begin{description}
\item[{\rm (1)}] (Stability) A forward orbit $\mathcal{O}^+(u_0,y_0)$ of
\eqref{skew-product-semiflow} is said to be {\it
stable} if for every $\varepsilon>0$ and $s\ge 0$, there is a
$\delta=\delta(\varepsilon,s)>0$ such that, for every $u\in X$, if
$d_X(\varphi(s,u_0,y_0),\varphi(s,u,y_0))\le \delta$ then
$$d_X(\varphi(t+s,u_0,y_0),\varphi(t+s,u,y_0))<\varepsilon \textnormal{ for each
}t\ge 0.$$ A minimal set $M$ is called {\it stable} if $\mathcal{O}^+(u_*,y_*)$ is stable
for any point $(u_*,y_*)\in M$;
\item[{\rm (2)}] (Uniform stability) A forward orbit $\mathcal{O}^+(u_0,y_0)$ of \eqref{skew-product-semiflow}
is said to be {\it uniformly stable} if for every $\varepsilon>0$
there is a $\delta=\delta(\varepsilon)>0$, called the {\it modulus
of uniform stability}, such that, for every $u\in X$, if $s\ge 0$
and $d_X(\varphi(s,u_0,y_0),\varphi(s,u,y_0))\le \delta(\varepsilon)$ then
$$d_X(\varphi(t+s,u_0,y_0),\varphi(t+s,u,y_0))<\varepsilon \textnormal{ for
each }t\ge 0.$$
\end{description}}
\end{definition}

\begin{remark}\label{stable-all}
{\rm  If $\mathcal{O}^+(u,y)$ is relatively compact and uniformly
stable, then for every point $(u_*,y_*)\in \omega(u,y)$,
$\mathcal{O}^+(u_*,y_*)$ is uniformly stable with the same modulus
of uniform stability as that of the $\mathcal{O}^+(u,y)$ (see
\cite{Sell,Shen1998})}.
\end{remark}

We now assume additionally that $X$ is a Banach space. Assume also that the cocycle $\varphi$ in \eqref{skew-product-semiflow}
is
$C^{1+\alpha}$ ($0<\alpha \leq 1$) for $u\in X$, that is, $\varphi$ is $C^{1}$ in
$u$, and the derivative $\varphi_{u}$ is continuous in $y\in Y,\,t>0$ and is $C^{\alpha}$ in
$u$; and moreover, for any $v\in X$,
$$\varphi_{u}(t,u,y)v\rightarrow v \quad \textnormal{ as }\quad t\rightarrow 0^+,$$
uniformly for $(u,y)$ in compact subsets of $X\times Y$. Let $K\subset X\times Y$ be a
compact, positively invariant set which admits a flow extension. Define
$\Phi(t,u,y)=\varphi_{u}(t,u,y)$
for $(u,y)\in K,\,t\geq 0$. Then the operator $\Phi$ generates a linear skew-product semiflow $\Psi$ on $(X\times K,\mathbb{R}^{+})$ associated with \eqref{skew-product-semiflow} over $K$ as follows:
\begin{equation}\label{linearized-skew-product}
\Psi(t,v,(u,y))=(\Phi(t,u,y)v,\Pi^t(u,y)),\,\,t\geq 0,\,(u,y)\in K,\, v\in X.
\end{equation}
For each $(u,y)\in K$, define the Lyapunov exponent
$\lambda(u,y)=\limsup\limits_{t \to \infty }{\frac{\ln||\Phi(t,u,y)||}{t}},$ where
$\norm{\cdot}$ is the operator norm of $\Phi(t,u,y)$.
We call the number $\lambda_{K}={\sup}_{(u,y)\in K}\lambda(u,y)$ the {\it upper
Lyapunov exponent} on $K$.
\begin{definition}
{\rm $K$ is said to be {\it linearly stable}, if $\lambda_{K}\leq 0$.}
\end{definition}

Let $K\subset X\times Y$ (here $X$ is a strongly ordered Banach space, see \cite[Definition II.4.4]{Shen1998}) be a compact invariant set of the strongly monotone skew-product semiflow $\Pi^t$ and $X_1(u,y)=\mathrm{span}\{v(u,y)\}$, $X_2(u,y)$ ($(u,y)\in K$, $v(u,y)\in \mathrm{Int} X^+$ with $\|v(u,y)\|=1$) be the subspaces associated to the continuous separation of $(K,\mathbb{R})$ (see \cite[Definition II.4.6 and Theorem II.4.4]{Shen1998} for the definition and existence of continuous separation of $(K,\mathbb{R}$)). Define
\begin{equation}\label{principal}
  l(t,u,y)=\|\Phi(t,u,y)v(u,y)\|,
\end{equation}
where $(u,y)\in K$, $t\in\mathbb{R}$. By invariance of $X_1(u,y)$, \eqref{principal} generates a linear skew-product flow $\tilde\Psi:\mathbb{R}\times\mathbb{R}\times K\to \mathbb{R}\times K$,
\begin{equation}\label{induced-linear}
  \tilde\Psi(t,z,(u,y))=(l(t,u,y)z,\Pi^t(u,y)).
\end{equation}

\begin{lemma}\label{upper-lya}
Let $K\subset X\times Y$ be a compact invariant set of the strongly monotone skew-product semiflow $\Pi^t$. Then the upper Lyapunov exponent of \eqref{linearized-skew-product} coincides with the upper Lyapunov exponent of \eqref{induced-linear}.
\end{lemma}
\begin{proof}
  See \cite[Proposition II 4.10]{Shen1998}.
\end{proof}

\begin{definition}
  {\rm
  Two flows $(Y,\mathbb{R})$ and $(Z,\mathbb{R})$ are said to be {\it topologically conjugate} if there is a homeomorphism $h:Y\to Z$ such that $h(y\cdot t)=h(y)\cdot t$ for all $y\in Y$ and $t\in \mathbb{R}$.}
\end{definition}
\begin{definition}\label{residual-imbed}
{\rm Let $X_1$ be a metric space. A minimal subset $K\subset X\times Y$ (here we $X$, $Y$ are metric spaces) is said to be {\it residually
embedded} into $X_1\times Y$, if there exist a residual invariant set $Y_*\subset Y$ and a flow $\hat{\Pi}^t$ defined on some subset $\hat{K}\subset X_1\times Y_*$ such that the flow $\Pi^t|_{K\cap p^{-1}(Y_*)}$ is topologically conjugate to $\hat{\Pi}^t$  on $\hat{K}$.
Moreover, when $Y_*=Y$, we call $K$ is {\it embedded} into $X_1\times Y$.
}\end{definition}

\subsection{Almost periodic and almost automorphic functions}

 A function $f\in C(\RR,\RR)$ is {\it almost periodic} if, for any
$\varepsilon>0$, the set
$T(\varepsilon):=\{\tau:\abs{f(t+\tau)-f(t)}<\varepsilon,\,\forall
t\in \RR\}$ is relatively dense in $\RR$. $f$ is {\it almost automorphic} if for every $\{t'_k\}\subset\mathbb{R}$ there is a subsequence $\{t_k\}$
and a function $g:\mathbb{R}\to \mathbb{R}$ such that $f(t+t_k)\to g(t)$ and $g(t-t_k)\to f(t)$ point wise.
Let $D\subseteq \RR^m$ be a nonempty set. A continuous function
$f:\RR\times D\to\RR;(t,w)\mapsto f(t,w),$ is said to be {\it
admissible} if $f(t,w)$ is bounded and uniformly continuous on
$\RR\times K$ for any compact subset $K\subset D$.  A function $f\in
C(\RR\times D,\RR)(D\subset \RR^m)$ is {\it uniformly almost
periodic (automorphic)} {\it in $t$}, if $f$ is both
admissible and almost periodic (automorphic) in $t\in \RR$.

Let $f\in C(\RR\times D,\RR) (D\subset \RR^m)$ be admissible. Then
$H(f)={\rm cl}\{f\cdot\tau:\tau\in \RR\}$ is called the {\it hull of
$f$}, where $f\cdot\tau(t,\cdot)=f(t+\tau,\cdot)$ and the closure is
taken under the compact open topology. Moreover, $H(f)$ is compact
and metrizable under the compact open topology (see \cite{Sell,Shen1998}).
The time translation $g\cdot t$ of $g\in H(f)$ induces a natural
flow on $H(f)$ (cf. \cite{Sell}).

\begin{remark}\label{a-p-to-minial}
{\rm  If $f$ is a uniformly almost periodic function in $t$, then $H(f)$ is always {\it minimal and distal}.
 Moreover, every $g\in H(f)$ is an  uniformly almost periodic function (see, e.g. \cite{Shen1998}). Further, assume that $(X,\mathbb{R})$ is a minimal flow and let $p:X\to H(f)$ be a flow homomorphism such that $X$ is an almost $1$-cover of $H(f)$. Then,
 $$
 \{x\in X: \text{$x\cdot t$ is an almost automorphic function of $t$}\}=\{x\in X: p^{-1}p(x)=\{x\}\}
 $$(see \cite[Section 5.4]{Veech1965} or \cite[Corollary I 2.15]{Shen1998}). }
\end{remark}

\subsection{Zero number properties of linear parabolic equations on $S^1$}

As the zero number is a very important tool in our proofs, we provide in this section the definition of the
zero number and list some related properties.

Given a $C^{1}$-smooth function $u:S^{1}\rightarrow \mathbb{R}^{1}$, the zero number of $u$ is
 defined as
$$z(u(\cdot))={\rm card}\{x\in S^{1}:u(x)=0\}.$$
 We now list some properties of the zero number (see, e.g. \cite{2038390,H.MATANO:1982} or \cite[Lemma 2.2]{SF1}).
\begin{lemma}\label{zero-number}
Consider the linear system
\begin{equation}\label{linear-equation}
\begin{cases}
\varphi_{t}=\varphi_{xx}+b \varphi_{x}+c\varphi,\quad x\in S^1,\\
\varphi_{0}=\varphi(0,\cdot)\in H^{1}(S^{1}),
\end{cases}
\end{equation}
where the coefficients $b,c$ are allowed to depend on $t$ and $x$ such that
$b,b_{t},b_{x},c\in L_{loc}^\infty$. Let $\varphi(t,x)$ be a nontrivial solution of \eqref{linear-equation}.
Then the following properties holds.\par
{\rm (a)} $z(\varphi(t,\cdot))<\infty,\forall t>0$ and is non-increasing in t.\par
{\rm (b)}  $z(\varphi(t,\cdot))$ can drop only at $t_{0}$ such that $\varphi(t_{0},\cdot)$ has a
multiple zero in $S^{1}$.\par
{\rm (c)}  $z(\varphi(t,\cdot))$ can drop only finite many times,and there exists a $T>0$
such that $\varphi(t,\cdot)$ has only simple zeros in $S^{1}$ as $t\geq T$(hence
$z(\varphi(t,\cdot))=constant$ as $t\geq T$).
\end{lemma}

\begin{lemma}\label{sigma-function}
For any $g\in H(f)$, Let $\varphi(t,\cdot;u,g)$ and $\varphi(t,\cdot;\hat{u},g)$ be
distinct solutions of {\rm (\ref{equation-lim1})} on
$ \mathbb{R}^+$. Then

\begin{description} \item[{\rm (a)}]
$z(\varphi(t,\cdot;u,g)-\varphi(t,\cdot;\hat{u},g))<\infty$ for $t>0$ and is non-increasing in t;

\item[{\rm (b)}] $z(\varphi(t,\cdot;u,g))-\varphi(t,\cdot;\hat{u},g)))$
strictly decreases at $t_0$ such that the function $\varphi(t_0,\cdot;u,g))-\varphi(t_0,\cdot;\hat{u},g)$ has a multiple
zero in $S^1$;

\item[{\rm (c)}] $z(\varphi(t,\cdot;u,g))-\varphi(t,\cdot;\hat{u},g)))$ can drop only finite many times, and there exists a $T>0$ such that  $$z(\varphi(t,\cdot;u,g))-\varphi(t,\cdot;\hat{u},g)))\equiv
\textnormal{constant}$$ for all $t\ge T$.
\end{description}
\end{lemma}
\begin{proof}
Denote $u(t,x)=\varphi(t,x;u,g)$ and $\hat{u}(t,x)=\varphi(t,x;\hat{u},g)$. Let $v(t,x)=u(t,x)-\hat{u}(t,x)$ be a nontrivial solution of the linear parabolic equation \eqref{linear-equation}, where
$$b(t,x)=\int_0^1\dfrac{\partial g}{\partial
p}(t,x,u(t,x),su_x(t,x)+(1-s)\hat{u}_x(t,x))ds,$$
$$c(t,x)=\int_0^1\dfrac{\partial g}{\partial
u}(t,x,su(t,x)+(1-s)\hat{u}(t,x),\hat{u}_x(t,x))ds.$$
 Then the results directly follows from Lemma \ref{zero-number}.
\end{proof}

 \begin{lemma}\label{sequence-limit}
 Fix $g,\ g_{0}\in H(f)$. Let $(u^{i},g)\in p^{-1}(g),(u_{0}^{i},g_{0})\in p^{-1}(g_{0}) (i=1,\ 2,\ u^{1}\neq u^{2},\ u_{0}^{1}\neq u_{0}^{2})$ be such that $\Pi^{t}(u^{i},g)$ is defined on $\mathbb{R}^{+}$ (resp. $\mathbb{R}^{-}$) and $\Pi^{t}(u_{0}^{i},g_{0})$ is defined on $\mathbb{R}$. If there exists a sequence $t_{n}\rightarrow +\infty$ (resp. $s_{n}\rightarrow -\infty$) as $n\rightarrow \infty$, such that $\Pi^{t_{n}}(u^{i},g)\rightarrow (u_{0}^{i},g_{0})$ (resp. $\Pi^{s_{n}}(u^{i},g)\rightarrow (u_{0}^{i},g_{0})$) as $n\rightarrow \infty (i=1,2)$, then
$$z(\varphi(t,\cdot;u_{0}^{1},g_{0})-\varphi(t,\cdot;u_{0}^{2},g_{0}))\equiv \textnormal{constant},$$
for all $t\in \mathbb{R}$.
\end{lemma}
\begin{proof}
We only prove the case of $t_n\to +\infty$. The case of $s_n\to -\infty$ is similar.
By virtue of Lemma \ref{sigma-function}(c),  there exist $T>0$ and $N_{1}\in\mathbb{N}$ such that
$z(\varphi(t,\cdot;u^{1},g)-\varphi(t,\cdot;u^{2},g))= N_{1},$
for all $t\geq T$. Fix $t_{0}\in \mathbb{R}$ so that $\varphi(t_0,\cdot;u_{0}^{1},g_{0})-\varphi(t_0,\cdot;u_{0}^{2},g_{0})$ only has simple zeros. Note that with $\Pi^{t_{n}}(u^{i},g)\rightarrow (u_{0}^{i},g_{0})(i=1,2)$ and the continuity of $z(\cdot)$, one has $z(\varphi(t_{n}+t_{0},\cdot;u^{1},g)-\varphi(t_{n}+t_{0},\cdot;u^{2},g))=z(\varphi(t_{0},\cdot;u_{0}^{1},g_{0})-\varphi(t_{0},\cdot;u_{0}^{2},g_{0}))$ for all $n$ sufficiently large.
Consequently, $z(\varphi(t_{0},\cdot;u_{0}^{1},g_{0})-\varphi(t_{0},\cdot;u_{0}^{2},g_{0}))=N_{1}$.
By Lemma \ref{sigma-function}(b)-(c) and the arbitrariness of $t_{0}$, we have $z(\varphi(t,\cdot;u_{0}^{1},g_{0})-\varphi(t,\cdot;u_{0}^{2},g_{0}))=N_{1}$ for all $t\in \mathbb{R}$.
\end{proof}

\subsection{Floquet bundles and invariant subspaces for linear  parabolic equations on $S^1$}

Consider the following linear parabolic equation:
\begin{equation}\label{linear-equation2}
\psi_t=\psi_{xx}+a(x,\omega\cdot t)\psi_x+b(x,\omega\cdot t)\psi,\,\,t>0,\,x\in S^{1}=\mathbb{R}/2\pi \mathbb{Z},
\end{equation}
where $\omega\in \Omega$, $\omega\cdot t$ is a flow on a compact  metric space $\Omega$; and $a^\omega(t,x):=a(x,\omega\cdot t)$,  $b^\omega(t,x):=b(x,\omega\cdot t)$ are continuously differentiable in $(t,x)$; and $a^\omega,a^\omega_t,a^\omega_x,b^\omega:\mathbb{R}\times S^1\to \mathbb{R}$ are bounded functions uniformly for $\omega\in \Omega$.

For any $w\in L^2(S^1)$, let $\psi(t,x;w, \omega)$ be the solution  of \eqref{linear-equation2} with $\psi(0,x;w,\omega)=w(x),x\in S^1$; and we write
 $$\Psi(t,\omega):L^2(S^1)\rightarrow L^2(S^1);w(\cdot)\mapsto \psi(t,\cdot;w, \omega)$$ as the evolution operator generated by  \eqref{linear-equation2}. It is known that, for each $t>0$ and $\omega\in \Omega$, $\Psi(t,\omega)$ is injective (see \cite[Proposition 2.5.8]{Mierczynski} or \cite[Chapter 6]{Friedman}); and by standard a priori estimates (see, e.g. \cite[Section 2.5]{Mierczynski}), $\Psi(t,\omega)$ is a compact (completely continuous) operator. Moreover, we have the following Lemma:

\begin{lemma}\label{floquet-bundle}
\begin{itemize}
\item [\rm{(i)}] There is a sequence $\{w_n,\tilde{w}_n\}_{n=1}^\infty\cup \{w_0\}$,  $w_0,w_n,\tilde {w}_n:S^1\times \Omega\rightarrow \mathbb{R}$, such that $w_0(\cdot,\omega),w_n(\cdot,\omega),\tilde {w}_n(\cdot,\omega)\in C^{1,\gamma}(S^1)$ for any $\gamma$ with $0\leq \gamma <1$, and $\|w_0(\cdot,\omega)\|_{L^2(S^1)}=\|w_n(\cdot,\omega)\|_{L^2(S^1)}=\|\tilde{w}_n(\cdot,\omega)\|_{L^2(S^1)}=1$ for any $\omega \in \Omega$. $\{w_n(\cdot,\omega),\tilde{w}_n(\cdot,\omega)\}_{n=1}^\infty\cup \{w_0(\cdot,\omega)\}$ forms a Floquet basis of $L^2(S^1)$, $z(w_0(\cdot,\omega))=0$ and $z(w_n(\cdot,\omega))=2n$ ($z(\tilde{w}_n(\cdot,\omega))=2n$) for all $\omega\in \Omega$.
\item [\rm{(ii)}]
     Let $W_0(\omega)={\rm span} \{w_0(\cdot,\omega)\}$ and $W_n(\omega)={\rm span}\{w_n(\cdot,\omega),\tilde{w}_n(\cdot,\omega)\}$, $n=1,2,\cdots$.
    Then $\Psi(t,\omega)W_{n}(\omega)=W_{n}(\omega\cdot t)$ for $t\geq 0$; and hence, $\Psi(t,\omega)|_{W_{n}(\omega)}$ is a linear isomorphism. Moreover,
     \begin{eqnarray*}
     {\bigoplus}_{i=n_1}^{n_2}W_i(\omega)&=\{w\in L^2(S^1):\psi(t,\cdot;w,\omega)\ \text{is exponentially bounded in } L^2(S^1),\\ &\ and \ 2n_1\leq z(\psi(t,\cdot;w,\omega))\leq 2n_2\ \text{for all }t\in \mathbb{R}\} \cup\{0\}
     \end{eqnarray*}
      for any $n_1$, $n_2$ with $0\leq n_1\leq n_2<\infty$.
\item[\rm{(iii)}] Let $w=c_0^0w_0(x,\omega)+{\Sigma}_{n=1}^{\infty}(c_n^0w_n(x,\omega)+\tilde{c}_n^0\tilde{w}_n(x,\omega))$ in $L^2(S^1)$. Then
    \begin{equation}\label{Floquet-rep}
    \psi(t,x,w,\omega)=c_0(t)w_0(x,\omega\cdot t)+{\Sigma}_{n=1}^{\infty}(c_n(t)w_n(x,\omega\cdot t)+\tilde{c}_n(t)\tilde{w}_n(x,\omega\cdot t)),
    \end{equation}
    where $c_0(t)$, $c_n(t)$, $\tilde{c}_n(t)$ are continuous functions with $c_0(0)=c^0_0$, $c_n(0)=c^0_n$, $\tilde{c}_n(0)=\tilde{c}^0_n$.
\item[\rm{(iv)}] Define $\Lambda(\cdot):\Omega\rightarrow L(L^2(S^1),l^2)$ by $\Lambda(\omega)w_0=\{c_0^0\}\cup\{c_n^0,\tilde{c}_n^0 \}_{n=1}^\infty$, where $w_0=c_0^0w_0(x,\omega)+{\Sigma}_{n=1}^{\infty}(c_n^0w_n(x,\omega)+\tilde{c}_n^0\tilde{w}_n(x,\omega))$. Then $\Lambda(\omega\cdot t)\psi(t,x,w_0,\omega)=\{c_0(t)\}\cup\{c_n(t),\tilde{c}_n(t) \}_{n=1}^\infty$. Moreover, $\Lambda$ is continuous, $\Lambda(\omega)$ is an isomorphism for each $\omega\in \Omega$, and there are positive constants $K_1$, $K_2$ which are independent of $\omega$ such that
    \begin{equation*}
      \|\Lambda(\omega)\|\leq K_1\ \ and \quad \|\Lambda^{-1}(\omega)\|\leq K_2.
    \end{equation*}
\item[\rm{(v)}] {\rm (Exponential Separation)} Let $Z_k^+(\omega)={\bigoplus}_{i=0}^{k}W_i(\omega)$ and $Z_k^-(\omega)=\mathrm{cl}\{\cup_{l\geq k+1}\bigoplus_{i=k+1}^{l}W_i(\omega)\}$ in $L^2(S^1)$. Then, $\Psi(t,\omega)Z_k^+(\omega)=Z_k^+(\omega)$ and  $\Psi(t,\omega)Z_k^-(\omega)\subset Z_k^-(\omega)$
    for $t\ge 0$, with $\Psi(t,\omega)Z_k^-(\omega)$ being dense in $Z_k^-(\omega)$ in $L^2$-norm. Moreover, there are constants $K>0,\mu>0$ which are independent of $\omega$, such that
    if $\psi(t,\cdot;v_0,\omega)\in Z_k^+(\omega\cdot t)$ and $\psi(t,\cdot;w_0,\omega)\in Z_k^-(\omega\cdot t)$ are nontrivial solutions for \eqref{linear-equation2}, then

   \begin{equation}\label{E:Es-proerty-L2}
    \dfrac{\norm{\psi(t,\cdot;w_0,\omega)}_{L^2}}{\norm{\psi(t,\cdot;v_0,\omega)}_{L^2}}
    \le Ke^{-\mu (t-s)}\dfrac{\norm{\psi(s,\cdot;w_0,\omega)}_{L^2}}{\norm{\psi(s,\cdot;v_0,\omega)}_{L^2}},
    \end{equation}
    for all $t\ge s,w_0\in Z_k^-(\omega),v_0\in Z_k^+(\omega)\setminus\{0\}$.
\end{itemize}
\end{lemma}
\begin{proof}
See \cite [Section 9]{Chow1995}.
\end{proof}

Throughout this subsection, $X$ is as in the introduction, that is,  $X$ is the fractional power space associated with the operator $u\rightarrow
-u_{xx}:H^{2}(S^{1})\rightarrow L^{2}(S^{1})$ such that the embedding
relations  $H^{2}(S^{1})\hookrightarrow X\hookrightarrow C^{1}(S^{1})\hookrightarrow L^2(S^1)$ are satisfied (these embeddings are also compact). Hereafter, $\|\cdot\|$ always denotes the norm in $X$.

Note that $\Psi(t,\omega)|_{X}\in \mathcal{L}(X)$. Moreover, for any $t>0$ and $\omega\in \Omega$, $\Psi(t,\omega)|_X$ is injective (see \cite[Proposition 2.5.8]{Mierczynski}), and by standard a priori estimates, $\Psi(t,\omega)|_X$ is a compact operator from $X$ to $X$. In the following, if no confusion occurs, we  may write $\Psi(t,\omega)|_X$ as $\Psi(t,\omega)$.

We put $\Pi^t:X\times \Omega\to X\times \Omega$ as
$$\Pi^t(v,\omega)=(\Psi(t,\omega)v,\omega\cdot t)\,\text{ for } t\ge 0.$$ Then the family of  maps $\{\Pi^t\}_{t\ge 0}$ so defined is indeed a linear skew-product semiflow on $X\times\Omega$ (see, e.g. \cite[p.57]{Mierczynski}).

Now we recall the conception of exponential dichotomy (ED) and Sacker-Sell spectrum.
Let $\lambda\in \mathbb{R}$ and define $\Pi^t_{\lambda}: X\times\Omega\rightarrow X\times\Omega$ by
\begin{equation}\label{exponential-dicho}
  \Pi^t_{\lambda}(v,\omega)=(\Psi_{\lambda}(t,\omega)v,\omega\cdot t),
\end{equation}
where $\Psi_{\lambda}(t,\omega)=e^{-\lambda t}\Psi(t,\omega)$. It is easy to verify that $\Pi^t_{\lambda}$ is also a linear skew-product semiflow on $X\times\Omega$. We say $\Psi_{\lambda}$ admits an {\it exponential dichotomy over} $\Omega$ if there exist $K>0$, $\alpha>0$ and continuous projections $P(\omega):X\rightarrow X$ such that for all $\omega \in \Omega$, $\Psi_{\lambda}(t,\omega)|_{R(P(\omega))}:R(P(\omega))\rightarrow R(P(\omega\cdot t))$ is an isomorphism satisfying $\Psi_{\lambda}(t,\omega)P(\omega)=P(\omega \cdot t)\Psi_{\lambda}(t,\omega)$, $t\in \mathbb{R}^+$; moreover,
\begin{equation*}
\begin{split}
\|\Psi_{\lambda}(t,\omega)(I-P(\omega))\|\leq Ke^{-\alpha t},\quad t\geq 0,\\
\|\Psi_{\lambda}(t,\omega)P(\omega)\|\leq Ke^{\alpha t},\quad t\leq 0.
\end{split}
\end{equation*}
Here $R(P(\omega))$ is the range of $P(\omega)$. We call
$$
\sigma(\Omega)=\{\lambda\in \mathbb{R}:\Pi^t_{\lambda}\ \text{has no exponential dichotomy over }\Omega\}
$$
the {\it Sacker-Sell spectrum} of \eqref{linear-equation2}. If $\Omega$ is  {\it compact and connected}, then the Sacker-Sell spectrum $\sigma(\Omega)=\bigcup_{k=0}^\infty I_k$, where $I_k=[a_k,b_k]$ and $\{I_k\}$ is ordered from right to left, that is, $\cdots<a_k\leq b_k<a_{k-1}\leq b_{k-1}<\cdots<a_0\leq b_0$ (see \cite{Chow1994,Sacker1978,Sacker1991}).\par
For any given $0\leq n_1\leq n_2\leq\infty$, if $n_2\neq \infty$, let
\begin{equation}\label{twoside-estimate}
\begin{split}
V^{n_1,n_2}(\omega)=\{v\in X:&\|\Psi(t,\omega)v\|=o(e^{a^-t})\ \text{as}\ t\rightarrow -\infty\\ & \|\Psi(t,\omega)v\|=o(e^{b^+t})\ \text{as}\ t\rightarrow \infty\}
\end{split}
\end{equation}
where $a^-$, $b^+$ are such that $\lambda_1<a^-<a_{n_2}\leq b_{n_1}<b^+<\lambda_2$ for any $\lambda_1\in \cup_{k=n_2+1}^{\infty}I_k$ and $\lambda_2\in\cup_{k=0}^{n_1-1}I_k$. If $n_2=\infty$, let
\begin{equation}\label{infty-estimate}
V^{n_1,\infty}(\omega)=\{v\in X :\|\Psi(t,\omega)v\|=o(e^{b^+t})\text{ as }t\to\infty\}
\end{equation}
where $b^+$ is such that $b_{n_1}<b^+<\lambda$ for any $\lambda\in \cup_{k=0}^{n_1-1}I_k$.
\par

\begin{remark}\label{invariant-space}
{\rm (i) Since the solution operators $\Psi(t,\omega)$ are compact (completely continuous) for $t>0$, it follows from the similar deduction in \cite[Theroem B(4)]{Sacker1991} (or see \cite[Lemma 3.1(1)]{Chow1994}) that $\dim V^{n_1,n_2}(\omega)<\infty$ for any fixed $0\leq n_1\leq n_2<\infty$.

(ii) $V^{n_1,n_2}(\omega)$ is {\it invariant} in the sense that, for $t\geq 0$, $\Psi(t,\omega)V^{n_1,n_2}(\omega)=V^{n_1,n_2}(\omega\cdot t)$ when $n_2<\infty$, while  $V^{n_1,\infty}(\omega)$ satisfies $\Psi(t,\omega)V^{n_1,\infty}(\omega)\subset V^{n_1,\infty}(\omega\cdot t)$ with $\Psi(t,\omega)V^{n_1,\infty}(\omega)$ being dense in $V^{n_1,\infty}(\omega\cdot t)$ in the $X$-norm. In fact, due to the adjoint operator $\Psi (t,\omega)^*$ of the adjoint equation is injective (see, e.g.  \cite[p.62-63]{Mierczynski}), such density can be directly obtained by the Hahn-Banach Theorem (cf. \cite[Corollary 4.12(b)]{Rudin}).}
\end{remark}

Suppose that $0\in \sigma(\Omega)$ and $n_0$ is such that $0\in I_{n_0}\subset\sigma(\Omega)$. Then $V^s(\omega)=V^{n_0+1,\infty}(\omega)$, $V^{cs}(\omega)=V^{n_0,\infty}(\omega)$, $V^{c}(\omega)=V^{n_0,n_0}(\omega)$, $V^{cu}(\omega)=V^{0,n_0}(\omega)$, and $V^u(\omega)=V^{0,n_0-1}(\omega)$ (If $n_0=0$, we set $V^u(\omega)=\{0\}$) are referred to as {\it stable, center stable, center, center unstable}, and {\it unstable subspaces} of \eqref{linear-equation2} at $\omega\in \Omega$, respectively.

Suppose that $0\not\in \sigma(\Omega)$ and $n_0$ is such that $I_{n_0}\subset (0,\infty)$ and $I_{n_0+1}\subset (-\infty,0)$.
$V^s(\omega)=V^{n_0+1,\infty}(\omega)$ and $V^u(\omega)=V^{0,n_0}(\omega)$ are referred to as {\it stable} and {\it unstable subspaces} of \eqref{linear-equation2} at $\omega\in \Omega$, respectively.

Fix any $\omega\in \Omega$, we consider \eqref{linear-equation2}.
By means of the transformation introduced in \cite[p.247-248]{Chow1995}, we let
\begin{equation}\label{E:transform-congr-1}
\hat\psi(t,x):=r^{\omega}(t,x+c^{\omega}(t))\cdot\psi(t,x+c^{\omega}(t)),
 \end{equation}
 where the function $c^{\omega}(t)$ satisfies
$$\dot c^{\omega}(t)=-a_0^{\omega}(t),\,\,\text{ with }\,\, a_0^{\omega}(t)=\frac{1}{2\pi}\int_0^{2\pi}a^{\omega}(t,z)dz,$$ and
$$
 r^{\omega}(t,x)=\exp\Big(\frac{1}{2}\int_0^ x (a^{\omega}(t,z)-a_0^{\omega}(t))dz\Big).
$$
Then the equation \eqref{linear-equation2} can be changed into
 \begin{equation}\label{no-gradient}
   \hat \psi_t=\hat \psi_{xx}+\hat b^{\omega}(t,x)\hat \psi,
 \end{equation}
where
$$
\hat b^{\omega}(t,x-c^{\omega}(t))=b^{\omega}(t,x)+\frac{1}{2}\int_0^{x}[a^{\omega}(t,z)-a_0^{\omega}(t)]_tdz
-\frac{1}{4}(a^{\omega}(t,x)^2-a_0^{\omega}(t)^2)-\frac{1}{2}a^{\omega}_x(t,x).
$$
Recall that the functions $a^\omega,a^\omega_t,a^\omega_x,b^\omega$ are bounded functions uniformly for $\omega\in \Omega$. Then $\hat b^{\omega}$ is a bounded function uniformly for $\omega\in \Omega$.

Hereafter, we write $\hat\sigma(\Omega)=\bigcup_{k=0}^\infty \hat I_k$ as the associated Sacker-Sell spectrum of \eqref{no-gradient} on $\Omega$, with $\hat{I}_k=[\hat{a}_k,\hat{b}_k]$ satisfying $\cdots<\hat{a}_k\leq \hat{b}_k<\hat{a}_{k-1}\leq \hat{b}_{k-1}<\cdots<\hat{a}_0\leq \hat{b}_0$. Let also $\{\hat w_0(\cdot,\omega)\}\cup\{\hat w_n(\cdot,\omega),\tilde{\hat w}_n(\cdot,\omega)\}_{n=1}^\infty$ be the Floquet basis of \eqref{no-gradient} as stated in Lemma \ref{floquet-bundle}. We further write $\hat W_0(\omega)=\mathrm{span}\{\hat w_0(\cdot,\omega)\}$, $\hat W_n(\omega)=\mathrm{span}\{\hat w_n(\cdot,\omega),\tilde{\hat w}_n(\cdot,\omega)\}$ for $n=1,2,\cdots.$
\begin{remark}\label{no-chang-spetr}
{\rm
  Since $0<\inf_{(t,x)\in\mathbb{R}\times S^1, \omega\in \Omega} |r^{\omega}(t,x)|\le \sup _{(t,x)\in\mathbb{R}\times S^1, \omega\in \Omega} |r^{\omega}(t,x)|<\infty$, it is not difficult to see that the Sacker-Sell spectrum $\sigma(\Omega)$ of \eqref{linear-equation2} is equal to the Sacker-Sell spectrum $\hat\sigma(\Omega)$ of \eqref{no-gradient}, i.e., $\hat\sigma(\Omega)=\sigma(\Omega)$. Moreover, we have  $W_n(\omega)=\{v:v(x)=(r^{\omega}(0,x))^{-1}\hat v(x-c^{\omega}(0)),\ \hat v\in \hat W_n(\omega)\}$ for $n=0,1,2,\cdots$; and  $V^{n_1,n_2}(\omega)=\{v:v(x)=(r^{\omega}(0,x))^{-1}\hat v(x-c^{\omega}(0)), \ \hat v\in \hat V^{n_1,n_2}(\omega)\}$, where $\hat V^{n_1,n_2}(\omega)$ are associated invariant subspaces of \eqref{no-gradient} at $\omega\in\Omega$.
  }
\end{remark}

Before ending this section, we present a lemma which will be useful in the forthcoming sections.

\begin{lemma}\label{zero-inva}
For given $0\leq n_1\leq n_2\leq\infty$ (when $n_2=\infty$, $n_1<\infty$ is needed), we have
\begin{equation*}
N_1\leq z(v(\cdot))\leq N_2,\,\,\text{ for any }v\in V^{n_1,n_2}(\omega)\setminus\{0\},
\end{equation*}
where
\begin{equation*}
N_1=\left\{
\begin{split}
 &{\rm dim}V^{0,n_1-1}(\omega),\,\quad\,\,\,\text{ if }{\rm dim}V^{0,n_1-1}(\omega)\text{ is even;}\\
 &{\rm dim}V^{0,n_1-1}(\omega)+1,\,\text{ if }{\rm dim}V^{0,n_1-1}(\omega)\text{ is odd,}
\end{split}\right.
\end{equation*} and
\begin{equation*}
N_2=\left\{
\begin{split}
 &{\rm dim}V^{0,n_2}(\omega),\,\quad\,\,\,\text{ if }{\rm dim}V^{0,n_2}(\omega)\text{ is even;}\\
 &{\rm dim}V^{0,n_2}(\omega)-1,\,\text{ if }{\rm dim}V^{0,n_2}(\omega)\text{ is odd.}
\end{split}\right.
\end{equation*}
Here, we define $V^{0,-1}(\omega)=\{0\}.$
\end{lemma}

In order to prove Lemma \ref{zero-inva}, we first need the following two lemmas (Lemmas \ref{floquet-invar} and
\ref{floquet-zeron}). Based on these two lemmas, we will then give Lemma \ref{zero-inva0}, which can be viewed as a lemma parallel to Lemma \ref{zero-inva}. Together with Remark \ref{no-chang-spetr}, one can easily see that Lemma \ref{zero-inva0} implies Lemma \ref{zero-inva} immediately.

\begin{lemma}\label{floquet-invar}
      Let $\hat V^{0,n_1}(\omega)$ be the associated invariant subspace with $\bigcup_{k=0}^{n_1}\hat I_k$.

    {\rm (i)} If $\hat W_{n}(\omega)\cap \hat V^{n_1,\infty}(\omega)\ne \{0\}$ for some $n,n_1\ge 0$, then $\hat W_k(\omega)\subset \hat V^{n_1,\infty}(\omega)$ whenever $k>n$;

    {\rm (ii)} If $\hat W_{n}(\omega)\cap \hat V^{0,n_1}(\omega)\ne \{0\}$ for some $n,n_1\ge 0$, then $\hat W_k(\omega)\subset \hat V^{0,n_1}(\omega)$ whenever $k<n$.
\end{lemma}

\begin{proof} We only give the proof of (i), because that of (ii) is similar.
  Choose any $w\in (\hat W_{n}(\omega)\cap \hat V^{n_1,\infty}(\omega))\setminus\{0\} $ and let $\hat \psi(t,\cdot,w,\omega)$ be a solution of \eqref{no-gradient} with $\hat\psi(0,\cdot,w,\omega)=w$. For any $v\in \hat W_k(\omega)\setminus \{0\}$ (with $k>n$), let $\hat \psi(t,\cdot,v,\omega)$ be the solution of \eqref{no-gradient} with $\hat\psi(0,\cdot,v,\omega)=v$. Then by Lemma \ref{floquet-bundle}(v), there exist positive quantities $\hat K$ and $\nu$ such that \begin{equation}\label{L:ex-exp11}
   \|\hat \psi(t,\cdot;v,\omega)\|_{L^2}\leq \hat K e^{-\nu t}\|\hat \psi(t,\cdot;w,\omega)\|_{L^2}\frac{\|v\|_{L^2}}{\|w\|_{L^2}}\,\text{ for any } t\ge 0.
 \end{equation}

 Recall that $w\in\hat V^{n_1,\infty}(\omega)$. Then, for any $b^+\in (\hat{b}_{n_1}, \hat{a}_{n_1-1})$, one has $e^{-b^+t}\|\hat \psi(t,\cdot;w,\omega)\|\to 0$ as $t\to \infty$. Hence,
  $\|e^{-b^+t}\hat \psi(t,\cdot;w,\omega)\|_{L^2}\to 0$ as $t\to \infty$. By virtue of \eqref{L:ex-exp11}, we have
\begin{equation}\label{asym-zeo}
\|e^{-b^+t}\hat \psi(t,\cdot;v,\omega)\|_{L^2}\to 0.
\end{equation}
Hereafter, we write $\tilde\psi(t,x)=e^{-b^+t}\hat \psi(t,x;v,\omega)$ for brevity. Clearly, $\tilde\psi(t,x)$ satisfies the equation
\begin{equation}\label{E:renormalized-eqns}
  \tilde\psi_t=\tilde\psi_{xx}+\tilde b^{\omega}(t,x)\tilde\psi,
\end{equation}
 with $\tilde b^{\omega}(t,x)=\hat b^{\omega}(t,x)-b^+$. By the standard regularity estimate for \eqref{E:renormalized-eqns} (see \cite[Lemma 3.2 and Theorem 3.3]{Chow1995}, and the statement in the last paragraph on \cite[p.296]{Chow1995}), we have
  \begin{equation}\label{E:regularity-es}
  \|\tilde\psi(t,\cdot)\|\le [\sup_{n\geq 1} n^{\alpha}e^{-n^2}+1]\cdot\|\tilde\psi(\tau,\cdot)\|_{L^2}
  + K(\alpha)(1+t-\tau)\sup_{s\in[\tau,t]} \|\tilde\psi(s,\cdot)\|_{L^2},
\end{equation}
whenever $t\ge \tau+1\ge \tau\ge 0$. Here $K(\alpha)$ is a constant only depending on $\alpha$.

Then for any $b\in (b^+,\hat{a}_{n_1-1})$, the estimates \eqref{E:regularity-es} implies that
\begin{equation}
e^{-(b-b^+)t}\norm{\tilde\psi(t,\cdot)}\le e^{-(b-b^+)t}K_1(\alpha)(1+t-\tau)\cdot\sup_{s\in[\tau,t]} \|\tilde\psi(s,\cdot)\|_{L^2},
\end{equation}
whenever $t\ge \tau+1\ge \tau\ge 0$. Here $K_1(\alpha)$ is a constant only depending on $\alpha$. Together with
\eqref{asym-zeo}, this implies that $e^{-(b-b^+)t}\norm{\tilde\psi(t,\cdot)}\to 0$ as $t\to \infty$, that is, $$e^{-bt}\norm{\hat \psi(t,\cdot;v,\omega)}\to 0,\quad\text{ as }n\to \infty.$$ By the arbitrariness of $b$ and $b^+$, we have obtained that $v\in \hat V^{n_1,\infty}(\omega)$, which completes the proof.
\end{proof}

\begin{lemma}\label{floquet-zeron}
 Let $\hat V^{0,n_1}(\omega)$ be in Lemma \ref{floquet-invar} with $\dim \hat V^{0,n_1}(\omega)=N_1>0$. Then:
 \begin{itemize}
   \item [\rm{(i)}] If $N_1$ is odd, let $N'_1=\frac{N_1-1}{2}$, then
  \begin{equation}\label{E:V-0-n1-charac}
    \hat V^{0,n_1}(\omega)=\bigoplus_{k=0}^{N'_1}\hat W_k(\omega)
  \end{equation}
  and
 \begin{equation}\label{E:V-n1+1-charac}
  \hat V^{n_1+1,\infty}(\omega)=\mathrm{cl}\{\cup_{m\geq  N'_1+1}\bigoplus_{k=N'_1+1}^m\hat W_k(\omega)\},
  \end{equation}
 for any $\omega\in \Omega$. Here the closure is taken in the $X$-norm.
  \item [\rm{(ii)}] If $N_1$ is even, let $N'_1=\frac{N_1}{2}$, then  for any $\omega\in \Omega$,
\begin{equation}\label{E:V-0-n1-charac-even}
  \hat V^{0,n_1}(\omega)=(\bigoplus_{k=0}^{ N^\prime_1-1}\hat W_k(\omega))\oplus V'(\omega)
\end{equation}
  and
\begin{equation}\label{E:V-n1+1-charac-even}
  \hat V^{n_1+1,\infty}(\omega)=\mathrm{cl}\{\cup_{m\geq N'_1+1}\bigoplus_{k=N'_1+1}^m\hat W_k(\omega)\oplus V''(\omega)\},
  \end{equation}
  where $V'(\omega)$, $V''(\omega)$ are 1-dimensional subspaces satisfying $V'(\omega)\bigoplus V''(\omega)=\hat W_{N'_1}(\omega)$. Here the closure is taken in the $X$-norm.
 \end{itemize}
 \end{lemma}

 \begin{proof}
Take some $w\in \hat V^{0,n_1}(\omega)\setminus\{0\}$ and let $w=c^0_0\hat w_0(\cdot,\omega)+{\sum}_{n=1}^{\infty}(c^0_n\hat w_n(\cdot,\omega)+\tilde{c}^0_n\tilde{\hat w}_n(\cdot,\omega))$ in $L^2(S^1)$ for some $\{c^0_0,c^0_n,\tilde c^0_n\}\in l^2$. Since $w\neq 0$, there exists some $n_0\in \mathbb{N}\cup \{0\}$ such that $c^0_{n_0}\neq 0$ or $\tilde c^0_{n_0}\neq 0$. Without loss of generality, one may assume that $c^0_{n_0}\neq 0$. {\it We claim that} $\hat w_{n_0}(\cdot,\omega)\in \hat V^{0,n_1}(\omega)$. In fact, by Lemma \ref{floquet-bundle}(iii), $\hat\psi(t,x;w,\omega)=c_0(t)\hat w_0(x,\omega\cdot t)+\sum_{n=1}^{\infty}(c_n(t)\hat w_n(x,\omega\cdot t)+\tilde{c}_n(t)\tilde{\hat w}_n(x,\omega\cdot t))$; and moreover, Lemma \ref{floquet-bundle}(iv) implies that
\begin{equation}\label{coefficient-control}
  |c_{n_0}(t)|\leq\{c_0(t)^2+\sum_{n=1}^{\infty}(c_n(t)^2+\tilde c_n(t)^2)\}^{\frac{1}{2}}
  \leq \hat K_1\|\hat\psi(t,\cdot;w,\omega)\|_{L^2} \leq \hat K_1\|\hat\psi(t,\cdot;w,\omega)\|,
\end{equation}
where $\hat K_1$ is a constant. Since $w\in \hat V^{0,n_1}(\omega)$, one has  $\|e^{-a^-t}\hat \psi(t,\cdot;w,\omega)\|_{X^{\alpha}}\to 0$ ($t\to -\infty$), for any $a^-<\hat{a}_{n_1}:=\inf\{ {\cup_{k=0}^{n_1}}\hat I_k\}$. It then follows from \eqref{coefficient-control} that $e^{-a^-t}\abs{c_{n_0}(t)}\to 0$ as $t\to -\infty$; and hence,  $\|e^{-a^-t}\hat\psi(t,\cdot;\hat w_{n_0}(\cdot,\omega),\omega)\|_{L^2}\to 0$ as $t\to-\infty$.
By the same the argument {\bl as} in Lemma \ref{floquet-invar}, one has $\|e^{-at}\hat\psi(t,\cdot;\hat w_{n_0}(\cdot,\omega),\omega)\|_{X^{\alpha}}\to 0\, (t\to -\infty)$, for any $a<\hat{a}_{n_1}$. Thus, we have proved the claim that $\hat w_{n_0}(\cdot,\omega)\in \hat V^{0,n_1}(\omega)$. Moreover, Lemma \ref{floquet-invar}(ii) directly implies that
\begin{equation}\label{E:V-01-including}
\bigoplus_{i=0}^{n_0-1}\hat W_i(\omega)\oplus \mathrm{span}\{\hat w_{n_0}(\cdot,\omega)\}\subset \hat V^{0,n_1}(\omega).
 \end{equation}
 Here, we set $\bigoplus_{i=0}^{-1}\hat W_i(\omega)=\{0\}$ whenever $n_0=0$.

 In the following, we will consider the cases that $N_1$ is odd and even, respectively.

{\it Case {\rm (i)}: $N_1$ is odd.} If $N_1=1$, then it is clear from \eqref{E:V-01-including} that $n_0=0$ and
$\hat V^{0,n_1}(\omega)=\mathrm{span}\{\hat w_{0}(\cdot,\omega)\}.$ Thus, we are done. So, we only consider $N_1>1$ here.

 Noticing in \eqref{E:V-01-including} that ${\rm dim}\{\bigoplus_{i=0}^{n_0-1}\hat W_i(\omega)\oplus \mathrm{span}\{\hat w_{n_0}(\cdot,\omega)\}\}=2n_0,$ we have $2n_0<N_1$.
So, one may choose some $w\in \hat V^{0,n_1}(\omega)\setminus(\bigoplus_{i=0}^{n_0-1}\hat W_i(\omega)\oplus \mathrm{span}\{\hat w_{n_0}(\cdot,\omega)\})$ with (a) $\tilde c_{n_0}\neq 0$; or otherwise, (b) there is some $\hat{n}>n_0$ such that at least one of the coefficients $c_{\hat{n}},\tilde c_{\hat{n}}$ is nonzero. In the former case, one can use the same argument in the claim above to obtain that $\tilde{\hat w}_{n_0}(\cdot,\omega)\in \hat V^{0,n_1}(\omega)$. Therefore, $\bigoplus_{i=0}^{n_0}\hat W_i(\omega)\subset \hat V^{0,n_1}(\omega)$. So, $\hat V^{0,n_1}(\omega)=\bigoplus_{i=0}^{n_0}\hat W_i(\omega)$ when $2n_0+1=N_1$. If $2n_0+1<N_1$, it then falls into the latter case (b). For case (b), by repeating the same argument in the claim above, we have either $\hat w_{\hat{n}}(\cdot,\omega)\in \hat V^{0,n_1}(\omega)$, or $\tilde {\hat w}_{\hat{n}}(\cdot,\omega)\in \hat V^{0,n_1}(\omega)$. Since $\dim \hat V^{0,n_1}(\omega)$ is finite, by finite steps, one can finally find an $n'_0$ (in fact $n'_0=N'_1=\frac{N_1-1}{2}$) such that
\begin{equation}\label{E:V-0-n-prime}
\hat V^{0,n_1}(\omega)=\bigoplus_{i=0}^{n'_0}\hat W_i(\omega)=\bigoplus_{i=0}^{N'_1}\hat W_i(\omega),
\end{equation} which is exactly \eqref{E:V-0-n1-charac}.

 In order to prove \eqref{E:V-n1+1-charac}, we choose any $\omega\in \Omega$ and $v_0\in \hat V^{n_1+1,\infty}(\omega)\setminus \{0\}$, it is not difficult to see that \begin{equation}\label{E:v0-V-infty-L2}
v_0=\sum_{k=N'_1+1}^{\infty}(c^0_k\hat w_k(\cdot,\omega)+\tilde c^0_k\tilde {\hat w}_k(\cdot,\omega))
\quad\text{ in }L^2(S^1).
 \end{equation}
 (Otherwise, by a similar estimate as in \eqref{coefficient-control} and in the claim above, one can find some $n_*\leq N'_1$ such that $\hat w_{n_*}(\cdot, \omega)\in \hat V^{n_1+1,\infty}(\omega)$ or $\tilde {\hat w}_{n_*}(\cdot, \omega)\in \hat V^{n_1+1,\infty}(\omega)$. Together with \eqref{E:V-0-n-prime}, it then follows that $\hat V^{0,n_1}(\omega)\cap \hat V^{n_1+1,\infty}(\omega)\ne\{0\}$, a contradiction).

 Since $v_0\ne 0$, there is $n_0\geq N'_1+1$ such that $c^0_{n_0}\neq 0$ or $\tilde c^0_{n_0}\neq 0$. Similarly as the claim above again, one has $\hat w_{n_0}(\cdot,\omega)\in \hat V^{n_1+1,\infty}(\omega)$ or $\tilde {\hat w}_{n_0}(\cdot,\omega)\in \hat V^{n_1+1,\infty}(\omega)$. So, Lemma \ref{floquet-invar}(i) implies that $\hat W_k(\omega)\subset \hat V^{n_1+1,\infty}(\omega)$ for all $k\ge n_0+1$.
Moreover, we can even obtain that
 \begin{equation}\label{E:W-k-V-infty}
\hat W_k(\omega)\subset \hat V^{n_1+1,\infty}(\omega),\text{ for all }k\geq N'_1+1.
 \end{equation}
 As a matter of fact, suppose that \eqref{E:W-k-V-infty} does not hold. Then
there is at least some 1-dimensional subspace $V'(\omega)\subset \bigoplus_{k=N'_1+1}^{n_0}\hat W_k(\omega)(\subset X)$ such that $V'(\omega)\cap \hat V^{n_1+1,\infty}(\omega)=\{0\}$. We here assert that $V'(\omega)\subset \hat V^{0,n_1}(\omega)$. (Otherwise, one can find a nonzero $v_*=v_1+v_2\in V'(\omega)$ with $v_1\in \hat V^{0,n_1}(\omega)\setminus \{0\}$ and $v_2\in \hat V^{n_1+1,\infty}(\omega)\setminus \{0\}$.  So, by \eqref{E:V-0-n-prime}-\eqref{E:v0-V-infty-L2}, we obtain a contradiction to $V'(\omega)\subset \hat V^{0,n_1}(\omega)$). Based on the assertion, one has $V'(\omega)\subset \hat V^{0,n_1}(\omega)\cap \bigoplus_{k=N'_1+1}^{n_0}\hat W_k(\omega)$, contradicting
\eqref{E:V-0-n-prime}. Therefore, \eqref{E:W-k-V-infty} has been proved.

 In the following, we will show how \eqref{E:W-k-V-infty} implies \eqref{E:V-n1+1-charac}. On the one hand, it is clear from \eqref{E:W-k-V-infty} that $\mathrm{cl}\{\cup_{m\geq  N'_1+1}\bigoplus_{k=N'_1+1}^m\hat W_k(\omega)\}\subseteq\hat V^{n_1+1,\infty}(\omega)$. On the other hand, for any nonzero $v\in \hat V^{n_1+1,\infty}(\omega)$, by the density mentioned in Remark \ref{invariant-space}(ii), there is a sequence $v^j\in \hat V^{n_1+1,\infty}(\omega_{-1})$ such that $\|\hat\psi(1,\cdot;v^j,\omega_{-1})-v\|\to 0$ as $j\to \infty$. By applying \eqref{E:v0-V-infty-L2} to $v^j$, one has $\norm{v^j_m-v^j}_{L^2(S^1)}\to 0(m\to \infty)$, where $v^j_m=\sum_{k=N'_1+1}^{m}(c^{0,j}_k\hat w_k(\cdot,\omega_{-1})+\tilde c^{0,j}_k\tilde {\hat w}_k(\cdot,\omega_{-1}))$ for some $c^{0,j}_k,\tilde c^{0,j}_k$ with $k=N'_1+1,\cdots,m.$ It then follows from the basic regularity result (see \cite[Theorem 3.3 and p.296]{Chow1995}) that  $\norm{\hat\psi(1,\cdot;v_m^j,\omega_{-1})-\hat\psi(1,\cdot;v^j,\omega_{-1})}_{X}\to 0$. As a consequence, $\hat\psi(1,\cdot;v^j,\omega_{-1})\in \mathrm{cl}\{\cup_{m\geq  N'_1+1}\bigoplus_{k=N'_1+1}^m\hat W_k(\omega)\}$; and hence $v\in\mathrm{cl}\{\cup_{m\geq  N'_1+1}\bigoplus_{k=N'_1+1}^m\hat W_k(\omega)\}$. By the arbitrariness of $v$, we have obtained that $\hat V^{n_1+1,\infty}(\omega)\subseteq\mathrm{cl}\{\cup_{m\geq  N'_1+1}\bigoplus_{k=N'_1+1}^m\hat W_k(\omega)\}$, which completed the proof of \eqref{E:V-n1+1-charac}.\vskip 2mm

{\it Case {\rm (ii)}: $N_1$ is even}. By virtue of \eqref{E:V-01-including}, it is easy to see that, if $2n_0=N_1$, then $$\bigoplus_{i=0}^{n_0-1}\hat W_i(\omega) \oplus \mathrm{span}\{\hat w_{n_0}(\cdot,\omega)\}=\hat V^{0,n_1}(\omega).$$ Thus, by letting $N_1^\prime=n_0=\frac{N_1}{2}$, we have obtained \eqref{E:V-0-n1-charac-even} with $V'(\omega)=\mathrm{span}\{\hat w_{n_0}(\cdot,\omega)\}\subset \hat W_{n_0}(\omega)$.

If $2n_0<N_1$, then one can repeat the similar argument for the proof of \eqref{E:V-0-n-prime} (i.e.,
\eqref{E:V-0-n1-charac}) to find an  $n^*_0=\frac{N_1}{2}>n_0$ such that $\hat V^{0,n_1}(\omega)=\bigoplus_{i=0}^{n^*_0-1}\hat W_i(\omega)\oplus V'(\omega)$, where $V'(\omega)$ is 1-dimensional subspace of $\hat W_{n^*_0}(\omega)$. Thus, we have proved  \eqref{E:V-0-n1-charac-even}.

Finally, the proof of \eqref{E:V-n1+1-charac-even} is similar as the proof of \eqref{E:V-n1+1-charac}.
\end{proof}

\begin{lemma}\label{zero-inva0}
For given $0\leq n_1\leq n_2\leq\infty$ (when $n_2=\infty$, $n_1<\infty$ is needed), we have
\begin{equation}\label{E:esti-V-n1-n2-lap}
N_1\leq z(v(\cdot))\leq N_2,\,\,\text{ for any }v\in \hat V^{n_1,n_2}(\omega)\setminus\{0\},
\end{equation} where
\begin{equation*}
N_1=\left\{
\begin{split}
 &{\rm dim}\hat V^{0,n_1-1}(\omega),\,\quad\,\,\,\text{ if }{\rm dim}\hat V^{0,n_1-1}(\omega)\text{ is even;}\\
 &{\rm dim}\hat V^{0,n_1-1}(\omega)+1,\,\text{ if }{\rm dim}\hat V^{0,n_1-1}(\omega)\text{ is odd,}
\end{split}\right.
\end{equation*} and
\begin{equation*}
N_2=\left\{
\begin{split}
 &{\rm dim}\hat V^{0,n_2}(\omega),\,\quad\,\,\,\text{ if }{\rm dim}\hat V^{0,n_2}(\omega)\text{ is even;}\\
 &{\rm dim}\hat V^{0,n_2}(\omega)-1,\,\text{ if }{\rm dim}\hat V^{0,n_2}(\omega)\text{ is odd.}
\end{split}\right.
\end{equation*}
Here, we define $\hat V^{0,-1}(\omega)=\{0\}.$
\end{lemma}
\begin{proof}
Given any $n\in\mathbb{N}\cup\{0\}$ and $v\in \hat V^{0,n}(\omega)\setminus \{0\}$, it follows from Lemma \ref{floquet-bundle}(ii), \eqref{E:V-0-n1-charac} and \eqref{E:V-0-n1-charac-even} in Lemma \ref{floquet-zeron} that $0\leq z(v)\leq N$, where
\begin{equation*}
N=\left\{
\begin{split}
 &{\rm dim}\hat V^{0,n}(\omega),\,\quad\,\,\,\text{ if }{\rm dim}\hat V^{0,n}(\omega)\text{ is even;}\\
 &{\rm dim}\hat V^{0,n}(\omega)-1,\,\text{ if }{\rm dim}\hat V^{0,n}(\omega)\text{ is odd.}
\end{split}\right.
\end{equation*}

Given $w\in \hat V^{n+1,\infty}(\omega)\setminus\{0\}$, we will show that $z(w)\geq N+1$ whenever
$\dim \hat V^{0,n}(\omega)$ is odd. To this purpose, let $\hat\psi(t,\cdot;w,\omega)$ be the solution of \eqref{no-gradient} with $\hat\psi(0,\cdot;w,\omega)=w$.  Then, by Lemma \ref{zero-number}(c), there exist $T>0$ and $N_0\in\mathbb{N}\cup\{0\}$ such that $z(\hat\psi(T,\cdot;w,\omega))=N_0$ and $\hat\psi(T,\cdot;w,\omega)$ only has simple zeros.
Consequently, one can find a small $\delta>0$ such that $z(v)=N_0$ whenever $\|v-\hat\psi(T,\cdot;w,\omega)\|<\delta$. On the other hand,
noticing that $\hat\psi(T,\cdot;w,\omega)\in \hat V^{n+1,\infty}(\omega\cdot T)$,
it then follows from \eqref{E:V-n1+1-charac} in Lemma \ref{floquet-zeron}(i) that $\hat\psi(T,\cdot;w,\omega)\in \mathrm{cl}\{\cup_{m\geq \frac{N}{2}+1}\bigoplus_{k=\frac{N}{2}+1}^m\hat W_k(\omega\cdot T)\}$; and hence,
there exists a sequence $\{w^n\}\subset X$, with $w^n\in \bigoplus_{k=\frac{N}{2}+1}^{\frac{N}{2}+1+n}\hat W_k(\omega\cdot T)$, such that $\norm{w^n-\hat\psi(T,\cdot;w,\omega)}\to 0$ as $n\to\infty$. Therefore,
$z(w^n)=N_0$ for all $n$ sufficiently large.
Recall that $w^n\in \bigoplus_{k=\frac{N}{2}+1}^{\frac{N}{2}+1+n}\hat W_k(\omega\cdot T)$. Then Lemma \ref{floquet-bundle}(ii) implies that $N_0\geq N+1$. Together with Lemma \ref{zero-number}(a), we have $z(w)\ge N_0\geq N+1$. Thus, we have proved $z(w)\geq N+1$ whenever
$\dim \hat V^{0,n}(\omega)$ is odd.
Similarly, we can prove $z(w)\geq N$ whenever $\dim \hat V^{0,n}(\omega)$ is even, by
utilizing \eqref{E:V-n1+1-charac-even} in Lemma \ref{floquet-zeron}(ii).

Based on the argument in the previous paragraph, one can easily obtain \eqref{E:esti-V-n1-n2-lap}. As a matter of fact, due to the definition of $\hat V^{n_1,n_2}(\omega)$ in \eqref{twoside-estimate}, it is easy to see that $\hat V^{n_1,n_2}(\omega)=\hat V^{0,n_2}(\omega)\cap \hat V^{n_1,\infty}(\omega)$.
So, for any $v\in \hat V^{n_1,n_2}(\omega)$, on the one hand, $v\in \hat V^{n_1,\infty}(\omega)$ yields that $z(v)\geq N_1$. On the other hand, $v\in \hat V^{0,n_2}(\omega)$ implies $z(v)\leq N_2$. Thus, we have completed the proof.
\end{proof}

\subsection{Invariant manifolds  of nonlinear parabolic equations on $S^1$}

Consider
\begin{equation}\label{lipschitz}
v'=A(\omega \cdot t)v+F(v,\omega \cdot t),\quad t>0,\ \omega\in \Omega,\ v\in X
\end{equation}
where $A(\omega)v=v_{xx}+a(x,\omega)v_x+b(x,\omega)v$, and $\omega\cdot t$ is as in \eqref{linear-equation2}, $F(\cdot,\omega)\in C^1(X,X_0)$, $F(v,\cdot)\in C^0(\Omega,X_0)$ ($v\in X$), $F^\omega(t,v)=F(v,\omega\cdot t)$ is H\"older continuous in $t$, and $F(v,y)=o(\|v\|)$ ($X_0=L^2(S^1)$), here  $X$ is as in the previous subsection. It is well-known that the solution operator $\Phi_t(\cdot,\omega)$ of \eqref{lipschitz} exists in the usual sense, that is, for any $v\in X$, $\Phi_0(v,\omega)=v$, $\Phi_t(v,\omega)\in \mathcal{D}(A(\omega\cdot t))$; $\Phi_t(v,\omega)$ is differentiable in $t$ with respect to $X_0$ norm and satisfies \eqref{lipschitz} for $t>0$.\par

Suppose that $\sigma(\Omega)=\cup_{k=0}^{\infty}I_k$ is the spectrum of \eqref{linear-equation2}. The following lemma can be proved by using arguments as in \cite{Chow1991, Chow1994-2, Hen}
\begin{lemma}\label{invari-mani}
 There is a $\delta_0>0$ such that for any $0<\delta^*<\delta_0$ and $0\leq n_1\leq n_2\leq\infty$, \eqref{lipschitz} admits for each $\omega\in \Omega$ a local invariant manifold $W^{n_1,n_2}(\omega,\delta^*)$ with the following properties.
 \begin{itemize}
 \item [\rm{(i)}] There are $K_0>0$, and bounded continuous function $h^{n_1,n_2}(\omega): V^{n_1,n_2}(\omega)$ $\rightarrow V^{n_2+1,\infty}(\omega)\oplus V^{0,n_1-1}(\omega)) $ being $C^1$ for each fixed $\omega\in \Omega$, and $h^{n_1,n_2}(v,\omega)$ $=o(\|v\|)$ (uniformly in $\omega\in\Omega$), $\|(\partial h^{n_1,n_2}/\partial v)(v,\omega)\|\leq K_0$ for all $\omega\in \Omega$, $v\in V^{n_1,n_2}(\omega)$ such that
 \begin{eqnarray*}
  W^{n_1,n_2}(\omega,\delta^*)=\left\{ v_0^{n_1,n_2}+h^{n_1,n_2}(v_0^{n_1,n_2},\omega):v_0^{n_1,n_2}\in V^{n_1,n_2}(\omega) \cap \{v\in X: \|v\|<\delta^*\}\right\}.
 \end{eqnarray*}
 Moreover, $W^{n_1,n_2}(\omega,\delta^*)$ are diffeomorphic to $V^{n_1,n_2}(\omega)\cap\{v\in X| \|v\|<\delta^*\}$, and $W^{n_1,n_2}(\omega,\delta^*)$ are tangent to $V^{n_1,n_2}(\omega)$ at $0\in X$ for each $\omega\in \Omega$.
 \item [\rm{(ii)}] $W^{n_1,n_2}(\omega,\delta^*)$ is locally invariant in the sense that if $v\in W^{n_1,n_2}(\omega,\delta^*)$ and $\norm{\Phi_t(v,\omega)}<\delta^*$ for all $t\in [0,T]$, then $\Phi_t(v,\omega)\in W^{n_1,n_2}(\omega\cdot t,\delta^*)$ for all $t\in [0,T]$.

 \end{itemize}
\end{lemma}
\begin{remark}\label{invariantly}
{\rm
(1) The existence of $\delta_0$ in the above lemma which is independent of $n_1$ and $n_2$ is due to the increasing of the gaps between the spectrum intervals $I_n$ and $I_{n+1}$ as increases.\par
(2) Note that, as usual, $W^{n_1,n_2}(\omega,\delta^*)$ is constructed in terms of appropriate rate conditions for the solutions of \eqref{lipschitz} by replacing $F$ by a cutoff function $\tilde{F}$. It then follows that for any $n_1\leq n_2\leq n_3\leq \infty$ and $\omega\in \Omega$, $W^{n_1,n_2}(\omega,\delta^*)\subset W^{n_1,n_3}(\omega,\delta^*)$, and for any $u\in W^{n_1,\infty}(\omega,\delta^*)$, there are $u_n\in W^{n_1,n}(\omega,\delta^*)$ ($n_1\leq n <\infty$) such that $u_n\rightarrow u$ as $n\rightarrow \infty$.\par
(3) For any $0\leq n_1\leq n_2<\infty$ and $\omega\in \Omega$, there are $\tau>0$ and $0<\delta_1^*<\delta_0$ (here $\tau$ and $\delta_1^*$ are independent of $\omega$), such that $\Phi_t(W^{n_1,n_2}(\omega,\delta_1^*),\omega)\subset W^{n_1,n_2}(\omega\cdot t,\delta^*)$ for any $t$ with $|t|<\tau$.
}
\end{remark}
 If $0\in I_{n_0}=[a_{n_0},b_{n_0}]\subset\sigma(\Omega)$. Then $W^s(\omega,\delta^*)=W^{n_0+1,\infty}(\omega,\delta^*)$, $W^{cs}(\omega,\delta^*)=W^{n_0,\infty}(\omega,\delta^*)$, $W^c(\omega,\delta^*)=W^{n_0,n_0}(\omega,\delta^*)$, $W^{cu}(\omega,\delta^*)=W^{0,n_0}(\omega,\delta^*)$, and $W^u(\omega,\delta^*)=W^{0,n_0-1}(\omega,\delta^*)$ are referred to as {\it local stable, center stable, center, center unstable, and unstable manifolds} of \eqref{lipschitz} at $\omega\in \Omega$, respectively.\par
\begin{remark}\label{overflow-invariant}
{\rm
(1) $W^s(\omega,\delta^*)$ and $W^u(\omega,\delta^*)$ are overflowing invariant in the sense that if $\delta^*$ is sufficiently small, then
\[
\Phi_t(W^s(\omega,\delta^*),\omega)\subset W^s(\omega\cdot t,\delta^*),
\]
for $t$ sufficiently positive, and
\[
\Phi_t(W^u(\omega,\delta^*),\omega)\subset W^u(\omega\cdot t,\delta^*),
\]
for $t$ sufficiently negative. Moreover, one can find constants $\alpha$, $C>0$, such that for any $\omega\in \Omega$, $v^s\in W^s(\omega,\delta^*)$, $v^u\in W^u(\omega,\delta^*)$,
\begin{equation}\label{exponen-decrea}
\begin{split}
  \|\Phi_t(v^s,\omega)\|\leq Ce^{-\frac{\alpha}{2}t}\|v^s\|\quad \text{for}\ t\geq 0,\\
  \|\Phi_t(v^u,\omega)\|\leq Ce^{\frac{\alpha}{2}t}\|v^u\|\quad \text{for}\ t\leq 0.
\end{split}
\end{equation}

(2) By the invariant foliation theory in \cite{Chow1991, Chow1994-2}, one has that for any $\omega\in \Omega$,
\[
W^{cs}(\omega,\delta^*)={\cup}_{u_c\in W^c(\omega,\delta^*)}\bar{W}_s(u_c,\omega,\delta^*),
\]
where $\bar{W}_s(u_c,\omega,\delta^*)$ is the so-called {\it stable leaf} of \eqref{lipschitz} at $u_c$, and it is invariant in the sense that if $\tau>0$ is such that $\Phi_t(u_c,\omega)\in W^{c}(\omega\cdot t,\delta^*)$ and $\Phi_t(u,\omega)\in W^{cs}(\omega,\delta^*)$ for all $0\leq t<\tau$, where $u\in \bar{W}_s(u_c,\omega,\delta^*)$, then $\Phi_t(u,\omega)\subset \bar{W}_s(\Phi_t(u_c,\omega),\omega\cdot t,\delta^*)$ for $0\leq t<\tau$. Moreover, there are $K,\beta >0$ independent of the $\omega\in\Omega$ such that for any $u\in\bar{W}_s(u_c,\omega,\delta^*)$ ($u_c\neq 0$) and $\tau>0$ with $\Phi_t(u,\omega)\in W^{cs}(\omega\cdot t,\delta^*)$, $\Phi_t(u_c,\omega)\in W^{c}(\omega\cdot t,\delta^*)$ for $0\leq t<\tau$, one has that
\[
\frac{\|\Phi_t(u,\omega)-\Phi_t(u_c,\omega)\|}{\|\Phi_t(u_c,\omega)\|}\leq Ke^{-\beta t}\frac{\|u-u_c\|}{\|u_c\|}
\] and
\[
\|\Phi_t(u,\omega)-\Phi_t(u_c,\omega)\|\le  Ke^{-\beta t}\|u-u_c\|
\]
for $0\leq t<\tau$.
}
\end{remark}

\section{Almost Automorphically and Almost Periodically Forced Circle Flows}

In this section, we  consider  the structure of minimal sets of \eqref{equation-lim1} in {\it the case that
$f=f(t,u,u_x)$.} \vskip 2mm

Note that  \eqref{equation-lim1} generates a (local) skew-product semiflow $\Pi^t$ on $X\times H(f)$:
 \begin{equation}\label{skew-product2}
\Pi^t(u_0,g)=(\varphi(t,\cdot;u_0,g),g\cdot t),
\end{equation}
where $X$ is also defined as in introduction and $\varphi(t,\cdot;u_0,g)$ is the solution of \eqref{equation-lim1}    with $\varphi(0,\cdot,u_0,g)=u_0(\cdot)$, and $g\cdot t$ denotes the flow on $H(f)$.

Let $\Omega \subset X\times H(f)$ be a compact and connected invariant set of \eqref{skew-product2}. For any $\omega=(u_0,g)\in \Omega$, denote $\omega\cdot t=\Pi^t(u_0,g)$. Let $v=\varphi(t,\cdot;u,g)-\varphi(t,\cdot;u_0,g)$. Then $v$ satisfies the following equation:
\begin{equation}\label{variation-equation}
v_t=v_{xx}+a(x,\omega\cdot t)v_x+b(x,\omega\cdot t)v+F(v,\omega\cdot t),\,\,t>0,\,x\in S^{1}=\mathbb{R}/2\pi \mathbb{Z},
\end{equation}
where $F(v,\omega\cdot t)=g(t,v+u_0,v_x+u_{0x})-g(t,u_0,u_{0x})-a(x,\omega\cdot t)-b(x,\omega\cdot t)v$, $a(x,\omega\cdot t)=g_p(t,u_0,u_{0x})$, $ b(x,\omega\cdot t)=g_u(t,u_0,u_{0x})$.\par

Denote $A(\omega)=\frac{\partial^2}{\partial x^2}+a(\cdot,\omega)\frac{\p }{\p x}+b(\cdot,\omega)$. Then \eqref{equation-lim1} can be rewritten as
\begin{equation}\label{linear-opera2}
v'=A(\omega\cdot t)v+F(v,\omega\cdot t).
\end{equation}
Suppose that $\sigma(\Omega)=\cup_{k=0}^{\infty}I_k$ ($I_k$ is ordered from right to left) is the Sacker-Sell spectrum of the linear equation associated with \eqref{linear-opera2}:
\begin{equation}\label{associate-eq}
v'=A(\omega\cdot t)v,\,\,t>0,\,\omega\in \Omega,\,v\in X.
\end{equation}
For any given $0\leq n_1\leq n_2\leq \infty$, let $V^{n_1,n_2}(\omega)$ be the invariant subspace of \eqref{associate-eq} associated with the spectrum set $\cup_{k=n_1}^{n_2}I_k$ at $\omega\in \Omega$.
If $0\in\sigma(\Omega)$,
let  $V^s(\omega)$, $V^{cs}(\omega)$, $V^{c}(\omega)$, $V^{cu}(\omega)$, and $V^u(\omega)$  be the
 {\it stable, center stable, center, center unstable}, and {\it unstable subspaces} of \eqref{associate-eq} at $\omega\in \Omega$, respectively.

For given $u\in X$, let $\sigma_a$ be the $S^1$-action on $u$ induced by shifting $x$
\begin{equation*}
  (\sigma_au)(x):=u(x+a),\quad a\in S^1=\mathbb{R}/2\pi\mathbb{Z}
\end{equation*}
and define
$$
\Sigma u=\{\sigma_au\,:\, a\in S^1\}.
$$
Since $f=f(t,u,u_x)$ in this section, it is not difficult to see that the action of $S^1$ on the solution $\sigma_a\varphi(t,\cdot;u,g)$ is equivariant (also called translation invariance), i.e., $\sigma_a\varphi(t,\cdot;u,g)=\varphi(t,\cdot;\sigma_au,g)$ for all $t\in \mathbb{R}^+,u\in X,a\in S^1$ and $g\in H(f)$.

The main results of this section are stated in the following theorems.

\begin{theorem}\label{sysm-embed}
 Let $M \subset X\times H(f)$ be a spatially inhomogeneous minimal set of \eqref{skew-product2}.
\begin{itemize}
\item[{\rm (1)}] If ${\rm dim}V^c(\omega)=2$ with ${\rm dim}V^u(\omega)$ being odd,  then there is a residual invariant set $Y_0\subset H(f)$ such that, for any $g\in Y_0$, there exists $u_g \in X$ satisfying $p^{-1}(g)\cap M\subset (\Sigma u_g,g)$. Moreover, there is a $C^1$-function $c^g:\mathbb{R}\to \mathcal{S}^1; t\mapsto c^g(t)$ (with its derivative $\dot c^g(t)$ being almost-automorphic in $t$) such that
\begin{equation}\label{E:transla-cg-func}
\varphi(t,x,u_g,g)=u_{g\cdot t}(x+c^g(t)),
\end{equation}
where $\mathcal{S}^1=\mathbb{R}/L\mathbb{Z}$ and $L$ is the smallest common spatial-period of any element in $M$.

\item[{\rm (2)}] If ${\rm dim}V^c(\omega)=1$, then $Y_0=H(f)$ in {\rm (1)}; and moreover, the derivative $\dot c^g(t)$ is almost-periodic in $t$.


\end{itemize}
\end{theorem}

\begin{theorem}\label{ergodic-thm}
Let $M \subset X\times H(f)$ be a minimal set of \eqref{skew-product2}.
\begin{itemize}
\item[{\rm (1)}]
If $M$ is uniquely ergodic, then ${\rm dim}V^c(\omega)\le 2$. Moreover, if ${\rm dim} V^c(\omega)=2$, then $\dim V^u(\omega)$ must be odd.

\item[{\rm (2)}] If $M$ is spatially homogeneous, then either ${\rm dim}V^u(\omega)=0$ or
${\rm dim}V^u(\omega)$ is odd.
\end{itemize}
\end{theorem}

\begin{remark}\label{rotation-rmark}
{\rm Theorem \ref{sysm-embed} implies that, under certain circumstances, any spatially inhomogeneous minimal set $M$ can be residually embedded into an almost automorphically forced circle-flow $S^1\times H(f)$. In particular, if $M$ is {\it normally hyperbolic} (i.e., ${\rm dim}V^c(\omega)=1$), $M$ can be totally embedded into an almost periodically forced circle-flow. Thus we have generalized the results of Sandstede and Fiedler \cite{SF1} to time almost-periodic systems.
However, for almost periodically forced circle flow, it is already known that such flow can still be very complicated (See \cite{HuYi} and the references therein).}
\end{remark}

To prove the above  theorem,  we need to present some lemmas.

For given $0\leq n_1\leq n_2\leq\infty$ and $\omega=(u_0,g)\in \Omega$, by Lemma \ref{invari-mani}, there is a well-defined local invariant manifold $W^{n_1,n_2}(\omega,\delta^*)$ of \eqref{linear-opera2}. Let
\begin{equation}\label{gener-inva}
M^{n_1,n_2}(\omega,\delta^*)=\{u\in X:u-u_0\in W^{n_1,n_2}(\omega,\delta^*)\}.
\end{equation}
$M^{n_1,n_2}(\omega,\delta^*)$ is referred to as a {\it local invariant manifold} of \eqref{skew-product2} at $(u_0,g)$.\par
Suppose  that $0\in \sigma(\Omega)$ and $n_0$ is such that $I_{n_0}=[a_{n_0},b_{n_0}]\subset \sigma(\Omega)$ with $a_{n_0}\leq 0\leq b_{n_0}$. For given $\omega=(u_0,g)\in \Omega$, $\delta^*>0$, let
\begin{eqnarray*}
\begin{split}
  M^s(\omega,\delta^*)=& M^{n_0+1,\infty}(\omega,\delta^*)\\
  M^{cs}(\omega,\delta^*)=& M^{n_0,\infty}(\omega,\delta^*)\\
  M^c(\omega,\delta^*)=& M^{n_0,n_0}(\omega,\delta^*)\\
  M^{cu}(\omega,\delta^*)=& M^{0,n_0}(\omega,\delta^*)\\
  M^{u}(\omega,\delta^*)=& M^{0,n_0-1}(\omega,\delta^*)
\end{split}
\end{eqnarray*}
Then $M^s(\omega,\delta^*)$, $M^{cs}(\omega,\delta^*)$, $M^{c}(\omega,\delta^*)$, $M^{cu}(\omega,\delta^*)$, and $M^{u}(\omega,\delta^*)$ are continuous in $\omega\in \Omega$ and referred to as {\it local stable, center stable, center, center unstable}, and {\it unstable manifolds} of \eqref{skew-product2} at $\omega=(u_0,g)\in \Omega$, respectively.\par
\begin{remark}\label{stable-leaf}
{\rm
By Remark \ref{overflow-invariant}(2), for any $\omega=(u_0,g)\in \Omega$, one has
\[
M^{cs}(\omega,\delta^*)=\cup_{u_c\in M^c(\omega,\delta^*)}\bar{M}_s(u_c,\omega,\delta^*)
\]
where $\bar{M}_s(u_c,\omega,\delta^*)=\{u\in X:u-u_0\in \bar{W}^s(u_c-u_0,\omega,\delta^*)\}$ and $\bar{M}_s(u_0,\omega,\delta^*)=M^s(\omega,\delta^*)$. Moreover, there are $K$, $\beta>0$, which are independent of $\omega\in\Omega$, such that for any $u^*\in \bar{M}_s(u_c,\omega,\delta^*)$, $u_c\neq u_0$, and $\tau>0$ with $\varphi(t,\cdot;u^*,g)\in M^{cs}(\omega\cdot t,\delta^*)$, $\varphi(t,\cdot;u_c,g)\in M^c(\omega\cdot t,\delta^*)$ for any $0\leq t<\tau$, one has that
\begin{equation}\label{foliation-contr}
\frac{\|\varphi(t,\cdot;u^*,g)-\varphi(t,\cdot;u_c,g)\|}{\|\varphi(t,\cdot;u_c,g)-\varphi(t,\cdot;u_0,g)\|}\leq Ke^{-\beta t}\frac{\|u^*-u_c\|}{\|u_c-u_0\|}
\end{equation}
and
\begin{equation}\label{foliation-contr-a1}
\norm{\varphi(t,\cdot;u^*,g)-\varphi(t,\cdot;u_c,g)} \leq  Ke^{-\beta t}\|u^*-u_c\|
\end{equation}
for $0\leq t<\tau$.
}
\end{remark}
\begin{lemma}\label{center-foliation}
Let $\delta^*>0$ as in Lemma \ref{invari-mani}. Then for any $\delta\in (0,\delta^*)$, there exists some $\delta_{cs}\in (0,\delta)$ such that, for any $\omega=(u_0,g)\in \Omega$, $u_c\in M^c(\omega,\delta_{cs})\setminus\{u_0\}$ and $u^*\in \bar{M}_s(u_c,\omega,\delta_{cs})\setminus\{u_0\}$, the following statements hold:
 \begin{itemize}
 \item[{\rm (i)}]
  Let $\tau>0$ be any number such that $\varphi(t,\cdot;u_c,g)\in M^c(\omega\cdot t,\delta)\setminus\{\varphi(t,\cdot;u_0,g)\}$ for $0\leq t<\tau$. Then $\varphi(t,\cdot;u^*,g)\in M^{cs}(\omega\cdot t,\delta^*)$ for $0\leq t<\tau$.
\item[{\rm (ii)}]
 Let $\tau>0$ be  any number such that $\varphi(t,\cdot;u_c,g)\in M^c(\omega\cdot t,\delta)\setminus\{\varphi(t,\cdot;u_0,g)\}$ for $0\leq t<\tau$, but $\varphi(\tau,\cdot;u_c,g)\notin M^c(\omega\cdot \tau,\delta)\setminus\{\varphi(\tau,\cdot;u_0,g)\}$. Then for any $\epsilon>0$, one may take such $\delta_{cs}>0$ smaller so that
\begin{equation}\label{bound-estimate1}
  \frac{\|\varphi(\tau,\cdot;u^*,g)-\varphi(\tau,\cdot;u_c,g)\|}{\|\varphi(\tau,\cdot;u_c,g)-
  \varphi(\tau,\cdot;u_0,g)\|}\leq \epsilon.
\end{equation}
\item[{\rm (iii)}] If $\varphi(t,\cdot;u_c,g)\in M^c(\omega\cdot t,\delta)\setminus\{\varphi(t,\cdot;u_0,g)\}$ for any $t>0$, then
\begin{equation}\label{bound-estimate12}
  \frac{\|\varphi(t,\cdot;u^*,g)-\varphi(t,\cdot;u_c,g)\|}{\|\varphi(t,\cdot;u_c,g)-
  \varphi(t,\cdot;u_0,g)\|}\leq \epsilon.
\end{equation}
for all $t>0$ sufficiently large.
 \end{itemize}
\end{lemma}
\begin{proof}
 The statement in this lemma can also be found in \cite[p.313]{ShenYi-JDDE96}. For the sake of completeness, we give a detailed proof below.

(i) Suppose on the contrary that there exist sequences $\delta_n\rightarrow 0$, $\omega_n\triangleq(u_n,g_n)\in \Omega$, $u_c^n\in M^c(\omega_n,\delta_n)$, $u_n^*\in \bar{M}_s(u_c^n,\omega_n,\delta_n)\setminus\{u_c^n\}$ and $\tau_n>0$ satisfying that
\begin{equation}\label{local-ome-cs-0}
\varphi(t,\cdot; u_c^n,g_n)\in M^c(\omega_n\cdot t,\delta)\,\, \text{ for }0\leq t<\tau_n,
 \end{equation}
 while one can find some $t_n\in [0,\tau_n)$ such that
\begin{equation}\label{local-ome-cs}
\varphi(t,\cdot;u_n^*,g_n)\in M^{cs}(\omega_n\cdot t,\delta^*)\text{ for }t\in [0,t_n), \text{ but } \varphi(t_n,\cdot;u_n^*,g_n)\notin M^{cs}(\omega_n\cdot t_n,\delta^*).
\end{equation}
Recall that $\delta_n\to 0$. Then
$\norm{u_c^n-u_n}\to 0$ and $\norm{u^*_n-u_c^n}\to 0$ as $n\to \infty.$
If $\{t_n\}$ is bounded, then $\|\varphi(t,\cdot;u^*_n,g_n)-\varphi(t,\cdot;u_n,g_n)\|\rightarrow 0$ uniformly for $t\in [0,t_n]$. This entails that $\|\varphi(t_n,\cdot;u^*_n,g_n)-\varphi(t_n,\cdot;u_n,g_n)\|<\delta^*$ for all $n$ sufficiently large. So, the local invariance of $M^{cs}(\omega,\delta^*)$ implies that $\varphi(t_n,\cdot;u_n^*,g_n)\in M^{cs}(\omega_n\cdot t_n,\delta^*)$, a contradiction to the second statement of \eqref{local-ome-cs}.

If $\{t_n\}$ is unbounded, then it follows from \eqref{local-ome-cs-0} and the first statement of \eqref{local-ome-cs} that $\varphi(t,\cdot;u^*_n,g_n)\in \bar{M}_s(\varphi(t,\cdot; u_c^n,g_n),\omega_n\cdot t,\delta^*),$ for $t\in [0,t_n)$. So, by \eqref{foliation-contr-a1},
there exist $K>0,\beta>0$ such that $\|\varphi(t,\cdot;u^*_n,g_n)-\varphi(t,\cdot;u_c^n,g_n)\|\leq Ke^{-\beta t}\|u^*_n-u_c^n\|$ for $0\le t<t_n$. Consequently, $\|\varphi(t_n,\cdot;u^*_n,g_n)-\varphi(t_n,\cdot;u_c^n,g_n)\|<\frac{\delta^*-\delta}{2}$, for $n$ sufficiently large. Thus,
  \begin{equation*}
  \begin{split}
    &\|\varphi(t_n,\cdot;u^*_n,g_n)-\varphi(t_n,\cdot;u_n,g_n)\|\\
    &\leq \|\varphi(t_n,\cdot;u^*_n,g_n)-\varphi(t_n,\cdot;u_c^n,g_n)\|
    +\|\varphi(t_n,\cdot;u_c^n,g_n)-\varphi(t_n,\cdot;u_n,g_n)\|\\
    &<\frac{\delta^*-\delta}{2}+\delta<\delta^*, \,\,\text{ for } n \text{ sufficiently large. }
  \end{split}
  \end{equation*}
   (Here, $\|\varphi(t_n,\cdot;u_c^n,g_n)-\varphi(t_n,\cdot;u_n,g_n)\|\le \delta$ is due to
  \eqref{local-ome-cs-0}.)
 By the local invariance of $M^{cs}(\omega,\delta^*)$ again, we obtain $\varphi(t_n,\cdot;u_n^*,g_n)\in M^{cs}(\omega_n\cdot t_n,\delta^*)$ for $n$ sufficiently large, a contradiction to the second statement of \eqref{local-ome-cs}. Thus, we have proved (i).

(ii) By virtue of (i), one has $\varphi(t,\cdot;u^*,g)\in M^{cs}(\omega\cdot t,\delta^*)$ for $0\leq t<\tau$. It then follows from Remark \ref{stable-leaf} that \eqref{foliation-contr-a1} holds for $0\leq t<\tau$. As a consequence, we have
\begin{equation*}
 \frac{\|\varphi(\tau,\cdot;u^*,g)-\varphi(\tau,\cdot;u_c,g)\|}
  {\|\varphi(\tau,\cdot;u_c,g)-\varphi(\tau,\cdot;u_0,g)\|}\leq
  \frac{Ke^{-\beta\tau}\delta_{cs}}{\delta}\leq\frac{K\delta_{cs}}{\delta}.
  \end{equation*}
So, for any $\epsilon>0$, one can choose $\delta_{cs}$ smaller so that \eqref{bound-estimate1} holds.

(iii) Again, by (i), one has $\varphi(t,\cdot;u^*,g)\in M^{cs}(\omega\cdot t,\delta^*)$ for all $t>0$. So, \eqref{foliation-contr} in Remark \ref{stable-leaf} implies that \eqref{bound-estimate12} holds for any $t$ sufficiently large. Thus, we have completed the proof of this lemma.
\end{proof}

\begin{lemma}\label{L:lap-number-control-on-mfds}
 Let $\delta_0$ be as in Lemma \ref{invari-mani} and sufficiently small. For any given $0\leq n_1\leq n_2\leq\infty$, let $N_1$ and $N_2$ be as in Lemma \ref{zero-inva}. Then the following holds.
 \begin{itemize}
 \item[{\rm (i)}]
 If $n_2<\infty$, then there is a $0<\delta_{n_1,n_2}^*<\delta_0$ such that $N_1\leq z(u(\cdot)-u_0(\cdot))\leq N_2$ for any $u\in M^{n_1,n_2}(\omega,\delta^*_{n_1,n_2})\setminus \{u_0\}$ ($\omega=(u_0,g)\in \Omega$).
 \item[{\rm (ii)}]
 If $n_1\geq0$ is such that $I_{n_1}\subset \mathbb{R}^-=\{\lambda\in \mathbb{R}: \lambda<0\}$, then for any $\delta^*\in (0,\delta_0)$, $n_1\leq n_2\leq \infty$, and $u\in M^{n_1,n_2}(\omega,\delta^*)\setminus\{u_0\}$ ($\omega=(u_0,g)\in \Omega$), one has $z(u(\cdot)-u_0(\cdot))\geq N_1$.
 \item[{\rm (iii)}]
Let $n_1\geq0$ be such that $0\in I_{n_1}$. If $\dim V^{n_1,n_1}(\omega)=1$, or $\dim V^{n_1,n_1}(\omega)=2$ with $\dim V^u(\omega)$ being odd, then there is a $\delta_{n_1,\infty}^*\in (0,\delta_0)$ such that $z(u(\cdot)-u_0(\cdot))\geq N_1$, for any $u\in M^{n_1,\infty}(\omega,\delta^*_{n_1,\infty})\setminus\{u_0\}$ ($\omega=(u_0,g)\in \Omega$).
 \end{itemize}
\end{lemma}
\begin{proof}
(i) Suppose that there are some $0\leq n_1\leq n_2< \infty$, some sequences $\delta_n\rightarrow 0$, $\omega_n=(u_n,g_n)$ and $\tilde{u}_n\in M^{n_1,n_2}(\omega_n,\delta_n)\setminus \{u_n\}$ such that $z(\tilde{u}_n(\cdot)-u_n(\cdot))<N_1$ or $z(\tilde{u}_n(\cdot)-u_n(\cdot))>N_2$. Let $v_n(\cdot)=\frac{\tilde{u}_n(\cdot)-u_n(\cdot)}{\|\tilde{u}_n(\cdot)-u_n(\cdot)\|}$. Since $\tilde{u}_n-u_n\in W^{n_1,n_2}(\omega_n,\delta_n)$ and $n_2<\infty$, it follows from Lemma \ref{invari-mani} that $\tilde{u}_n-u_n=v_n^{n_1,n_2}+h^{n_1,n_2}(v_n^{n_1,n_2},\omega_n)$, where $v_n^{n_1,n_2}\in V^{n_1,n_2}(\omega_n)\setminus\{0\}$ with $\|v_n^{n_1,n_2}\|<\delta_n$ and $\|h^{n_1,n_2}(v,\omega_n)\|=o(\|v\|)$. So, one has
\begin{equation}\label{decompose}
v_n(\cdot) = \frac{v_n^{n_1,n_2}+h^{n_1,n_2}(v_n^{n_1,n_2},\omega_n)}{\|v_n^{n_1,n_2}+h^{n_1,n_2}(v_n^{n_1,n_2},\omega_n)\|}\\
\end{equation}
 Note that $n_2<\infty$ and $\Omega$ is compact. Since $V^{n_1,n_2}(\omega)$ are finite-dimensional, we may assume without loss of generality that $\frac{v_n^{n_1,n_2}}{\|v_n^{n_1,n_2}\|}\rightarrow v^*$ and $\omega_n=(u_n,g_n)\rightarrow \omega^*=(u^*,g^*)$ as $n\rightarrow \infty$. Then by \eqref{decompose}, $v_n\rightarrow v^*\in V^{n_1,n_2}(\omega^*)$. Moreover, for $|t|\ll 1$, by Remark \ref{invariantly}(3),
 \begin{equation*}
 \begin{split}
   \frac{\varphi(t,\cdot;\tilde{u}_n,g_n)-\varphi(t,\cdot;u_n,g_n)}{\|\tilde{u}_n-u_n\|} &=\frac{\frac{v(t,v_n^{n_1,n_2})+h^{n_1,n_2}(v(t,v_n^{n_1,n_2}),\omega_n\cdot t)}{\|v(t,v_n^{n_1,n_2})\|}}{\frac{\|v_n^{n_1,n_2}+h^{n_1,n_2}(v_n^{n_1,n_2},\omega_n)\|}{\|v_n^{n_1,n_2}\|}}\cdot\frac{\|v(t,v_n^{n_1,n_2})\|}{\|v_n^{n_1,n_2}\|}\\
   &\rightarrow \frac{v(t,v^*)}{\|v(t,v^*)\|}\cdot \|v(t,v^*)\|
   =v(t,v^*),
  \end{split}
  \end{equation*}
  as $n\rightarrow \infty$, where $v(t,v^*)$ is the solution of \eqref{associate-eq}, with $\omega$ replaced by $\omega^*$, such that $v(0,v^*)=v^*$. By Lemma \ref{zero-inva}, $N_1\leq z(v(t,v^*))\leq N_2$ for $|t|$ sufficiently small. Let $t_2<0<t_1$ and $|t_1|$, $|t_2|$ so  small that $v(t_1,v^*)$, $v(t_2,v^*)$ have only simple zeros. Then $z(\varphi(t_1,\cdot;\tilde{u}_n,g_n)-\varphi(t_1,\cdot;u_n,g_n))=z(v(t_1,v^*))\geq N_1$ and $z(\varphi(t_2,\cdot;\tilde{u}_n,g_n)-\varphi(t_2,\cdot;u_n,g_n))\leq N_2$ for $n$ sufficiently large. Thus, by Lemma \ref{sigma-function}(a), we obtain $N_1\leq z(\tilde{u}_n(\cdot)-u_n(\cdot))\leq N_2$, a contradiction to the definition of $\tilde u_n(\cdot)$ and $u_n(\cdot)$.
  \vskip 2mm

  (ii) We first prove that (ii) is true for any $n_2<\infty$. In fact, when $\delta_0$ is sufficiently small, Remark \ref{overflow-invariant}(1) implies that, for any $n_1\leq n_2<\infty$, $0<\delta^*<\delta_0$ and $u^*\in M^{n_1,n_2}(\omega,\delta^*)\setminus \{u_0\}$, one has $\varphi(t,\cdot;u^*,g)\in M^{n_1,n_2}(\omega\cdot t,\delta_{n_1,n_2}^*)\setminus\{\varphi(t,\cdot;u_0,g)\}$ for $t$ sufficiently positive, where $\delta^*_{n_1,n_2}$ is as defined in (i). Then, by (i), we obtain that $z(\varphi(t,\cdot;u^*,g)-\varphi(t,\cdot;u_0,g))\geq N_1$ for $t$ sufficiently positive. It then follows from Lemma \ref{sigma-function} (a) that $z(u^*(\cdot)-u_0(\cdot))\geq N_1$.\par
We next consider the case $n_2=\infty$. Let $u^*\in M^{n_1,\infty}(\omega,\delta^*)\setminus\{u_0\}$. By Remark \ref{invariantly} (2), there are $u_n\in M^{n_1,n}(\omega,\delta^*)\setminus\{u_0\}$ ($n_1\leq n<\infty$) such that $u_n\rightarrow u^*$ as $n\rightarrow \infty$. Choose a $t_0>0$ so small that $\varphi(t_0,\cdot;u^*,g)-\varphi(t_0,\cdot;u_0,g)$ has only simple zeros. Then $z(\varphi(t_0,\cdot;u^*,g)-\varphi(t_0,\cdot;u_0,g))=z(\varphi(t_0,\cdot;u_n^*,g)-\varphi(t_0,\cdot;u_0,g))\geq N_1$ for $n$ sufficiently large. This implies that $z(u^*(\cdot)-u_0(\cdot))\geq N_1$.
\vskip 2mm

(iii) If $\dim V^{n_1,n_1}(\omega)=1$, or $\dim V^{n_1,n_1}(\omega)=2$ with $\dim V^u(\omega)$ being odd, then Lemma \ref{zero-inva} implies that $z(w(\cdot))=N_1$ for any $w\in V^{n_1,n_1}(\omega)\setminus\{0\}$.\par

 Since $\dim V^{n_1,n_1}(\omega)<\infty$ and $\Omega$ is compact, one can find $\delta_c>0$ such that
 \begin{equation}\label{local-constant}
 z(w(\cdot,\omega))=z(\frac{w(\cdot,\omega)}{\|w(\cdot,\omega)\|}+v(\cdot))=N_1
\end{equation}
for any $w(\cdot,\omega)\in V^{n_1,n_1}(\omega)\setminus\{0\}$, $v\in X$ with $\norm{v}<\delta_c$ and $\omega\in \Omega$.\par
Choose any $\delta^*>0$ be as defined in Lemma \ref{invari-mani}, so that
\begin{equation}\label{invariant-manifold2}
u-u_0=v_0^{n_1,n_1}+h^{n_1,n_1}(v_0^{n_1,n_1},\omega)
\end{equation}
for any $u\in M^{n_1,n_1}(\omega,\delta^*)\setminus\{u_0\}$, where $ v_0^{n_1,n_1}\in V^{n_1,n_1}(\omega)\setminus\{0\}$ and

\begin{equation}\label{higher-order}
  \frac{\|h^{n_1,n_1}(v_0^{n_1,n_1},\omega)\|}{\|v_0^{n_1,n_1}\|}<\frac{\delta_c}{2},\quad \text{for all}\ \norm{v_0^{n_1,n_1}}\leq \delta^*.
\end{equation}
\par

Now note that $M^{n_1,\infty}(\omega,\delta^*)=M^{cs}(\omega,\delta^*)$ and $M^{n_1,n_1}(\omega,\delta^*)=M^c(\omega,\delta^*)$. By Remark \ref{stable-leaf}, one has $M^{n_1,\infty}(\omega,\delta^*)=\cup_{u_c\in M^c(\omega,\delta^*)}\bar{M}_s(u_c,\omega,\delta^*)$.
Moreover, fix any $\delta\in (0,\delta^*_{n_1,n_1})\subset(0,\delta^*)$, it follows from Lemma \ref{center-foliation} that there is some $\delta_{cs}\in(0,\delta)$ such that for any $\omega\in \Omega$,$u_c\in M^c(\omega,\delta_{cs})\setminus\{u_0\}$ and $u^*\in \bar{M}_s(u_c,\omega,\delta_{cs})\setminus\{u_0\}$, the following statements (a)-(c) hold:\par

(a) If $\tau>0$ is such that $\varphi(t,\cdot;u_c,g)\in M^c(\omega\cdot t,\delta)\setminus\{\varphi(t,\cdot;u_0,g)\}$ for $0\leq t<\tau$, then $\varphi(t,\cdot;u^*,g)\in M^{n_1,\infty}(\omega\cdot t,\delta^*)$ for $0\leq t<\tau$.\par
(b) If $\tau>0$ is such that $\varphi(t,\cdot;u_c,g)\in M^c(\omega\cdot t,\delta)\setminus\{\varphi(t,\cdot;u_0,g)\}$ for $0\leq t<\tau$ and $\varphi(\tau,\cdot;u_c,g)\notin M^c(\omega\cdot \tau,\delta)\setminus\{\varphi(\tau,\cdot;u_0,g)\}$ then
\begin{equation}\label{bounded1}
  \frac{\|\varphi(\tau,\cdot;u^*,g)-\varphi(\tau,\cdot;u_c,g)\|}{\|\varphi(\tau,\cdot;u_c,g)-\varphi(\tau,\cdot;u_0,g)\|}\leq\frac{\delta_c}{2}
\end{equation}\par
(c) If for any $t>0$, $\varphi(t,\cdot;u_c,g)\in M^c(\omega\cdot t,\delta)\setminus\{\varphi(t,\cdot;u_0,g)\}$, then
\begin{equation}\label{bounded2}
  \frac{\|\varphi(t,\cdot;u^*,g)-\varphi(t,\cdot;u_c,g)\|}{\|\varphi(t,\cdot;u_c,g)-\varphi(t,\cdot;u_0,g)\|}\leq\frac{\delta_c}{2}
\end{equation}
for $t>0$ sufficiently large. Hereafter, we write $\delta_{cs}$ as $\delta^*_{n_1,\infty}$.\par
Now let $u^*\in M^{n_1,\infty}(\omega,\delta^*_{n_1,\infty})\setminus\{u_0\}$. If $u^*\in M^{n_1+1,\infty}(\omega,\delta^*_{n_1,\infty})\setminus\{u_0\}$, then by (ii) of this lemma, we have $z(u^*(\cdot)-u_0(\cdot))\geq N_1+2>N_1$. If $u^*\in M^{n_1,\infty}(\omega,\delta^*_{n_1,\infty})\setminus M^{n_1+1,\infty}(\omega,\delta^*_{n_1,\infty})$, then by Remark \ref{stable-leaf}, there is a $u_c\in M^{n_1,n_1}(\omega,\delta^*_{n_1,\infty})\setminus\{u_0\}$ such that $u^*\in \bar{M}_s(u_c,\omega,\delta^*_{n_1,\infty})\setminus\{u_0\}$. Therefore, for any $t\geq 0$, whenever $\varphi(t,\cdot;u_c,g)\in M^{n_1,n_1}(\omega\cdot t,\delta)\setminus\{\varphi(t,\cdot;u_0,g)\}$, it follows from \eqref{invariant-manifold2} that
\begin{eqnarray*}
\begin{split}
& \varphi(t,\cdot;u^*,g)-\varphi(t,\cdot;u_0,g) \\
& =\varphi(t,\cdot;u_c,g)-\varphi(t,\cdot;u_0,g)+\varphi(t,\cdot;u^*,g)-\varphi(t,\cdot;u_c,g)\\
& =\psi(t,\cdot;w,\omega)+ h^{n_1,n_1}(\psi(t,\cdot;w,\omega),\omega\cdot t)+\varphi(t,\cdot;u^*,g)-\varphi(t,\cdot;u_c,g) \\
& =\|\psi(t,\cdot;w,\omega)\|\cdot\left\{\frac{\psi(t,\cdot;w,\omega)}{\|\psi(t,\cdot;w,\omega)\|}
 +\frac{h^{n_1,n_1}((\psi(t,\cdot;w,\omega),\omega\cdot t)}{\|\psi(t,\cdot;w,\omega)\|}
+\frac{\varphi(t,\cdot;u^*,g)-\varphi(t,\cdot;u_c,g)}{\|\psi(t,\cdot;w,\omega)\|}\right\}
\end{split}
\end{eqnarray*}
where $\psi(t,\cdot;w,\omega)=c(t)w_{k}(\cdot,\omega\cdot t)+\tilde{c}(t)\tilde{w}_{k}(\cdot,\omega\cdot t)$, for some $k$, with $c(t)$, $\tilde{c}(t)$
being continuous functions. By \eqref{higher-order} and \eqref{bounded1}-\eqref{bounded2},
\[
\frac{\|h^{n_1,n_1}(\psi(\tau,\cdot;w,\omega),\omega\cdot \tau)\|}{\|\psi(\tau,\cdot;w,\omega)\|}+\frac{\|\varphi(\tau,\cdot;u^*,g)-\varphi(\tau,\cdot;u_c,g)\|}{\|\psi(t,\cdot;w,\omega)\|}\leq \delta_c
\]
for some $\tau>0$. Hence, \eqref{local-constant} directly yields that
\[
z(\varphi(\tau,\cdot;u^*,g)-\varphi(\tau,\cdot;u_0,g))= N_1
\]
which implies that $z(u^*(\cdot)-u_0(\cdot))\geq N_1$. Thus, we have proved (iii).
\end{proof}

\begin{corollary}\label{zero-center}
Let $\omega=(u_0,g)\in \Omega$ and
\begin{equation*}
N_u=\left\{
\begin{split}
 &{\dim} V^u(\omega),\,\quad\,\,\,\text{ if }{\rm dim}V^u(\omega)\text{ is even,}\\
 &{\rm dim}V^u(\omega)+1,\,\text{ if }{\rm dim}V^u(\omega)\text{ is odd.}
\end{split}\right.
\end{equation*}
Suppose that  ${\rm dim}V^u(\omega)\ge 1$ and $0\le {\rm dim}V^c(\omega)\le 2$.
Then for $\delta^*>0$ small enough (independent of $\omega\in \Omega$), one has
\begin{itemize}
\item[{\rm (1)}] If ${\rm dim}V^c(\omega)=0$ and ${\rm dim}V^u(\omega)$ is odd, then
\begin{eqnarray*}
\begin{split}
 & z(u(\cdot)-u_0(\cdot))\geq N_u\quad \text{for }u\in M^{s}(\omega,\delta^*)\setminus\{u_0\},\\
 & z(u(\cdot)-u_0(\cdot))\leq N_u-2\quad \text{for }u\in M^{u}(\omega,\delta^*)\setminus\{u_0\}.\\
\end{split}
\end{eqnarray*}
\item[{\rm (2)}] If ${\rm dim}V^c(\omega)=0$ and ${\rm dim}V^u(\omega)$ is even, then
\begin{eqnarray*}
\begin{split}
 & z(u(\cdot)-u_0(\cdot))\geq N_u\quad \text{for }u\in M^{s}(\omega,\delta^*)\setminus\{u_0\},\\
 & z(u(\cdot)-u_0(\cdot))\leq N_u\quad \text{for }u\in M^{u}(\omega,\delta^*)\setminus\{u_0\}.\\
\end{split}
\end{eqnarray*}

\item[{\rm (3a)}] If ${\rm dim}V^c(\omega)=1$, and ${\rm dim}V^u(\omega)$ is even, then
\begin{eqnarray*}
\begin{split}
 & z(u(\cdot)-u_0(\cdot))> N_u\quad \text{for }u\in M^{s}(\omega,\delta^*)\setminus\{u_0\},\\
 & z(u(\cdot)-u_0(\cdot))=N_u\quad \text{for }u\in M^{c}(\omega,\delta^*)\setminus\{u_0\},\\
 & z(u(\cdot)-u_0(\cdot))\leq N_u\quad \text{for }u\in M^{u}(\omega,\delta^*)\setminus\{u_0\}.\\
\end{split}
\end{eqnarray*}

\item[{\rm (3b)}] If ${\rm dim}V^c(\omega)=1$ and ${\rm dim}V^u(\omega)$ is odd, then
\begin{eqnarray*}
\begin{split}
 & z(u(\cdot)-u_0(\cdot))\geq N_u\quad \text{for }u\in M^{cs}(\omega,\delta^*)\setminus\{u_0\},\\
 & z(u(\cdot)-u_0(\cdot))=N_u\quad \text{for }u\in M^{c}(\omega,\delta^*)\setminus\{u_0\},\\
 & z(u(\cdot)-u_0(\cdot))< N_u\quad \text{for }u\in M^{u}(\omega,\delta^*)\setminus\{u_0\}.\\
\end{split}
\end{eqnarray*}

\item[{\rm (4)}] If ${\rm dim}V^c(\omega)=2$ and ${\rm dim}V^u(\omega)$ is odd, then
\begin{eqnarray*}
\begin{split}
 & z(u(\cdot)-u_0(\cdot))\geq N_u\quad \text{for }u\in M^{cs}(\omega,\delta^*)\setminus\{u_0\},\\
 & z(u(\cdot)-u_0(\cdot))=N_u\quad \text{for }u\in M^{c}(\omega,\delta^*)\setminus\{u_0\},\\
 & z(u(\cdot)-u_0(\cdot))\leq N_u\quad \text{for }u\in M^{cu}(\omega,\delta^*)\setminus\{u_0\}.\\
\end{split}
\end{eqnarray*}

\end{itemize}
\end{corollary}

%
\begin{proof}
 It directly follows from Lemma \ref{L:lap-number-control-on-mfds}.
\end{proof}

\vskip 2mm

\begin{lemma}\label{cap-nonempty}
  Suppose that $0\in\sigma(\Omega)$. Then for any $(u_1,g),(u_2,g)\in \Omega$ with $\norm{u_1-u_2}\ll 1$, one has $V^s(u_1,g)\oplus V^{cu}(u_2,g)=X$, and $V^{cs}(u_1,g)\oplus V^u(u_2,g)=X$. Consequently, $(u_1+V^s(u_1,g))\cap (u_2+V^{cu}(u_2,g))\neq \emptyset$, and $(u_1+V^{cs}(u_1,g))\cap (u_2+V^{u}(u_2,g))\neq \emptyset$.

\end{lemma}
\begin{proof}
 We only prove that $V^s(u_1,g)\oplus V^{cu}(u_2,g)=X$ and $(u_1+V^s(u_1,g))\cap(u_2+V^{cu}(u_2,g))\neq \emptyset$. Let $\tilde P(\omega)$ be projections of \eqref{associate-eq} satisfying: $V^s(\omega)=(I-\tilde P(\omega))X$ and $V^{cu}(\omega)=\tilde P(\omega)X$. Then $\tilde P:\Omega\to L(X,X)$ is continuous. We claim that for any distinct points $(u_1,g),(u_2,g)\in \Omega$ with $\|u_1-u_2\|$ sufficiently small, one has $V^s(u_1,g)\cap V^{cu}(u_2,g)=\{0\}$.

To prove the claim, we assume that for any $n\in \mathbb{N}\setminus\{0\}$, there exist $(u_{1n},g_n), (u_{2n},g_n)\in \Omega$ satisfying $\|u_{1n}-u_{2n}\|\leq \frac{1}{n}$ and $v_n\in V^s(u_{1n},g)\cap  V^{cu}(u_{2n},g)$ with $\|v_n\|=1$. Since $\Omega$ is compact, one can assume that $(u_{1n},g_n)\to (u^*,g^*)$ and $(u_{2n},g_n)\to (u^*,g^*)$ as $n\to \infty$.
The fact that $\dim V^{cu}(\omega)$ is finite, entails that $v_n\to v^*\in V^{cu}(u^*,g^*)\setminus\{0\}$. While on the other hand, the continuity of $\tilde P(\omega)$ in $\omega$ implies that $I-\tilde P(\omega)$ is also continuous with respect to $\omega$. So, $v_n=(I-\tilde P(u_{2n},g_{n}))v_n\to (I-\tilde P(u^*,g^*))v^*$ as $n\to \infty$, that is $v^*\in V^s(u^*,g^*)\cap V^{cu}(u^*,g^*)$, a contradiction. Thus, we have proved our claim.

  Since $\dim V^{cu}(u_1,g)=\dim V^{cu}(u_2,g)<\infty$, $V^s(u_1,g)\oplus V^{cu}(u_1,g)=X$, one has that $V^s(u_1,g)\oplus V^{cu}(u_2,g)=X$. Thus, there is a unique $u_1^s\in  V^s(u_1,g)$ and a unique $u_2^{cu}\in V^{cu}(u_2,g)$ such that $u_1-u_2=u_2^{cu}-u_1^s$, that is $u_1+u_1^s=u_2+u_2^{cu}$. So, $(u_1+V^s(u_1,g))\cap (u_2+V^{cu}(u_2,g))\neq \emptyset$.
\end{proof}

\begin{lemma}\label{cap-nonempty2}
  Suppose that $0\in \sigma(\Omega)$. Then for any $(u_1,y)$, $(u_2,y)\in \Omega$ with $\|u_1-u_2\|\ll 1$, one has that $M^s(u_1,g,\delta^*)\cap M^{cu}(u_2,g,\delta^*)\neq\emptyset$, and $M^{cs}(u_1,g,\delta^*)\cap M^{u}(u_2,g,\delta^*)\neq\emptyset$.
\end{lemma}
\begin{proof}
We only prove that $M^s(u_1,g,\delta^*)\cap M^{cu}(u_2,g,\delta^*)\neq \emptyset$, for $\|u_1-u_2\|\ll 1$.
 By Lemma \ref{invari-mani} and \eqref{gener-inva}, for any $\omega=(u_0,g)\in \Omega$, there are $C^1$ functions $h^s(\cdot,\omega):V^s(\omega)\to V^{cu}(\omega)$, $h^{cu}(\cdot,\omega):V^{cu}(\omega)\to V^{s}(\omega)$ such that $h^s(u,\omega)$, $h^{cu}(u,\omega)=o(\|u\|)$, $\|\partial h^{s,cu}(u,\omega)\|\leq K_0$ for $u\in V^{s,cu}(\omega)$, and
\begin{eqnarray}
\begin{split}
 M^s(\omega,\delta^*)=&\{u_0+u^s+h^s(u^s,u_0,g):u^s\in V^s(\omega)
\cap\{u\in X: \|u\|<\delta^*\}\}\\
 M^{cu}(\omega,\delta^*)=&\{u_0+u^{cu}+h^{cu}(u^{cu},u_0,g):u^{cu}\in V^{cu}(\omega)
\cap\{u\in X: \|u\|<\delta^*\}\}.
\end{split}
\end{eqnarray}
Now, by Lemma \ref{cap-nonempty}, there exists $\delta>0$ such that for any $(u_1,g)$, $(u_2,g)\in \Omega$ with $\norm{u_1-u_2}<\delta$, one has $X=V^s(u_1,g)\oplus V^{cu}(u_2,g)$.
Consider the mapping $Q(u_1,u_2,g):X=V^s(u_1,g)\oplus V^{cu}(u_2,g)\to X: u=u_1^s+u_2^{cu}\mapsto u_2^{cu}-u_1^s $. It is easy to know that $Q(u_1,u_2,g)$ is an isomorphism.
Define $\tilde{Q}(u_1,u_2,g): X=V^s(u_1,g)\oplus V^{cu}(u_2,g)\to X: u=u_1^s+u_2^{cu}\mapsto Q(u_1,u_2,g)(u_1^s+u_2^{cu})+h^{cu}(u_2^{cu},u_2,g)-h^{s}(u_1^{s},u_1,g)$ with $\|u_1^s\|<\delta^*$, $\|u_2^{cu}\|<\delta^*$.
Noticing that $\tilde{Q}(u_1,u_2,g)0=0$ and $|\partial h^{s,cu}(u,\omega_0)| $ are small, let $D_u\tilde{Q}(u_1,u_2,g)v=v_2^{cu}+\partial h_u^{cu}(u_2^{cu},u_2,g)v_2^{cu}-v_1^{s}-\partial h_u^{s}(u_1^{s},u_1,g)v_1^{s}$, here $v_1^s\in V^s(u_1,g)$, $v_2^{cu}\in V^{cu}(u_2,g)$ and $v=v_1^s+v_2^{cu}$.
Then there is a $\delta_1>0$, such that for $\|u_1-u_2\|<\delta_1$, $D_u\tilde{Q}(u_1,u_2,g)$ is a surjective map from $X$ to $X$.
By the implicit function theorem, there is a $\delta_2>0$ such that, for any $\tilde{u}\in X$ with $\|\tilde{u}\|<\delta_2$ and $\|u_1-u_2\|<\delta_1$, one has a unique solution $u=u_1^s+u_2^{cu}$ with $(u_1^s,u_2^{cu})\in (V^s(u_1,g)\cap \{u\in X:\|u\|<\delta^*\},V^{cu}(u_2,g)\cap \{u\in X:\|u\|<\delta^*\})$ satisfying $\tilde{Q}(u_1,u_2,g)u=\tilde{u}$. Particularly, for $\tilde{u}=u_1-u_2$ with $\|u_1-u_2\|<\min\{\delta,\delta_1,\delta_2\}$. That is $u_1+u_1^s+h^s(u_1^s,u_1,g)=u_2+u_2^{cu}+h^{cu}(u_2^{cu},u_2,g)$.
\end{proof}

Hereafter, we will focus on a spatially inhomogeneous minimal set $M\subset X\times H(f)$ for \eqref{equation-lim1}.
Let $\Sigma M:=\{\sigma_au:a\in S^1,u\in M\}$. Since the group $S^1$ is compact and connected, $\Sigma M$ is a connected and compact invariant subset in $X\times H(f)$.
So the Sacker-Sell spectrum, as well as the stable, unstable, center, center-stable, and center-unstable subspaces at each $\omega\in \Sigma M$, are well defined.

\begin{lemma}\label{center-constant}
  Let $M$ be a spatially inhomogeneous minimal set of \eqref{skew-product2}. Assume that
  ${\rm dim}V^c(\omega)= 1$, or ${\rm dim}V^c(\omega)=2$ with ${\rm dim}V^u(\omega)$ being odd. Then $$z(u_1(\cdot)-u_2(\cdot))=N_u, \quad \text{for any }(u_1,g), (u_2,g)\in M \text{ with }u_1\neq u_2,$$ where $N_u$ are as in Corollary \ref{zero-center}.
\end{lemma}
\begin{proof}
Fix $(u_1,g)$, $(u_2,g)\in M$.  We first claim that there is a sequence $t_n\to\infty$  such that $g\cdot t_n\to g^+$ and
\begin{equation}\label{asym-rota}
\varphi(t_n,\cdot;u_1,g)\to u_1^+,\,\,\, \varphi(t_n,\cdot;u_2,g)\to u_2^+,
\end{equation}
with $u_1^+\in \Sigma u_2^+$.
In fact,  by taking a sequence $t_n\to \infty$, we may assume that $g\cdot t_n\to g^+$ and
$\varphi(t_n,\cdot;u_1,g)\to u_1^+,\,\,\, \varphi(t_n,\cdot;u_2,g)\to u_2^+.$
If $u_1^+\in\Sigma u_2^+$, then  the claim holds.
If $u_1^+\not\in \Sigma u_2^+$, then by Lemma \ref{sequence-limit} and the connectivity of $S^1$,
 there is an integer $\tilde N$ such that
$$
z(\varphi(t,\cdot;u_1^+,g^+)-\varphi(t,\cdot;\sigma_a u_2^+,g^+))=\tilde N,
\quad \text{for all }t\in \RR \text{ and }a\in S^1.
$$
By the compactness of $S^1$, one can find a $T_0>0$ such that
$$
z(\varphi(t,\cdot;u_1,g)-\varphi(t,\cdot;\sigma_a u_2,g))=\tilde N,
\quad \text{for all }t\ge T_0 \text{ and }a\in S^1.
$$
As a consequence,
$$
m_1(t):=\max _{x\in S^1} \varphi(t,x;u_1,g)\not = m_2(t):= \max_{x\in S^1}\varphi(t,x;u_2,g),\quad \text{for all }\,t\ge T_0.
$$
In fact, suppose on the contrary that $m_1(t)=m_2(t)$ for some $t\geq T_0$. Take $x_1,x_2\in S^1$ such that $m_1(t)=\varphi(t,x_1;u_1,g)$ and $m_2(t)=\varphi(t,x_2;u_2,g)$. Put $b:=x_2-x_1$. Thus, $\varphi(t,x_1;\sigma_bu_2,g)-\varphi(t,x_1;u_1,g)=0$, so $x_1$ is a multiple zero of $\varphi(t,\cdot;\sigma_bu_2,g)-\varphi(t,\cdot;u_1,g)$, a contradiction to Lemma \ref{sigma-function}(b).

Without loss of generality, we may assume that
$m_1(t)>m_2(t)$ for all $t\ge T_0$. For the above $g^+$, let $m^+=\max\{\max_{x\in S^1} v(x):(v,g^+)\in M\}$ and $u_2^{++}$ be such that
$(u_2^{++},g^+)\in M$ with $\max_{x\in S^1}u_2^{++}(x)=m^+$. Since $M$ is minimal, one can take another sequence $t_n^+\to\infty$ such that
$$
(\varphi(t_n^+,\cdot;u_2,g),g\cdot t_n^+)\to (u_2^{++}(\cdot), g^+).
$$
Without loss of generality, we assume that $\varphi(t_n^+,\cdot;u_1,g)\to u_1^{++}(\cdot).$
Then, by the definition of $m^+$, we must have
\begin{equation}\label{max-equal}
\max_{x\in S^1}u_1^{++}(x)=\max_{x\in S^1}u_2^{++}(x).
\end{equation}
Suppose that $u_1^{++}\not \in  \Sigma_a u_2^{++}$. Then, again by Lemma \ref{sequence-limit} and the connectivity of $S^1$,
there is $\tilde N^+$ such that
$$
z(\varphi(t,\cdot;u_1^{++},g^+)-\varphi(t,\cdot;\sigma_a u_2^{++},g^+))=\tilde N^+,\quad \forall\,\, t\in\RR,\,\, a\in S^1.
$$
This contradicts \eqref{max-equal}. Hence,
$u_1^{++}\in\Sigma u_2^{++}$
and the claim is proved.

Together with the claim and the connectivity of $(S^1\times S^1)\setminus \{(a,a): a\in S^1\}$, it follows from Lemma \ref{sequence-limit} and \eqref{asym-rota} that there is a constant $C$ such that
$$
z(\varphi(t,\cdot;\sigma_b u_1^{+},g^+)-\varphi(t,\cdot;\sigma_a u_2^{+},g^+))\equiv C,
$$ whenever $t\in \RR$ and $a,b\in S^1$ with $\sigma_bu_1^{+}\not =\sigma_a u_2^{+}$.  In particular,
\begin{equation}\label{rot-consta-u++}
z(\varphi(t,\cdot;u_1^{+},g^+)-\varphi(t,\cdot;\sigma_a u_2^{+},g^+))\equiv C,
\end{equation} for any $t\in \R$, and $a\in S^1$ with $u_1^{+}\not =\sigma_a u_2^{+}$. Recall that $u_1^+=\sigma_{\tilde{a}} u_2^+$ for some $\tilde{a}\in S^1$, the spatial inhomogeneity of $M$ enables us to find some $a_*\in S^1$ sufficiently close to $\tilde{a}$ such that $u_1^+\ne\sigma_{a_*} u_2^+$; and moreover, due to the translation-group action $\sigma$ on the semiflow, as well as the compactness and invariance of $M$, one has
$\norm{\varphi(t,\cdot;\sigma_{a_*} u_2^{+},g^+)-\varphi(t,\cdot;u_1^{+},g^+)}$ is sufficiently small uniformly for all $t\in \R.$
This then implies that $\sigma_{a_*} u_2^+\in M^{cs}(u_1^+,g^+,\delta^*).$ Therefore, by virtue of Corollary \ref{zero-center}, one obtains that $$z(u_1^+-\sigma_{a_*}u_2^+)\ge N_u.$$ Together with \eqref{rot-consta-u++}, we have $C\ge N_u$. So
\begin{equation}\label{rot-consta-u-n}
z(\varphi(t,\cdot;u_1^{+},g^+)-\varphi(t,\cdot;\sigma_a u_2^{+},g^+))\ge N_u,
\end{equation}
for all $t\in \RR$ and $a\in S^1$ with $u_1^{+}\not =\sigma_a u_2^{+}$.

We now will show that
\begin{equation}\label{Nu-positve-direc}
z(\varphi(t,\cdot;u_1,g)-\varphi(t,\cdot;u_2,g))\ge N_u, \quad \text{for } t\ge 0 \text{ sufficiently large.}
\end{equation}
In fact, we note \eqref{asym-rota}. If $u_1^+\ne u_2^+$, then \eqref{rot-consta-u-n} entails that $z(u_1^+(\cdot)-u_2^+(\cdot))\ge N_u$ and $u_1^+(\cdot)-u_2^+(\cdot)$ can only have simple zeros; and hence, we have already obtained that $z(\varphi(t,\cdot;u_1,g)-\varphi(t,\cdot;u_2,g))\ge N_u$ for $t\gg 1.$ If $u_1^+=u_2^+$, it then follows from Lemma \ref{sigma-function}(c) that there is an integer $N_0$ such that
$z(\varphi(t,\cdot;u_1,g)-\varphi(t,\cdot;u_2,g))=N_0$ for all $t$ sufficiently positive.
Fix an $n_0\gg 1$ in \eqref{asym-rota}, one can obtain a neighborhood $B(e)$ of the unit $e$ in the group $S^1$ such that $z(\varphi(t_{n_0},\cdot;u_1,g)-\varphi(t_{n_0},\cdot;\sigma_au_2,g))=N_0$ for any $a\in B(e)$. So, by Lemma \ref{sigma-function}(c) again, we have $N_0\ge z(\varphi(t,\cdot;u_1,g)-\varphi(t,\cdot;\sigma_au_2,g))$ for any $t\ge t_{n_0}$ and any $a\in B(e)$. Note that there is at least some $a_0\in B(e)$ such that $u_1^+\ne \sigma_{a_0}u_2^+$ (Otherwise, $\sigma_au_2^+=u_1^+=u_2^+$ for any $a\in B(e)$, which implies that $u_2^+$ is spatially homogenous, a contradiction).
Then, again by Lemma \ref{sequence-limit} and \eqref{asym-rota}, one obtains that $z(\varphi(t,\cdot;u_1^{+},g^+)-\varphi(t,\cdot;\sigma_{a_0} u_2^{+},g^+))\equiv \text{constant}\le N_0.$ Combining with \eqref{rot-consta-u-n}, we obtain that $N_0\ge N_u$. Thus, we have proved \eqref{Nu-positve-direc}.

Next, we will prove that
\begin{equation}\label{Nu-negative-direc}
z(\varphi(t,\cdot;u_1,g)-\varphi(t,\cdot;u_2,g))\le N_u,\quad \text{for } t \text{ sufficiently negative.}
\end{equation}
If $\norm{\varphi(t,\cdot;u_1,g)-\varphi(t,\cdot;u_2,g)}\to 0$ as $t\to -\infty$, then $\varphi(t,\cdot;u_2,g)\in M^{cu}(\varphi(t,\cdot;u_1,g),g\cdot t,\delta^*)$ for all $t$ sufficiently negative. By Corollary \ref{zero-center}(2)-(4), it follows that \eqref{Nu-negative-direc} holds already. Hereafter, we assume that
\begin{equation}\label{negative-not-asy}
\norm{\varphi(t,\cdot;u_1,g)-\varphi(t,\cdot;u_2,g)}\nrightarrow 0\,\, \text{ as }\,t\to -\infty.
\end{equation}
Similarly as in the claim above, one may obtain a sequence $s_n\to -\infty$  such that $g\cdot s_n\to g^-$ and
\begin{equation}\label{asym-rota-negative}
\varphi(s_n,\cdot;u_1,g)\to u_1^-,\,\,\, \varphi(s_n,\cdot;u_2,g)\to u_2^-,
\end{equation}
with $u_1^-\in\Sigma u_2^-$. Moreover, by repeating similar argument as above, we can utilize
the estimate of the zero-number $z$ on  $M^{cu}(\omega,\delta^*)$ in Corollary \ref{zero-center}(3)-(4) to obtain
\begin{equation}\label{rot-consta-u-n-nega}
z(\varphi(t,\cdot;u_1^{-},g^-)-\varphi(t,\cdot;\sigma_a u_2^{-},g^-))\le N_u,
\end{equation}
for all $t\in \RR$ and $a\in S^1$ with $u_1^{-}\not =\sigma_a u_2^{-}$.

So, if $u_1^-\ne u_2^-$, then \eqref{rot-consta-u-n-nega} entails that $z(u_1^-(\cdot)-u_2^-(\cdot))\le N_u$ and $u_1^-(\cdot)-u_2^-(\cdot)$ can only have simple zeros; and hence, by \eqref{asym-rota-negative} we have already obtained that $z(\varphi(t,\cdot;u_1,g)-\varphi(t,\cdot;u_2,g))\le N_u$ for $t$ sufficiently negative. For the case $u_1^-= u_2^-$, recall that $\norm{\varphi(t,\cdot;u_1,g)-\varphi(t,\cdot;u_2,g)}\nrightarrow 0$ as $t\to -\infty$ in \eqref{negative-not-asy}. Since $M$ is compact, there is a sequence $\tau_n\to -\infty$ such that $\Pi^{\tau_n}(u_i,g)\to (u^*_i,g^*)\in M$ ($i=1,2$) as $n\to\infty$ with $u^*_1\neq u^*_2$. By Lemma \ref{sequence-limit}, there is $N_0$ such that $z(\varphi(t,\cdot;u^*_1,g^*)-\varphi(t,\cdot;u^*_2,g^*))=N_0$ for all $t\in\mathbb{R}$, in particular $z(u^*_1-u^*_2)=N_0$. It follows from Lemma \ref{sigma-function} that $z(\varphi(t,\cdot;u_1,g)-\varphi(t,\cdot;u_2,g))=N_0$ for all $t$ sufficiently negative. Hence, one can repeat the similar argument immediately after \eqref{Nu-positve-direc} to obtain that $N_0\le N_u$. Thus, we have  proved \eqref{Nu-negative-direc}.

By virtue of \eqref{Nu-positve-direc} and \eqref{Nu-negative-direc}, we have obtained
$$
z(\varphi(t,\cdot;u_1,g)-\varphi(t,\cdot;u_2,g))=N_u,\quad\forall\,\, t\in\RR,
$$ which completes the proof.
\end{proof}

For all $u\in X$, we write $m(u)$ as the maximum of $u$ on $S^1$, and define an equivalence relation on $X$ by putting $u \sim v$ if and only if $u=\sigma_a v$ for some $a \in S^1$. The equivalence class is denoted by $[u]$. Then we have the following very useful Corollary.
\begin{corollary}\label{translate-equivalence-1}
 Let $M$ be a minimal set of \eqref{skew-product2}. Assume that $M$ is spatially inhomogeneous and
  ${\rm dim}V^c(\omega)= 1$, or ${\rm dim}V^c(\omega)=2$ with ${\rm dim}V^u(\omega)$ being odd. Then, for any $g\in H(f)$ and any two elements $(u_1,g),(u_2,g)$ in $M\cap p^{-1}(g)$, one has:
  \begin{itemize}
    \item [\rm{(i)}]
    $
    z(\varphi(t,\cdot;u_1,g)-\varphi(t,\cdot;\sigma_a u_2,g))=N_u
    $, for any $a\in S^1$  with $\sigma_a u_2\not= u_1$.
  \item[\rm{(ii)}]If $m(u_1)<m(u_2)$, then
$m(\varphi(t,\cdot;u_1,g))<m(\varphi(t,\cdot; u_2,g)),\  for\ all\ t\in \mathbb{R}$.
\item[\rm{(iii)}]
$m(u_1)=m(u_2)$ $\Leftrightarrow$ $([u_1],g)=([u_2],g)$.
  \end{itemize}
\end{corollary}
\begin{proof}
(i) can be obtained by repeating the same arguments as in Lemma \ref{center-constant}.

(ii) Since $m(u_1)<m(u_2)$, it follows that $u_1\neq \sigma_a u_2$ for any $a\in S^1$. By virtue of (i),
\begin{equation}\label{tran-constant}
 z(\varphi(t,\cdot;u_1,g)-\varphi(t,\cdot;\sigma_a u_2,g))=N_u,
\end{equation}
for any $a\in S^1$ and $t\in \mathbb{R}$. Suppose that there is a $t_0>0$ (resp. $t_0<0$), such that
$m(u(t_0,\cdot;u_1,g))\geq m(u(t_0,\cdot; u_2,g)).$
Let $K(t)=m(\varphi(t,\cdot;u_1,g))-m(\varphi(t,\cdot; u_2,g))$, $t\in \mathbb{R}$. Then $K(t)$ is continuous for $t\in \mathbb{R}$. Moreover, one has $K(0)<0$ and $K(t_0)\geq 0$. So one can find a $t_1\in (0,t_0]$ (resp. $t_1\in [t_0,0)$) such that $m(\varphi(t_1,\cdot;u_1,g))=m(\varphi(t_1,\cdot;u_2,g))$. Hence, there exist some $a_0 \in S^1$ and $x_0\in S^1$ such that $\varphi(t_1,x_0;u_1,g)=m(\varphi(t_1,\cdot;u_1,g))=\varphi(t_1,x_0+a_0;u_2,g)
=\varphi(t_1,x_0;\sigma_{a_0}u_2,g)$ and $\varphi_x(t_1,x_0;u_1,g)=0=\varphi_x(t_1,x_0;\sigma_{a_0}u_2,g)$. Then Lemma \ref{sigma-function}(b) implies that $z(\varphi(t,\cdot;u_1,g)-\varphi(t,\cdot;\sigma_a u_2,g))$ must drop at $t=t_1$, contradicting  \eqref{tran-constant}.

(iii) Clearly, $([u_1],g)=([u_2],g)$ implies $m(u_1)=m(u_2)$. If $m(u_1)=m(u_2)$, then there exists some $a_0 \in S^1$ such that $u_1-\sigma_{a_0}u_2$ possesses a multiple zero. So, we must have
$u_1=\sigma_{a_0}u_2$ (Otherwise, by (i), $u_1\ne \sigma_{a_0}u_2$ will imply that $z(\varphi(t,\cdot;u_1,g)-\varphi(t,\cdot;\sigma_{a_0} u_2,g))=N_u$ for all $t$; and hence, $u_1-\sigma_{a_0}u_2$ possesses only simple zeros, a contradiction). Thus, $([u_1],g)=([u_2],g)$.
We have completed the proof.
\end{proof}

Now we define by $\tilde{X}$ the quotient space $``X/\sim"$ of $X$. Then $\tilde{X}$ is a metric space with metric $\tilde{d}_{\widetilde{X}}$ defined as $\tilde{d}_{\tilde{X}}([u],[v]):=d_H(\Sigma u,\Sigma v)$ for any $[u],[v]\in \tilde{X}$. Here $d_H(U,V)$ is the Hausdorff metric of the compact subsets $U,V$ in $X$, defined as $d_{H}(U,V)=\sup\{\sup_{u \in U} \inf_{v \in V} d_X(u,v),$ $\, \sup_{v \in V} \inf_{u \in U} d_X(u,v)\}$ where the metric $d_X(u,v)=||u-v||_{X}$. Observe that $\tilde d_{\tilde X}([u],[v])=\inf_{b\in S^1}d_X(u,\sigma_b v)$. Indeed, for a fixed $a\in S^1$, $\inf_{b\in S^1}d_{X}(\sigma_a u,\sigma_b v)=\inf_{b\in S^1}d_{X}(u,\sigma_{b-a} v)=\inf_{b\in S^1}d_{X}(u,\sigma_b v)$. It is clear that $d_X$ satisfies the $S^1$-translation invariance that is $d_X(\sigma_a u,\sigma_a v)=d_X(u,v)$ for any $u,v\in X$ and any $a \in S^1$. If we denote the metric on $H(f)$ by $d_Y$, one has the product metric $d$ on $X\times H(f)$ by setting $d((u_1,g_1),(u_2,g_2))=d_X(u_1,u_2)+d_Y(g_1,g_2)$ for any two points $(u_1,g_1),(u_2,g_2)\in X\times H(f)$, then $\tilde{d}$ is the induced metric on $\tilde{X}\times H(f)$ defined as $\tilde{d}(([u],g_1),([v],g_2))=\tilde{d}_{\tilde X}([u],[v])+d_Y(g_1,g_2)$. For any subset $K\subset X\times H(f)$, we write $\tilde{K}=\{([u],g)\in \tilde{X}\times H(f)|(u,g)\in K\}$.

By virtue of Corollary \ref{translate-equivalence-1}, we consider the induced mapping $\tilde{\Pi}^t$, $t\geq 0$, as follows:
\begin{align}\label{induced-skepro-semiflow}
\begin{split}
\tilde{\Pi}^t:\tilde{X}\times H(f)&\rightarrow \tilde{X}\times H(f);\\
([u],g)&\rightarrow (\tilde{\varphi}(t,\cdot;[u],g),g\cdot t):=([\varphi(t,\cdot;u,g)],g\cdot t).
\end{split}
\end{align}

\begin{lemma}\label{induced-skewproduct}
\begin{itemize}
\item [\rm{(i)}] $\tilde{\Pi}^t$ is a skew-product semiflow on $\tilde{X}\times H(f)$;
\item [\rm{(ii)}] If $M$ is a minimal invariant subset in $X\times H(f)$ w.r.t. $\Pi^t$, then $\tilde{M}$ is a minimal invariant subset in $\tilde{X}\times H(f)$ w.r.t. $\tilde{\Pi}^t$.
\end{itemize}
\end{lemma}
\begin{proof}
(i). For any $([u],g)\in \tilde{X}\times H(f)$ and any $t,s\geq 0$,
\begin{equation*}
\begin{split}
\tilde{\Pi}^{t+s}([u],g)&=([\varphi(t+s,\cdot;u,g)],g\cdot(t+s))=([\varphi(s,\cdot;\varphi(t,\cdot;u,g),g\cdot t)],g\cdot(t+s))\\
&=(\tilde{\varphi}(s,\cdot;[\varphi(t,\cdot;u,g],g\cdot t),(g\cdot t)\cdot s)=(\tilde{\varphi}(s,\cdot;\tilde{\Pi}^{t}([u],g)),(g\cdot t)\cdot s)\\
&=\tilde{\Pi}^{s}(\tilde{\Pi}^t([u],g)).
\end{split}
\end{equation*} So, $\tilde{\Pi}^t$ admits the cocycle property.
Since $\mathbb{R}^+$ is a Baire space and $\tilde{X}\times H(f)$ is metrizable,  it suffices to verify (see e.g., \cite{chernoff1970continuity} or \cite[Theorem 1]{Stephan}) the continuity of $\tilde{\Pi}$  with respect to each variable. Here we only show the continuity in $([u],g)\in\tilde{X}\times H(f)$. The continuity with respect to $t$ can be proved similarly. For any fixed $t\ge 0$, since $\Pi^t$ is continuous on $X\times H(f)$, it follows that for any $(u_0,g_0)$ and any $\varepsilon >0$, there exists a $\delta >0$ such that $ d(\Pi^t(u_0,g_0),\Pi^t(u,g))<\varepsilon$ for any $(u,g)\in X\times H(f)$ with  $d((u_0,g_0),(u,g))<\delta$. Fix $([u_0],g_0)\in \tilde{X}\times H(f)$, then for any $([u],g)\in \tilde{X}\times H(f)$ with $\tilde{d}(([u_0],g_0),([u],g))<\delta$, one has $d_H(\Sigma u_0,\Sigma u)+d_Y(g_0,g)<\delta$. By the definition of Hausdorff metric on $\tilde X$, one has $\inf_{b\in S^1}d_X(u_0,\sigma_bu)+d_Y(g_0,g)<\delta$, so we can find some $a\in S^1$ such that $d((u_0,g_0),(\sigma_a u,g))<\delta$, and hence, $d(\Pi^t(u_0,g_0),\Pi^t(\sigma_a u,g))<\varepsilon$. Thus,
\begin{equation*}
\begin{split}
  \tilde {d}(\tilde{\Pi}^t([u_0],g_0),\tilde{\Pi}^t([u],g))&=\inf_{b\in S^1}d_X(\varphi(t,\cdot;u_0,g),\varphi(t,\cdot;\sigma_bu,g))+d_Y(g_0\cdot t,g\cdot t)\\
  &\leq d_X(\varphi(t,\cdot;u_0,g),\varphi(t,\cdot;\sigma_au,g))+d_Y(g_0\cdot t,g\cdot t)\\
  &=d(\Pi^t(u_0,g_0),\Pi^t(\sigma_au,g))<\varepsilon,
\end{split}
\end{equation*}
 which completes the proof of (i).

(ii). Given any two points $([u],g),([u_0],g)\in \tilde{M}$, by the minimality of $M$ there exists a sequence $t_n\rightarrow +\infty$ (as $n\rightarrow \infty$) such that
\begin{equation}\label{limit-2}
\Pi^{t_n}(u_0,g_0)\rightarrow (u,g)
\end{equation}
as $n\rightarrow \infty$.  In order to show that $\tilde{\Pi}^{t_n}([u_0],g_0)\rightarrow([u],g)$ as $n\rightarrow \infty$, it suffices to prove that
$\tilde{d}_{\tilde{X}}([\varphi(t_n,\cdot;u_0,g_0)],[u])\rightarrow 0$ as $n\rightarrow \infty$.
To this end, it follows from \eqref{limit-2} that $\varphi(t_n,\cdot;u_0,g_0)\rightarrow u$ in $X$. So, by the translation-invariance  of $G$-group action, one has $d_X(\sigma_a \varphi(t_n,\cdot;u_0,g_0),\sigma_a u)=d_X(\varphi(t_n,\cdot;u_0,g_0),u)\rightarrow 0$ as $n\rightarrow \infty$ uniformly for $a\in S^1$. Therefore, one has
\begin{equation}\label{metric-1}
  \inf_{b\in S^1}d_X(\sigma_a\varphi(t_n,\cdot;u_0,g_0),\Sigma_b u)\leq d_X(\sigma_a\varphi(t_n,\cdot;u_0,g_0),\sigma_a u)\rightarrow 0
\end{equation}
as $n\rightarrow \infty$ uniformly for $a\in S^1$. On the other hand, we know that
\begin{align}\label{metric-2}
\inf_{b\in S^1}d_X(\sigma_b \varphi(t_n,\cdot;u_0,g_0),\sigma_a u)
 \leq d_X(\sigma_a\varphi(t_n,\cdot;u_0,g_0),\sigma_a u)\rightarrow 0.
 \end{align}
 Thus, from \eqref{metric-1} and \eqref{metric-2}, Hausdorff metric $d_H(\Sigma(\varphi(t_n,\cdot;u_0,g)),\Sigma u)\rightarrow 0$ as $n\rightarrow \infty$. Recall that
\[
\tilde{d}_{\tilde{X}}([\varphi(t_n,\cdot;u_0,g_0)],[u])=d_H(\Sigma(\varphi(t_n,\cdot;u_0,g_0)),\Sigma u).\] It entails that $\tilde{d}_{\tilde{X}}([\varphi(t_n,\cdot;u_0,g_0)],[u])\rightarrow 0$ as $n\rightarrow \infty$. We have proved that $\tilde{M}$ is a minimal invariant set in $\tilde{X}\times H(f)$.
\end{proof}
Let $\tilde{p}:\tilde{X}\times H(f)\rightarrow H(f)$ be the natural projection and $M$ be a spatially inhomogeneous minimal set of \eqref{skew-product2} with $\dim V^c(\omega)=1$, or $\dim V^c(\omega)=2$ and $\dim V^u(\omega)$ being odd.
 By virtue of Corollary \ref{translate-equivalence-1}(iii), one can define an ordering on each fiber $\tilde{M}\cap \tilde{p}^{-1}(g)$, with the base point $g\in H(f)$ as follows:
$$([u],g)\leq_g ([v],g)\,\, \text{ if }\,\, m(u)\leq m(v).$$ We also write the strict relation $([u],g)<_g ([v],g)$ if $m(u)<m(v)$. Without any confusion, we hereafter will drop the subscript $``g"$.\par

\begin{lemma}\label{total-ordering1}
$``\leq"$ is a total ordering on each $\tilde{M}\cap \tilde{p}^{-1}(g)$, ($g\in H(f)$) and $\tilde{\Pi}^t$ is strictly order preserving on $\tilde{M}$ in the sense that, for any $g\in H(f)$, $([u],g)<([v],g)$ implies that $\tilde{\Pi}^t([u],g)<\tilde{\Pi}^t([v],g)$ for all $t\geq 0$.
\end{lemma}
\begin{proof}
This is a direct result of Corollary \ref{translate-equivalence-1} (ii)-(iii).
\end{proof}

Let $E\subset \tilde{X}\times H(f)$ be a compact invariant subset of $\tilde{\Pi}^t$ which admits a flow extension.
For each $g\in H(f)$, we define {\it a fiberwise strong ordering $``\ll"$ on each fiber} $E\cap\tilde{p}^{-1}(g)$ as follows: $([u_1],g)\ll ([u_2],g)$ if there exist neighborhoods $\mathcal{N}_1,\mathcal{N}_2 \subset E\cap\tilde{p}^{-1}(g)$ of $([u_1],g),([u_2],g)$, respectively, such that $([u^*_1],g)< ([u^*_2],g)$ for all $([u_i^{*}],g)\in \mathcal{N}_i\ (i=1,2).$

Moreover, for each $g\in H(f)$, we say $([u_1],g),([u_2],g)$ forms {\it a strongly order-preserving pair} if $([u_1],g),([u_2],g)$ is strongly ordered on the fiber, written $([u_1],g)\ll([u_2],g)$, and there are neighborhoods $U_i$ of $([u_i],g)$ $(i=1,2)$ in $E$ respectively, such that whenever $([u^*_1],g),([u^*_2],g)\in E\cap \tilde{p}^{-1}(g)$, with $\tilde{\Pi}^{T}([u^*_1],g)\in U_i\ (i=1,2)$ for some $T<0$, then $([u^*_1],g)\ll([u^*_2],g)$.

The following lemma is essentially from \cite{Shen1998} and will play an important role in our proof.
\begin{lemma}\label{non-ordered}
Let $E$ be a minimal set of $\tilde{\Pi}^t$ which admits a flow extension and $Y'$ be as in Lemma \ref{epimorphism-thm}. Then for any $g\in Y',\, E\cap\tilde{p}^{-1}(g)$ admits no strongly order preserving pair.
\end{lemma}
\begin{proof}
One can repeat the arguments in \cite[Theorem II.3.1]{Shen1998} to obtain this lemma.
\end{proof}
Now we are ready to prove Theorem \ref{sysm-embed}.\par

\begin{proof}[Proof of Theorem \ref{sysm-embed}]
(1)
Let $Y_0=Y'$, where $Y'$ is defined as in Lemma \ref{epimorphism-thm}. By virtue of Lemma \ref{induced-skewproduct}, we consider the induced invariant minimal set $\tilde{M}$ for the skew-product semiflow $\tilde{\Pi}^t$ on $\tilde{X}\times H(f)$.

 In order to prove the statement of Theorem \ref{sysm-embed}(1), we first show that for any $g\in Y_0$, there exists $u_g\in X$ such that $M\cap p^{-1}(g)\subset (\Sigma u_g,g)$. To this end, it suffices to prove that $\tilde{M}\cap \tilde{p}^{-1}(g)$ is a singleton. Suppose that there are two distinct points $([u_1],g),([u_2],g)$ on $\tilde{M}\cap \tilde{p}^{-1}(g)$ for some $g\in Y_0$. It then follows from Lemma \ref{total-ordering1} that $([u_1],g)$ and $([u_2],g)$ is strictly order related, say $([u_1],g)<([u_2],g)$. Then, by the order defined before Lemma \ref{total-ordering1}, we have $m(u_1)<m(u_2)$.

We  prove that $([u_1],g)$ and $([u_2],g)$ are strongly ordered.
In fact, let $\varepsilon_0=m(u_2)-m(u_1)$. Because $m(u)$ is continuous with respect to $u \in X$, there is a $\delta >0$ such that
\begin{equation}\label{neighborhood}
|m(\tilde{u})-m(u_1)|<\frac{\varepsilon_0}{4}\,\, (\text{resp.}\, |m(\tilde{u})-m(u_2)|<\frac{\varepsilon_0}{4})
\end{equation}
for any $(\tilde{u},g)\in X\times H(f)$ with $d_X(\tilde{u},u_1)<\delta$ (resp. $d_X(\tilde{u},u_2)<\delta$). For such $\delta>0$ and $i=1,2$, we define the neighborhoods $\mathcal{N}_i=\{([u^*_i],g)\in \tilde{M}\,|\,\tilde{d}_{\tilde{X}}([u^*_i],[u_i])<\delta\}$ of $([u_i],g)$. Then if $([u^*_i],g)\in \mathcal{N}_i$, one can find some $a_i\in S^1$ such that $d_X(\sigma_{a_i}u^*_i,u_i)<\delta$. As a consequence, by \eqref{neighborhood}, we have $|m(\sigma_{a_i}u^*_i)-m(u_i)|<\frac{\varepsilon_0}{4}$, which implies that
\begin{equation}\label{diffe-ep}
m(u^*_1)=m(\sigma_{a_1}u^*_1)<m(\sigma_{a_2}u^*_2)-\frac{\varepsilon_0}{2}=m(u^*_2)-\frac{\varepsilon_0}{2}.
\end{equation}
By the definition of order defined in $\tilde{M}\cap \tilde{p}^{-1}(g)$, $g\in H(f)$ it entails that $([u^*_1],g)< ([u^*_2],g)$. Thus we have shown that $([u_1],g)\ll([u_2],g)$.

We now claim that, $([u_1],g)$ and $([u_2],g)$ forms a strongly order preserving pair. Indeed, for each $i=1,2$, we define a neighborhood $U_i=\{([v_i],h)|\tilde{d}(([u_i],g),([v_i],h))<\delta\}$ of $([u_i],g)$ in $\tilde{M}$. Given any $([u^*_i],g)\in \tilde{M}\cap\tilde{p}^{-1}(g)$ with $\tilde{\Pi}^{T}([u^*_i],g)\in U_i$ for some $T<0$, we write $\tilde{\Pi}^{T}([u^*_i],g)=(\tilde{\varphi}(T,[u^*_i],g),g\cdot T)\in U_i$. So $\tilde{d}_{\tilde{X}}(\tilde{\varphi}(T,[u^*_i],g),[u_i])<\delta$ for $i=1,2$. Recall that $m(u_1)<m(u_2)$, one can repeat the arguments for proving \eqref{diffe-ep} in the previous paragraph to obtain that $m(\varphi(T,u^*_1,g))<m(\varphi(T,u^*_2,g))$. It then follows from the definition of the order defined on the fiber $\tilde{M}\cap \tilde{p}^{-1}(g\cdot T)$ that $\tilde{\varphi}(T,[u^*_1],g)<\tilde{\varphi}(T,[u^*_2],g)$. Combining with the order-preserving property of $\tilde{\Pi}^t(t\geq 0)$ in Lemma \ref{total-ordering1}, we obtain that
\[
([u^*_1],g)=\tilde{\Pi}^{-T}(\tilde{\varphi}(T,[u^*_1],g),g\cdot T)<\tilde{\Pi}^{-T}(\tilde{\varphi}(T,[u^*_2],g),g\cdot T) =([u^*_2],g).
\]
In fact, by repeating the arguments in previous paragraph, we can get further $([u^*_1],g)\ll ([u^*_2],g)$. Thus, we have proved the claim.
  On the other hand, Lemma \ref{non-ordered} implies that there exists no such strongly order preserving pair on $\tilde{M}\cap \tilde{p}^{-1}(g)$, a contradiction.  Thus, we have proved that for any $g\in Y_0$, $\tilde{M}\cap \tilde{p}^{-1}(g)$ is a singleton (In other words, $\tilde{M}$ is an almost $1$-cover of $H(f)$), which also implies that $M\cap p^{-1}(g)\subset (\Sigma u_g,g)$ for some $u_g\in X$.

 Next, we will show that $M$ can be residually embedded into an almost automorphically  forced flow on $S^1$. Based on the above discussion, we have known that $\tilde{M}$ is an almost $1$-cover of $H(f)$ (with the residual subset $Y_0\subset H(f)$). Now,  define the mapping
\begin{equation}\label{E:natu-proj-h}
h:\tilde{M}\to \mathbb{R}\times H(f); ([u],g)\mapsto (m(u),g).
\end{equation}
Let $\hat{M}=h(\tilde{M})$. Clearly, $h$ is well-defined and continuous onto $\hat{M}$. Moreover, $h$ is injective due to Corollary \ref{translate-equivalence-1}(iii). Recall that $\tilde{M}$ and $\hat{M}$ are both compact, $h$ is also a closed mapping. Hence, $h$ is a homeomorphism from $\tilde{M}$ onto $\hat{M}$. On such $\hat{M}\subset \mathbb{R}\times H(f)$, one can naturally define the skew-product flow
\begin{equation}\label{E:induced-flow-hat-M}
\hat{\Pi}^t:\hat{M}\to \hat{M};(m(u),g)\mapsto (m(\varphi(t,\cdot,u,g)),g\cdot t),
\end{equation}
which is induced by $\Pi^t$ restricted to $M$. So, a straightforward check yields that
$$h\circ\tilde{\Pi}^t([u],g)=\hat{\Pi}^t\circ h([u],g)\,\,\,
\text{ for any }([u],g)\in \tilde{M}.$$ This entails that $h$ is a topologically-conjugate homeomorphism between $\tilde{M}\to \hat{M}\subset \mathbb{R}\times H(f)$. As a consequence, $\hat{M}$ is also an almost $1$-cover of $H(f)$ (with the residual subset $Y_0\subset H(f)$).

For each $g\in Y_0$, we choose some element, still denoted by $u_g(\cdot)$, from the $S^1$-group orbit $\Sigma u_g$ such that
\begin{equation}\label{E:u-g--base-trn}
\quad u_g(0)=m(u_g), \quad\hat M\cap p^{-1}(g)=(m(u_g),g)\,\text{ and }\,\tilde{M}\cap \tilde{p}^{-1}(g)=([u_g],g).
\end{equation}
Together with  \eqref{E:induced-flow-hat-M}, \eqref{E:u-g--base-trn} implies that $$u_{g\cdot t}(0)=m(u_{g\cdot t})=m(\varphi(t,\cdot,u_g,g)),\,\,\,\text{ for any }g\in H_0(f)\,\,\text{ and }t\in \mathbb{R}.$$
As a consequence, given any $g\in Y_0$, the function
$t\mapsto u_{g\cdot t}(0)$ is clearly continuous and is almost automorphic in $t$ (due to the fact that
$\hat{M}$ is an almost $1$-cover, see Remark \ref{a-p-to-minial}); and moreover, $u^{ g}(t,x):=u_{g\cdot t}(x)$ is almost automorphic in $t$ uniformly in $x$.

Now, for any $g\in Y_0$, we may define a non-negative function $t\mapsto c^g(t)\ge 0$ such that
\begin{equation}\label{E:rotation-spiral}
 \varphi(t,x;u_g,g)= u_{g\cdot t}(x+ c^g(t)),\quad \text{or equivalently,}\quad \varphi(t,x-c^g(t);u_g,g)=u_{g\cdot t}(x).
\end{equation}

In the following, we will show that $c^g(t)$ is indeed satisfied for \eqref{E:transla-cg-func} in the statement of Theorem \ref{sysm-embed}.
Before proceeding to this, we first {\it clarify two facts} due to the minimality and spatial-inhomogeneity of $M$:

(a) All the elements in $M$ share the common smallest positive spatial-period $L\in (0,2\pi]$;

(b) For the $u_g(\cdot)$ defined in \eqref{E:u-g--base-trn}, one has \begin{equation}\label{E:u'-u''-nontrivial-1}
u^{'}_g(0)=0\,\text{ and }u^{''}_g(0)\not =0\quad \text{ for any }g\in Y_0.
\end{equation}

To prove (a), we fix some $(u,g)\in M$ and let $L\in (0,2\pi]$ be the smallest positive spatial-period of $u(\cdot)$. Then, the minimality of $M$ implies that $L$ is a positive spatial-period of any element in $M$. We now assert that $L$ must be the smallest spatial-period of any other element of $M$. For otherwise, there exists some $(u_1,g_1)\in M$ and a positive number $L'<L$ such that $L'$ is the smallest spatial-period of $u_1(\cdot)$. So, the minimality of $M$ again implies that there is a sequence $t_n\to\infty$ such that $\Pi^{t_n}(u_1,g_1)\to (u,g)$; and hence, $\Pi^{t_n}(\sigma_{L'}u_1,g_1)\to (\sigma_{L'}u,g)$. Note that $\sigma_{L'}u_1=u_1$, we have $\sigma_{L'}u=u$, which contradicts the smallness of spatial-period $L$ for $u(\cdot)$.

 As for (b), since $M$ is spatially-inhomogeneous,  it follows that $\varphi_x(t,\cdot;u_g,g)\in V^c(\Pi^t(u_g,g))$ for each $t$. Recall that $M$ satisfies
  either $\dim V^c(M)=1$ with $\dim V^u(M)>0$,  or $\dim V^c(M)=2$ with $\dim V^u(M)$ being odd.
  Then, by virtue of Lemma \ref{zero-inva}, it follows that $\varphi_x(t,\cdot,u_g,g)$ only has simple zeros for any $t\in \mathbb{R}$. In particular, by letting $t=0$, one has $u_g^{'}(\cdot)$ only has simple zeros. Together with $u_g^\prime(0)=0$ (because $u_g(0)=m(u_g)$) and Corollary \ref{translate-equivalence-1}(i), one then obtain that
$u^{''}_g(0)\not =0$ for all $g\in Y_0$.

Now, we are ready to prove that $c^g(t)$ is satisfied for \eqref{E:transla-cg-func} in Theorem \ref{sysm-embed}.
By virtue of Fact (a), we let $\mathcal{S}^1:=\mathbb{R}/L\mathbb{Z}$. Then for each $t\in \mathbb{R}$, one can further choose $c^g(t)\in \mathcal{S}^1$ in \eqref{E:rotation-spiral} so that $c^g(t)$ is continuous in $t$. Indeed, suppose that there is a sequence $t_n\to t_0$ such that $\abs{c^g(t_0)-c^g(t_n)}\ge \epsilon_0>0$ in $\mathcal{S}^1$. For the sake of simplicity, we assume $c^g(t_n)\to c^*$ with $c^*\in \mathcal{S}^1$. So, $c^g(t_0)\ne c^*$ in $\mathcal{S}^1$. On the other hand, by \eqref{E:rotation-spiral}, one has
$u_{g\cdot t_0}(x+ c^g(t_0))=\varphi(t_0,x;u_g,g)=\lim_{n\to\infty}\varphi(t_n,x;u_g,g)=\lim_{n\to\infty}u_{g\cdot t_n}(x+ c^g(t_n))=u_{g\cdot t_0}(x+ c^*)$. This contradicts $c^g(t_0)\ne c^*$ with $c^g(t_0),c^*\in \mathcal{S}^1$, because $L$ is the smallest spatial-period.
So, the function $t\mapsto c^g(t)\in \mathcal{S}^1$ is continuous.

By \eqref{E:rotation-spiral} and the property of $u_{g\cdot t}(x)$ in \eqref{E:u'-u''-nontrivial-1} from Fact (b), we observe that
$$
\varphi_x(t,-c^g(t);u_g,g)=u^{'}_{g\cdot t}(0)=0\,\,
\text{ and }\,\,\varphi_{xx}(t,-c^g(t);u_g,g)= u^{''}_{g\cdot t}(0)\not =0.
$$
Then by the continuity of $c^g(t)$ in $t$ and Implicit Function Theorem, we have $c^g(t)$ is differentiable in $t$; and moreover, we have
\begin{eqnarray}\label{circle-flow-eq41}
\dot {c}^g(t)=G(t,c^g(t)),
 \end{eqnarray}
 where
\begin{equation*}
\begin{split}
G(t,z)=&g_p(t,\varphi(t,-z;u_g,g),\varphi_x(t,-z;u_g,g))\\
&+\dfrac{\varphi_{xxx}(t,-z;u_g,g)+g_u(t,\varphi(t,-z;u_g,g),\varphi_x(t,-z;u_g,g))\varphi_x(t,-z;u_g,g)}{\varphi_{xx}(t,-z;u_g,g)},\,\,\,\text{ for }z\in \mathcal{S}^1.
\end{split}
\end{equation*}
It is easy to see that $G(t,z+L)=G(t,z)$ and the function $G(t,c^g(t))=g_p(t,u_{g\cdot t}(0),0)+\frac{u^{'''}_{g\cdot t}(0)}{u^{''}_{g\cdot t}(0))}$, and hence $\dot{c}^g(t)$, is time almost-automorphic in $t$. Thus, we have obtained \eqref{circle-flow-eq41}, which naturally induces
an almost-automorphically forced skew-product flow on $\mathcal{S}^1\times H(f)$.

\vskip 2mm

  (2)  Note that Corollary \ref{translate-equivalence-1} also holds for
the case ${\rm dim}V^c(\omega)= 1$.
Then one can repeat the same argument in (1) to obtain a residual invariant subset $Y_0\subset H(f)$ such that $\tilde{M}\cap \tilde{p}^{-1}(g)$ is a singleton for any $g\in Y_0$.
In order to show $Y_0=H(f)$, we first note that, for any $(u_0,g)\in \Sigma M$, if $a\in S^1$ is close to $e$, then
$$\norm{\varphi(t,\cdot;\sigma_{a} u_0,g)-\varphi(t,\cdot;u_0,g)}\text{ is sufficiently small, for all }\,t\in \R.$$ It then follows that
that $\sigma_{a} u_0\in M^{c}((u_0,g),\delta^*)$ whenever $a\in S^1$ is close to $e$. Together with ${\rm dim}V^c(\omega)= 1$, one has
\begin{equation}\label{cmfd-contain-trans}
M^c((u_0,g),\delta^*)\subset \Sigma u_0,\,\, \text{ for any }(u_0,g)\in  \Sigma M.
\end{equation}
Moreover, due to Corollary \ref{zero-center}(3), one of the following two cases must occur:

(A): $z(u(\cdot)-u_0(\cdot))>N_u$ ($\dim V^u(\omega)$ is even), for $u\in M^{s}((u_0,g),\delta^*)\setminus\{u_0\}$; or otherwise

(B): $z(u(\cdot)-u_0(\cdot))<N_u$ ($\dim V^u(\omega)$ is odd), for $u\in M^{u}((u_0,g),\delta^*)\setminus\{u_0\}$.
\vskip 2mm

\noindent In the following, we will prove $Y_0=H(f)$ for case (A), and the proof for case (B) is analogous.

Suppose that there exist $g\in H(f)\setminus Y_0$ and $(u_1,g),(u_2,g)\in M$ such that $[u_1]\neq [u_2]$.  Then it follows from Corollary \ref{translate-equivalence-1}(i) that $z(\varphi(t,\cdot;u_1,g)-\varphi(t,\cdot;\sigma_a u_2,g))=N_u$ for all $t\in \mathbb{R}$ and $a\in S^1$.  Moreover, by the compactness of $S^1$, there exists $\delta>0$ (independent of $a\in S^1$) such that \begin{equation}\label{case-A-const-1}
z(u_1-\sigma_a u_2+v)=N_u\,\, \text{ for any }a\in S^1 \text{ and }\|v\|<\delta.
\end{equation}
Fix some $(u^+,g^+)\in M$ with $g^+\in Y_0$, then there exist a sequence $t_n\to\infty$ and some $b\in S^1$ such that $g\cdot t_n\to g^+$,
$$
\varphi(t_n,\cdot;u_1,g)\to u^+,\,\,\, \varphi(t_n,\cdot;u_2,g)\to \sigma_{b}u^+.
$$ Let $u_3=\sigma_{-b}u_2\in  \Sigma M$. Then $\varphi(t_n,\cdot;u_1,g)\to u^+,\,\,\, \varphi(t_n,\cdot;u_3,g)\to u^+.$
By virtue of Lemma \ref{cap-nonempty2}, one can find some $v_n^*\in M^u(\varphi(t_n,\cdot;u_1,g),\delta^*)\cap M^{cs}(\varphi(t_n,\cdot;u_3,g),\delta^*)$ for $t_n$ sufficiently large.

We now claim that $v_n^*\notin M^c(\varphi(t_n,\cdot;u_3,g),\delta^*)$. For otherwise, \eqref{cmfd-contain-trans} implies that $v_n^*=\sigma_{a^*} \varphi(t_n,\cdot;u_3,g)$ for some $a^*\in S^1$. Note also that $v_n^*\in M^u(\varphi(t_n,\cdot;u_1,g),\delta^*)$, then
\begin{equation}\label{backward-contracting-1}
\|\sigma_{a^*}u_3-u_1\|=\norm{\varphi(-t_n,\cdot;v_n^*,g\cdot t_n)-u_1}\leq Ce^{-\frac{\alpha}{2} t_n}\|v_n^*-\varphi(t_n,\cdot;u_1,g)\|.
\end{equation}
Let $\epsilon_0=|m(u_1)-m(u_2)|>0$ with $m(u_i)=\max_{x\in S^1} u_i(x)$ for $i=1,2$ (note that $\epsilon_0>0$ is due to $u_2\notin\Sigma u_1$ and Corollary \ref{translate-equivalence-1}(iii)). Since $X\hookrightarrow C^1(S^1)$, it is not difficult to see that $\|u_1-\sigma_{a}u_2\|\geq C_0|m(u_1)-m(\sigma_au_2)|=C_0|m(u_1)-m(u_2)|=C_0\epsilon_0$ for some constant $C_0>0$ and all $a\in S^1$. On the other hand, by letting $t_n$ large enough in \eqref{backward-contracting-1}, one has $\|\sigma_{a^*}u_3-u_1\|<\mathrm{min}\{\delta^*,C_0\epsilon_0\}$.
This contradicts $ \|\sigma_{a^*}u_3-u_1\|\geq C_0\epsilon_0$. Thus, we have proved the claim.

Recall that $v_n^*\in M^{cs}(\varphi(t_n,\cdot;u_3,g),\delta^*)$. Then it follows from Remark \ref{stable-leaf} and \eqref{cmfd-contain-trans} that there is some $a_0\in S^1$ such that $v_n^*\in M^{s}(\sigma_{a_0}\varphi(t_n,\cdot;u_3,g),\delta^*)$ with $\sigma_{a_0}\varphi(t_n,\cdot;u_3,g)\in M^c(\varphi(t_n,\cdot;u_3,g),\delta^*)$. Since case (A) holds, we obtain $z(v_n^*-\sigma_{a_0}\varphi(t_n,\cdot;u_3,g))>N_u$, and hence,
\begin{equation}\label{u-3-Nu-large}
z(\varphi(-t_n,\cdot;v_n^*,g\cdot t_n)-\sigma_{a_0}u_3)>N_u.
\end{equation}
On the other hand, one can deduce from \eqref{backward-contracting-1} that $\|\varphi(-t_n,\cdot;v_n^*,g\cdot t_n)-u_1\|<\delta$, for $t_n$ sufficiently large. Here $\delta>0$ is in \eqref{case-A-const-1}. Then \eqref{case-A-const-1} implies that $z(\varphi(-t_n,\cdot;v_n^*,g\cdot t_n)-\sigma_{a_0}u_3)=z(\varphi(-t_n,\cdot;v_n^*,g\cdot t_n)-u_1+u_1-\sigma_{a_0}u_3)=z(u_1-\sigma_{a_0}u_3)=z(u_1-\sigma_{a_0-b}u_2)=N_u$, contradicting \eqref{u-3-Nu-large}. Consequently, we have proved $Y_0=H(f)$.

 Therefore, one has $M\cap p^{-1}(g)=(\Sigma u_g,g)$ for any $g\in H(f)$, where $u_g$ is defined in \eqref{E:u-g--base-trn}. So, $u_{g\cdot t}(x)$ is almost periodic in $t$ uniformly in $x\in \mathcal{S}^1$; and moreover, $\dot{c}^g(t)$ in \eqref{E:transla-cg-func} is almost-periodic in $t$. Thus, we have completed the proof of Theorem \ref{sysm-embed}.
\end{proof}

\vskip 2mm

\begin{proof}[Proof of Theorem \ref{ergodic-thm}]
(1)  Suppose that ${\rm dim}V^c(\omega)=n_c>2$. Then there exists an integer $k\in \mathbb{N}$ such that the center space $V^c(\omega)$ admits an invariant splitting as $V^c(\omega)=\hat{G}(\omega)\bigoplus W_k(\omega)\bigoplus \hat{F}(\omega)$ (with  $\hat G(\omega)\subset \bigoplus_{i=0}^{k-1}W_i(\omega)$, $\hat F(\omega)\subset \bigoplus_{i=k+1}^{k+n_c}W_i(\omega)$; and at least one of $\hat G(\omega)$ and $\hat F(\omega)$ is nontrivial), which are exponentially separated in the sense of Lemma \ref{floquet-bundle}(v). Moreover, noticing that $V^c(\omega)$ is finite-dimensional, it then follows from the similar statement and argument in \cite[Corollary 4.7]{Chow1995} that such invariant splitting is also exponentially separated in $V^c(\omega)$ with respect to the $X$-norm $\norm{\cdot}$.

Since $M$ is uniquely ergodic, $M$ possesses a pure point Sacker-Sell spectrum. So, for any $v\in V^c(\omega)\setminus\{0\}$, one has
$$\lim_{t\to \infty}\dfrac{\ln \norm{\Psi(t,\omega)v}}{t}= 0.$$ In particular,
$\lim_{t\to \infty}\frac{\ln \norm{\Psi(t,\omega)v_1}}{t}= 0=\lim_{t\to \infty}\frac{\ln \norm{\Psi(t,\omega)v_2}}{t},$ for any $v_1\in W_k(\omega)\setminus\{0\}$, $v_2\in \hat{F}(\omega)\setminus\{0\}$ (or $v_2\in \hat{G}(\omega)\setminus\{0\}$). This contradicts the exponential separated property between $W_k(\omega)$ and $\hat{F}(\omega)$ (or $\hat{G}(\omega)$). Therefore, we have obtained ${\rm dim}V^c(\omega)\le 2.$

We then have either (i) ${\rm dim} V^c(\omega)=0$ or ${\rm dim} V^c(\omega)=1$ or (ii)
${\rm dim} V^c(\omega)=2$.  Moreover, by Lemma \ref{floquet-bundle} and similar arguments as in the previous paragraph, we must have ${\rm dim}V^u(\omega)$ is odd provided that ${\rm dim} V^c(\omega)=2$.

(2) 
 Since $M$ is spatially homogeneous, the variational equation associated with any solution in $M$  turns out to be $v_t=v_{xx}+a(t)v_x+b(t)v$ with periodic boundary condition. Using the transform $w=v(t,x+c(t))$ (with $\dot{c}(t)=-a(t)$),  we get $w_t=w_{xx}+b(t)w$; and moreover, by using another transformation $\hat{w}=we^{-\int_0^tb(s)ds}$, one obtains that $\hat{w}_t=\hat{w}_{xx}.$ So, it follows that the transform $\hat{w}=v(t,x+c(t))e^{-\int_0^tb(s)ds}$ (and hence, $v(t,x)=e^{\int_0^tb(s)ds}\hat{w}(t,x-c(t))$) results in a  simple equation $\hat{w}_t=\hat{w}_{xx}.$ Note that $\hat{w}_t=\hat{w}_{xx}$ possesses the simplest ``sin-cos"-mode eigenfunctions as $w_k(t,x)=e^{-k^2t}\sin kx,e^{-k^2t}\cos kx$ associated with the same eigenvalue $\lambda_k=-k^2$, $k=0,1,\cdots.$ Then it yields that $v_k(t,x)=e^{-k^2t+\int_0^tb(s)ds}\sin k(x-c(t)),e^{-k^2t+\int_0^tb(s)ds}\cos k(x-c(t))$. This immediately implies that, if  $V^u(\omega)\ne \{0\}$, then ${\rm dim}V^u(\omega)$ must be odd.
\end{proof}

\section{Almost Automorphic and Almost Periodic Minimal Flows}

In this section, we investigate the conditions under which a minimal set of \eqref{equation-lim1}  is almost automorphic or almost periodic in the case that $f=f(t,u,u_x)$ (see Theorem \ref{almost-automorphic-thm}). Moreover, we will also investigate the case that $f(t,u,p)=f(t,u,-p)$ (see Theorem  \ref{almost-periodic-thm}).

\begin{theorem}
\label{almost-automorphic-thm}
Let $M$ be a minimal set of \eqref{equation-lim1}.
\begin{itemize}

\item[{\rm (1)}] If $M$ is spatially homogeneous, then $M$ is topologically conjugate to a minimal flow in $\mathbb{R}\times H(f)$. Moreover, it is an almost $1$-cover of $H(f)$.

\item[{\rm (2)}] If $M$ is linearly stable, then $M$ is spatially homogeneous; and hence, is an almost $1$-cover of $H(f)$.

\item[{\rm (3)}] If $M$ is uniformly stable, then $M$ is  spatially homogeneous and is a $1$-cover of $H(f)$.

\item[{\rm (4)}]   If $M$ is hyperbolic (i.e., ${\rm dim}V^c(\omega)=0$), then $M$ is a $1$-cover of $H(f)$.
\end{itemize}
\end{theorem}
\begin{proof}
(1) Suppose that $M$ is a spatially homogeneous minimal set of \eqref{equation-lim1}. Since $f=f(t,u,u_x)$, $M$ is also a minimal set of
\begin{equation}
\label{scalar-ode}
\dot u=\tilde g(t,u),
\end{equation}
where $\tilde g(t,u)=g(t,u,0)$ and $g\in H(f)$. It then follows from \cite{Shen1998} that $M$ is an almost $1$-cover of $H(f)$.

(2)  Suppose that
$M$ is spatially inhomogeneous. Then for any $(u,g)\in M$, $u$  is spatially inhomogeneous.

Let $\Psi(t,v,(u,g))=(\Phi(t,u,g)v,\Pi^t(u,g))$
be the linearized skew-product semiflow of $M$. By the exponentially separated property of the strongly monotone
skew-product semiflows (see e.g. \cite[p.38]{Shen1998}), there is a continuous invariant splitting $X=X_1(u,g)\oplus X_2(u,g)$ with $X_1(u,g)={\rm span}\{\phi(u,g)\}$, $\phi(u,g)\in {\rm Int}X^{+}$ and $X_2(u,g)\cap X^+=\{0\}$ such that
$$
\Phi(t,u,g)X_1(u,g)=X_1(\Pi^t(u,g)),\quad \Phi(t,u,g)X_2(u,g)\subset X_2(\Pi^t(u,g)).
$$ Moreover, there are $K,\gamma>0$ such that
\begin{equation}\label{exponential-separation}
\|\Phi(t,u,g)|_{X_2(u,g)}\|\le Ke^{-\gamma t}\|\Phi(t,u,v)|_{X_1(u,g)}\|
\end{equation} for any $t\ge 0$ and $(u,g)\in M$. Since $M$ is linearly stable, one has
$$\limsup_{t\to\infty}\frac{\ln\|\Phi(t,u,g)\phi(u,g)\|}{t}\le 0.$$

Given any $(u_0,g_0)\in M$, let
$v(t,x)=\varphi_x(t,x;u_0,g_0)$. Then $\norm{v(t,\cdot)}$ is bounded. Moreover, one can find a $\delta_0>0$ such that $\|v(t,\cdot)\|\ge \delta_0$ for all $t$.
(Otherwise,
there is a sequence $t_n\to\infty$ such that $\|v(t_n,\cdot)\|\to 0$ as $n\to\infty$, which entails that $\varphi_x(t_n,x;u,g)\to 0$ as $n\to\infty$ uniformly in $x\in S^1$. Without loss of generality, we assume that $(\varphi(t_n,\cdot;u,g),g\cdot{t_n})\to (u^*,g^*)\in M$
as $n\to\infty$. Thus, $u^*$ must be spatially homogeneous, a contradiction.)  Therefore, we have
$\lim_{t\to\infty}\frac{\ln\|v(t,\cdot)\|}{t}=0.$
Note also that
$v(t,\cdot)=\Phi(t,u_0,g_0)v(0,\cdot)$. Then
\begin{equation}\label{E:center-dire-inho-1}
\lim_{t\to\infty}\frac{\ln\|\Phi(t,u_0,g_0)v(0,\cdot)\|}{t}=0.
\end{equation}
As a consequence, we have $v(0,\cdot)=\alpha\phi(u_0,g_0)+ \psi(u_0,g_0)$ for some $\alpha\not =0$  and
$\psi(u_0,g_0)\in X_2(u_0,g_0)$.
It follows from \eqref{exponential-separation} that $v(t,\cdot)\in {\rm Int}X^{+}\cup (-{\rm Int}X^{+})$ for $t$ sufficiently large, a contradiction to the periodic boundary condition.
Therefore, $M$ is spatially homogeneous; and hence, $M$ is an almost $1$-$1$ cover of $H(f)$.

(3) Because $M$ is uniformly stable, it then follows from Lemma \ref{stable-imply-linear-stable} (whose proof will be postponed in Section 5 as it is independent of the materials in Section 3) that $M$ is linearly stable. By (1), $M$ is almost $1$-$1$ cover of $H(f)$. Note that the uniform stability of $M$ also implies that it is distal (see \cite[Theorem II 2.8]{Shen1998}).
Therefore, $M$ is $1$-cover of $H(f)$.

(4) Since $\dim V^c(\omega)=0$, $M$ must be spatially homogeneous. Otherwise, by \eqref{E:center-dire-inho-1}, it is easy to see that ${\rm dim}V^c(\omega)\ge 1$. As a consequence, $M$ is a hyperbolic minimal set
of \eqref{scalar-ode}. By  \cite{ShenYi-2}, $M$ is a $1$-$1$ cover of $H(f)$.
\end{proof}

\vskip 2mm
In the following, we will consider the structure of the minimal set $M$ for the case  $f(t,u,p)=f(t,u,-p)$.
\begin{theorem}
\label{almost-periodic-thm}
 Assume that $f(t,u,p)=f(t,u,-p)$. Then any minimal set $M$ of \eqref{equation-lim2} is an almost $1$-cover of $H(f)$.

 Moreover,
$M$ is a $1$-cover of $H(f)$, if one of the following alternatives holds:

\indent $\quad$ {\rm (i)} $M$ is hyperbolic;

\indent $\quad$ {\rm (ii)} ${\rm dim}V^c(\omega)=1$ and $M$ is spatially inhomogeneous.
\end{theorem}
\vskip 2mm

Before proving Theorems \ref{almost-periodic-thm}, we present some lemmas.
For convenience, we also write in the following the solution $\varphi(t,\cdot;u,g)$ as $\varphi(t;u_1,g)(\cdot)$ in the context without any confusion.

\begin{lemma}\label{order-minimal}
  Assume that $f(t,u,p)=f(t,u,-p)$ and $M$ be a minimal invariant set. Then, there is a point $x_0\in S^1$ such that for any $(u,g)\in M$, one has $u_x(x_0)=0$.
\end{lemma}
\begin{proof}
  It can be proved by using similar deductions for Theorem B in \cite{Chen1989160}. See also \cite[Proposition 3.1]{SWZ2} for the detailed proof.
\end{proof}

Let $M$ be a minimal invariant set and assume $f(t,u,p)=f(t,u,-p)$.
 For each $g\in H(f)$ and any $(u,g),(v,g)\in M\cap p^{-1}(g)$, we define within this section the relation between $(u,g),(v,g)$ as
\begin{equation*}
 (u,g)\le_g(v,g) \Longleftrightarrow \text{ there is a }T>0 \text{ such that }
(\varphi(t;u,g)-\varphi(t;v,g))(x_0)\le 0 \text{ for all }t\ge T,
\end{equation*}
where $x_0\in S^1$ is as defined in Lemma \ref{order-minimal}. t is not difficult to see that the relation ``$\le_g$" is a partial order on the fibre $M\cap p^{-1}(g)$, which means that, if
$(u,g)\le_g(v,g)$ and $(v,g)\le_g(u,g)$, then $(u,g)=(v,g)$.

As usual, we say
$$(u,g)<_g(v,g) \Longleftrightarrow (u,g)\le_g(v,g) \textnormal{ and }
u\ne v.$$
We also call $(u,g)\ll_g(v,g)$ if there exist neighborhoods $\mathcal{N}_1$, $\mathcal{N}_2$ of $(u,g)$ and $(v,g)$, respectively, such that $(u_1,g)<_g(v_1,g)$ for any $(u_i,g)\in \mathcal{N}_1,i=1,2.$

\begin{lemma}\label{total-ordering}
Assume that $f(t,u,p)=f(t,u,-p)$.
Then for any $g\in H(f)$ and $(u,g),(v,g)\in M\cap p^{-1}(g)$, one has

{\rm (i)} $(u,g)<_g(v,g)$ if and only if there exists some $T>0$ such
that $$(\varphi(t;u,g)-\varphi(t;v,g))(x_0)<0\,\text{ for all }t\ge T.$$

{\rm (ii)} ``$\le_g$" is a total ordering on $M\cap p^{-1}(g).$

\end{lemma}
\begin{proof}
(i) The sufficiency is obvious. If $(u,g)<_g(v,g)$, then $u\ne v$ and there exists a $T>0$ such
that $(\varphi(t;u,g)-\varphi(t;v,g))(x_0)\le 0$ ($x_0$ is as defined in Lemma \ref{order-minimal}), for all $t\ge T_0.$ Choose a larger $T>0$, if necessary, one may also assume that $\varphi(t,\cdot;u,g)-\varphi(t,\cdot;v,g)$ only possesses simple zeros for each $t\ge T$.  Then Lemma \ref{order-minimal} immediately implies that $(\varphi(t;u,g)-\varphi(t;v,g))(x_0)<0\,\text{ for all }t\ge T_0.$

(ii) Given any two distinct $(u,g),(v,g)\in p^{-1}(g)\cap M$, one may find a $T>0$ such that $\varphi(t,\cdot;u,g)-\varphi(t,\cdot;v,g)$ only possesses simple zeros for each $t\ge T$. Without loss of generality, we assume that $(\varphi(T;u,g)-\varphi(T;v,g))(x_0)<0$. Then we have $(\varphi(t;u,g)-\varphi(t;v,g))(x_0)<0$ for all $t>T$. Otherwise, one can find a $t_0>T$ such that
$(\varphi(t_0;u,g)-\varphi(t_0;v,g))(x_0)=0$. Together with Lemma \ref{order-minimal}, we get a contradiction.
\end{proof}

\begin{lemma}\label{1-cover-trans-lm}
 If $M$ is almost $1$-$1$ cover of $H(f)$, then $M$ is $1$-$1$ if and only if for any $g\in H(f)$, there exists $u\in X$ such that $M\cap p^{-1}(g)\subset (\Sigma u,g)$.
\end{lemma}

\begin{proof}
We only need to prove the sufficiency part.
Suppose that $M$ is not a $1$-$1$ cover.
Then there exists some $g\in H(f)$ with two distinct points $(v_1,g),(v_2,g)\in M \cap p^{-1}(g)$.
By our assumption, we have $\sigma_{\tilde{a}}v_1=v_2$ for some $\tilde{a} \in S^1$ .
Choose some other $g^*\in H(f)$ with $M \cap p^{-1}(g^*)\triangleq\{(w^*,g^*)\}$.
The minimality of $M$ implies that one can find a sequence $t_n\rightarrow \infty$ such that $\Pi^{t_n}(v_1,g)\rightarrow (w^*,g^*)$, and hence, $\Pi^{t_n}(v_2,g)=\Pi^{t_n}(\sigma_{\tilde{a}}v_1,g)\rightarrow (\sigma_{\tilde{a}} w^*,g^*)$. This entails that $\sigma_{\tilde a} w^*=w^*$ (because $M \cap p^{-1}(g^*)=\{(w^*,g^*)\}$).
On the other hand, one can also choose another sequence $s_n\to \infty$ such that $\Pi^{s_n}(w^*,g^*)\rightarrow (v_1,g)$. So $\Pi^{s_n}(\sigma_{\tilde{a}}w^*,g^*)\rightarrow (\sigma_{\tilde{a}}v_1,g)$. Recall that $\sigma_{\tilde a} w^*=w^*$ , one has $\sigma_{\tilde{a}} v_1=v_1$. Thus we obtain that $v_2=v_1$, a contradiction.
\end{proof}

\begin{proof}[Proof of Theorem \ref{almost-periodic-thm}]
Let $Y_0$ be the invariant residual set in Lemma \ref{non-ordered}. Suppose that there exists some $g_0\in
Y_0$ such that $\textnormal{card}(M\cap p^{-1}(g_0))>1$. Then
we define
\begin{equation}\label{def-sigma-g}
 Z(g_0):=\min_{ {\small \begin{array}{c}
(u^1,g_0),(u^2,g_0)\in M\cap p^{-1}(g_0)\\ (u^1,g_0)\ne (u^2,g_0)
\end{array}}
} \left\{\inf_{t>0}z(\varphi(t,\cdot;u^1,g_0)-\varphi(t,\cdot;u^2,g_0))\right\}.
\end{equation}
By virtue of the definition $z$ and Lemma
\ref{sigma-function}(c), there exist
$(\hat{u}^1,g_0),(\hat{u}^2,g_0)\in M\cap p^{-1}(g_0)$ and some
$T>0$ such that
\begin{equation}\label{expression-sigma-g}
 Z(g_0)=z(\varphi(t,\cdot;\hat{u}^1,g_0)-\varphi(t,\cdot;\hat{u}^2,g_0))\qquad \textnormal{ for all } t\ge T.
\end{equation}
By Lemma \ref{total-ordering}(ii), we can assume without loss of generality that
$(\hat{u}^1,g_0)<_{g_0}(\hat{u}^2,g_0)$, and hence,
$$(\varphi(t;\hat{u}^1,g_0)-\varphi(t;\hat{u}^2,g_0))(x_0)<0 \qquad \textnormal{ for all } t\ge T,$$
where $x_0$ is as defined in Lemma \ref{order-minimal}.

Firstly, we will show that $(\hat{u}^1,g_0)\ll_{g_0}(\hat{u}^2,g_0)$. In fact, fix the $T>0$ in
(\ref{expression-sigma-g}), it then follows from the continuity of $z$ that there are neighborhoods
$\tilde{\mathcal{N}}_1,\tilde{\mathcal{N}}_2$ of $(\hat{u}^1,g_0),(\hat{u}^2,g_0)$ in $M$ respectively such that
\begin{equation}\label{time-T-g}
\begin{array}{l}
{\rm (i).} \qquad z(\varphi(T,\cdot;u^1,g)-\varphi(T,\cdot;u^2,g))\equiv Z(g_0);
\\
{\rm (ii).}\qquad (\varphi(T;u^1,g)-\varphi(T;u^2,g))(x_0)<0
\end{array}
\end{equation}
for all $(u^i,g)\in \tilde{\mathcal{N}}_i (i=1,2)$.
Let
$\mathcal{N}_i=\tilde{\mathcal{N}}_i\cap p^{-1}(g_0)$, $(i=1,2)$. Then one has
\begin{equation}\label{time-T-g0}
\begin{array}{l}
{\rm (i).} \qquad z(\varphi(T,\cdot;u^1,g_0)-\varphi(T,\cdot;u^2,g_0))\equiv Z(g_0);
\\
{\rm (ii).}\qquad (\varphi(T;u^1,g_0)-\varphi(T;u^2,g_0))(x_0)<0
\end{array}
\end{equation}
for all $(u^i,g_0)\in \mathcal{N}_i(i=1,2)$. We now {\it claim} that $(\varphi(t;u^1,g_0)-\varphi(t;u^2,g_0))(x_0)<0$, for all $t\ge T$ and
$(u^i,g_0)\in \mathcal{N}_i(i=1,2)$. Suppose on the contrary that
there are some $(u^i,g_0)\in \mathcal{N}_i(i=1,2)$ such that
$(\varphi(T_1;u^1,g_0)-\varphi(T_1;u^2,g_0))(x_0)\ge 0$ for some $T_1>T$. Then, by (\ref{time-T-g0})(ii),
 one can find some $T_2\in (T,T_1]$ such that
$(\varphi(T_2;u^1,g_0)-\varphi(T_2;u^2,g_0))(x_0)= 0$. Recall that $\varphi_x(T_2;u^1,g_0)(x_0)=0=\varphi_x(T_2;u^2,g_0)(x_0)$ in Lemma \ref{order-minimal}. Then
$x_0$ is a multiple zero of the function $\varphi(T_2,\cdot;u^1,g_0)-\varphi(T_2,\cdot;u^2,g_0)$.
So, by virtue of (\ref{time-T-g0})(i) and Lemma \ref{sigma-function}(b),
$$Z(g_0)=z(\varphi(T,\cdot;u^1,g_0)-\varphi(T,\cdot;u^2,g_0))
>z(\varphi(T_1,\cdot;u^1,g_0)-\varphi(T_1,\cdot;u^2,g_0)),$$
which contradicts the minimum definition of $Z(g_0)$ in
(\ref{def-sigma-g}). Thus, we have proved the claim. Together with Lemma
\ref{total-ordering}(1), the claim entails that $(u^1,g_0)<_{g_0}(u^2,g_0)$ for all
$(u^i,g_0)\in \mathcal{N}_i(i=1,2)$. In other words, we have obtained that
$(\hat{u}^1,g_0)\ll_{g_0}(\hat{u}^2,g_0)$.

Next, we will show that $(\hat{u}^1,g_0),(\hat{u}^2,g_0)$ forms
a {\it strongly order-preserving pair}, that is, one can find neighborhoods $U_i$ of $(\hat u_i,g_0)$ $(i=1,2)$ in $M$ respectively, such that whenever $(u_1,g_0),(u_2,g_0)\in M\cap p^{-1}(g_0)$, with $\Pi^{t_0}(u_i,g_0)\in U_i\ (i=1,2)$ for some $t_0<0$, then $(u_1,g_0)\ll_{g_0}(u_2,g_0)$.
 To this end, let us recall the neighborhoods
$\tilde{\mathcal{N}}_i$ of $(\hat{u}^i,g_0)$ $(i=1,2)$ obtained before (\ref{time-T-g}). So, for any $(u^i,g_0)\in M\cap
p^{-1}(g_0)$ with $\Pi^{t_0}(u^i,g_0)\in \tilde{\mathcal{N}}_i(i=1,2)$ for
some $t_0<0$,  it follows from (\ref{time-T-g}) and the
cocycle property of $\varphi$ that
\begin{equation}\label{E:time-T+t0-g0}
\begin{array}{l}
{\rm (i).} \qquad z(\varphi(T+t_0,\cdot;u^1,g_0)-\varphi(T+t_0,\cdot;u^2,g_0))\equiv Z(g_0);
\\
{\rm (ii).}\qquad (\varphi(T+t_0;u^1,g_0)-\varphi(T+t_0;u^2,g_0))(x_0)<0.
\end{array}
\end{equation}
So, similarly as in the claim after (\ref{time-T-g}), we can utilize \eqref{E:time-T+t0-g0} and Lemma \ref{order-minimal} to obtain that
\begin{equation}\label{E:keep-negative-T+t0}
(\varphi(t;u^1,g_0)-\varphi(t;u^2,g_0))(x_0)<0,\,\,\,\text{ for any } t\ge\max\{T+t_0,0\}.
 \end{equation}
 In fact, suppose not, one can repeat the exact same argument in that claim
 to find some $T_2>T+t_0$ such that $x_0$ is a multiple zero of the function $\varphi(T_2,\cdot;u^1,g_0)-\varphi(T_2,\cdot;u^2,g_0)$. As a consequence, Lemma \ref{sigma-function}(b) implies that
 $$z(\varphi(T+t_0,\cdot;u^1,g_0)-\varphi(T+t_0,\cdot;u^2,g_0))
>z(\varphi(t,\cdot;u^1,g_0)-\varphi(t,\cdot;u^2,g_0)),$$ for any $t> \max\{T_2,0\}$. Hence, together with
\eqref{E:time-T+t0-g0}(i), one has $Z(g_0)>z(\varphi(t,\cdot;u^1,g_0)-\varphi(t,\cdot;u^2,g_0))$ for any $t> \max\{T_2,0\}$, which again contradicts the definition of $Z(g_0)$ in (\ref{def-sigma-g}). Thus, \eqref{E:keep-negative-T+t0} has been proved; and moreover, one can also have
\begin{equation}\label{E:keep-negative-T+t0-zg}
z(\varphi(t,\cdot;u^1,g_0)-\varphi(t,\cdot;u^2,g_0))\equiv Z(g_0),\,\,\,\text{ for any } t\ge\max\{T+t_0,0\}.
 \end{equation}
By virtue of \eqref{E:keep-negative-T+t0} and Lemma \ref{total-ordering}(i), we have proved $(u^1,g_0)<_{g_0}(u^2,g_0)$. Furthermore, by utilizing \eqref{E:keep-negative-T+t0}-\eqref{E:keep-negative-T+t0-zg} and  repeating the same argument (for proving $(\hat{u}^1,g_0)\ll_{g_0}(\hat{u}^2,g_0)$ there) in the previous paragraph, we obtain that $(u^1,g_0)\ll_{g_0}(u^2,g_0)$. Thus, we have proved $(\hat{u}^1,g_0),(\hat{u}^2,g_0)$ forms a strongly order-preserving pair.

Therefore, Lemma \ref{non-ordered} directly implies that $\textnormal{card}(M\cap p^{-1}(g_0))=1$, which entails that $M$ is an almost $1$-$1$ cover of $H(f)$.

Finally, we will show that $M$ is a $1$-$1$ cover of $H(f)$ under the assumption (i) or (ii).

(i) Assume that $M$ is hyperbolic. Then this is a special case of Theorem \ref{almost-automorphic-thm}(3).

(ii) Assume that $M$ is ${\rm dim}V^c(\omega)=1$ and spatially inhomogeneous. Then it follows from Theorem \ref{sysm-embed}(2)  that, for any $g\in H(f)$, there is a $u\in X$, such that $M\cap p^{-1}(g)\subset \{(\Sigma u,g)\}$. Recall now that $M$ is almost $1$-$1$ cover of $H(f)$. Then Lemma \ref{1-cover-trans-lm} immediately implies $M$ is $1$-cover of $H(f)$. We have completed the proof.
\end{proof}

\section{Embedding Property of Minimal Sets in the General Case}\label{main result}
  In this section, we focus on the embedding property of minimal set of general spatially-dependent almost-periodic system \eqref{equation-1} with the smooth (say $C^{3}$) nonlinearity $f=f(t,x,u,u_x)$. The following two Theorems are our main results in this section.

\begin{theorem}\label{main-result-thm}
Let $M\subset X\times H(f)$ be a minimal invariant set of $\Pi^t$. Assume that $M$ is stable or linearly stable. Then there
is an invariant and residual set $Y_{0}\subset H(f)$ such that the flow $(M\cap p^{-1}(Y_0),\Pi^t)$ is topologically conjugate to a skew-product flow on some $\hat{M}\subset \mathbb{R}^2\times Y_0$.
\end{theorem}

\begin{theorem}\label{main-corllary}
 Let $\mathcal{O}^+(u,g)$ be a uniformly stable forward orbit in $X\times H(f)$. Then the flow on its $\omega$-limit set $\omega(u,g)$ is topologically conjugate to a skew-product flow on some $\tilde{M}\subset \mathbb{R}^2\times H(f)$.
\end{theorem}

\begin{remark}
{\rm Theorems \ref{main-result-thm}-\ref{main-corllary} have generalized the result of Tere\v{s}\v{c}\'{a}k \cite{Te}
to time almost-periodic systems.
}
\end{remark}

In order to prove the Theorem \ref{main-result-thm}, we first prove the following two lemmas:

\begin{lemma}\label{zero-number ncover}
Let $M\subset X\times H(f)$ be a minimal set of ${\Pi}^t$ which is linearly stable. Then there
is an invariant and residual set $Y_{0}\subset H(f)$ such that for any $g\in Y_0$ and any two distinct elements $(v,g),(w,g)\in M\cap p^{-1}(g)$,  one has
\begin{equation}
z(\varphi(t,\cdot;v,g)-\varphi(t,\cdot;w,g))\equiv {\rm constant}, \,\,\text{ for all }t\in \mathbb{R}.
\end{equation}

\end{lemma}
\begin{proof}
 Since $K$ is linearly stable, it follows from \cite[Theorem II.4.5]{Shen1998} that there exist some integer $N\ge 1$ and some invariant residual set $Y_1\subset H(f)$ such that, for each $g\in Y_1$,  {\rm card}$(M\cap p^{-1}(g))=N.$  Let $Y_0=Y_1\cap Y'$ where $Y'$ is as defined in Lemma \ref{epimorphism-thm}. Clearly, $Y_0$ is residual in $Y$. Then, for any $g\in Y_0$ and any $(v,g),(w,g)\in M\cap p^{-1}(g)$ with $v\neq w$, there is a sequence ${\tau}_{n}\rightarrow +\infty$ as $n\rightarrow \infty$ and $ \{(v_{n},g)\}\subset M\cap p^{-1}(g)$ such that ${\Pi}^{\tau_{n}}(w,g)\rightarrow (w,g)$ and $\Pi^{\tau_{n}}(v_{n},g)\longrightarrow (v,g)$ as $n\rightarrow \infty$.
Recall that card$(M\cap p^{-1}(g))=N<\infty$ for each $g\in Y_0$. One may assume without loss of generality that $v_n\equiv v_1$ for all $n\geq 1$. Consequently, one obtains ${\Pi}^{\tau_n}(v_1,g)\rightarrow (v,g)$ as $n\rightarrow \infty$. Thus, Lemma \ref{sequence-limit} implies the conclusion.
\end{proof}

\begin{lemma}
\label{stable-imply-linear-stable}
 Let $M\subset X\times H(f)$ be a minimal invariant set of the skew-product semiflow \eqref{equation-lim2} generated by \eqref{equation-lim1}. If $M$ is stable, then
$M$ is linearly stable.
\end{lemma}

\begin{proof}
 Assume that $M$ is stable. Then
for any $(u_0,g)\in M$, one has
 \begin{equation}\label{E:solu-suffi-close}
\|\varphi(t,\cdot;u_1,g)-\varphi(t,\cdot;u_0,g)\|\ll 1\,\,\, \text{ for all } t\ge 0,
\end{equation}
whenever $u_1\in X$ with $\|u_1-u_0\|$ being sufficiently small.

 Let
 $v(t,x)=\varphi(t,x;u_1,g)-\varphi(t,x;u_0,g)$. Then $v(t,x)$ satisfies
 \begin{equation}\label{E:diff-eqns-v-stable}
 v_t=v_{xx}+a(t,x)v_x+b(t,x)v,\quad x\in S^1.
 \end{equation}
 Here
 \begin{equation}\label{E:a(t,x)-es}
 a(t,x)=\int_0^ 1 \partial_4 g(t,x,\varphi(t,x;u_0,g), \varphi_x(t,x;u_0,g)+s(\varphi_x(t,x;u_1,g)-\varphi_x(t,x;u_0,g)))ds
 \end{equation}
 and
 $$
 b(t,x)=\int_0^ 1 \partial_3 g(t,x,\varphi(t,x;u_0,g)+s(\varphi(t,x;u_1,g)-\varphi(t,x;u_0,g)),\varphi_x(t,x;u_1,g))ds.
 $$
As in \eqref{E:transform-congr-1} (see \cite[p.247-248]{Chow1995} for the transformations), we let \begin{equation}\label{E:transform-congr-2}
\bar v(t,x):=r(t,x+c(t))\cdot v(t,x+c(t)),
 \end{equation}
 where the function $c(t)$ satisfies
$$\dot c(t)=-a_0(t),\,\,\text{ with }\,\, a_0(t)=\frac{1}{2\pi}\int_0^{2\pi}a(t,z)dz,$$ and
$$
 r(t,x)=\exp\Big(\frac{1}{2}\int_0^ x (a(t,z)-a_0(t))dz\Big).
$$
Then the equation \eqref{E:diff-eqns-v-stable} can be changed into
\begin{equation}\label{no-difference-gradient2}
 \bar v_t=\bar v_{xx}+\bar b(t,x)\bar v,\quad x\in S^1,
 \end{equation}
where $\bar b (t,x)=\tilde b(t,x+c(t))$ with
\begin{equation}\label{E:tilde-b-expression}
\tilde{b}(t,x)=b(t,x)+\frac{1}{2}\int_0^x[a(t,z)-a_0(t)]_tdz
-\frac{1}{4}(a(t,x)^2-a_0(t)^2)-\frac{1}{2}a_x(t,x).
\end{equation}
 \vskip 2mm

On the other hand, we consider the linearized equation, associated with $\varphi(t,x;u_0,g)$,
\begin{equation}\label{linear-variation-M}
 v_t=v_{xx}+a^0(t,x)v_x+b^0(t,x)v,
\end{equation}
 where
\begin{equation}\label{E:a0-(t,x)-es}
 a^0(t,x)=\partial_4 g(t,x,\varphi(t,x;u_0,g),\varphi_x(t,x;u_0,g))
 \end{equation}
 and
 $$
 b^0(t,x)=\partial_3 g(t,x,\varphi(t,x;u_0,g),\varphi_x(t,x;u_0,g))).
 $$
  By similar transformations as in \eqref{E:transform-congr-2}, we introduce
  \begin{equation}\label{E:transform-congr-3}
\bar v^0(t,x):=r^0(t,x+c^0(t))\cdot v^0(t,x+c^0(t)),
 \end{equation}
 for which $r^0(t,x)=\exp\Big(\frac{1}{2}\int_0^ x (a^0(t,z)-a^0_0(t))dz\Big)$, and $\dot c^0(t)=-a_0^0(t)$ with $a_0^0(t)=\frac{1}{2\pi}\int_0^{2\pi}a^0(t,z)dz$.
 So, the equation \eqref{linear-variation-M} can be changed into
 \begin{equation}\label{no-gradient2}
 \bar v^0_t=\bar v^0_{xx}+\bar b^0(t,x)\bar v^0,\quad x\in S^1
 \end{equation}
with $\bar b^0 (t,x)=\tilde b^0(t,x+c^0(t))$, where
\begin{equation}\label{E:tilde-b-0-expression}
\tilde{b}^0(t,x)=b^0(t,x)+\frac{1}{2}\int_0^x[a^0(t,z)-a^0_0(t)]_tdz
-\frac{1}{4}(a^0(t,x)^2-a^0_0(t)^2)-\frac{1}{2}a^0_x(t,x).
\end{equation}

For the coefficients $\bar{b}$ (in \eqref{no-difference-gradient2}) and $\bar b^0$ (in \eqref{no-gradient2}), {\it we claim that: Given any $T>0$ and any $\epsilon>0$, one has
\begin{equation}\label{b-b-bar}
 \bar b(t,\cdot)-\epsilon\le \bar  b^0(t,\cdot)\le \bar  b(t,\cdot)+\epsilon,\,\, \text{ for } \,\, t\ge T/2,
 \end{equation} whenever the initial value $\|u_1-u_0\|$ is sufficiently small.}
 For this purpose, we need to compare each term in \eqref{E:tilde-b-expression} with that in \eqref{E:tilde-b-0-expression}. By virtue of \eqref{E:solu-suffi-close}, it is clear that $|b(t,x)-b^0(t,x)|$ and $|a(t,x)-a^0(t,x)|$ are sufficiently small for any $t\ge 0$, whenever the initial value  $\|u_1-u_0\|$ is sufficiently small. So, it remains to show that $|a_x(t,x)-a^0_x(t,x)|$ and $|a_t(t,x)-a^0_t(t,x)|$ are sufficiently small whenever the initial value  $\|u_1-u_0\|$ is sufficiently small.

 In order to estimate $|a_x(t,x)-a^0_x(t,x)|$, it suffices to estimate the second-order derivative   $\abs{v_{xx}(t,x)}$ (by differentiating \eqref{E:a(t,x)-es} and \eqref{E:a0-(t,x)-es} with respect to $x$, respectively).
  So, fix any $T>0$. By the standard interior Schauder estimate (see \cite[Theorem 3.5]{Friedman}) for the linear equation \eqref{E:diff-eqns-v-stable} on the domain $D_0:=[0,T]\times S^1$, one can obtain that $\abs{v_{xx}(t,x)}\ll 1$ for any $(t,x)\in [T/2,T]\times S^1$, because \eqref{E:solu-suffi-close} guarantees that $\abs{v}_{0,D_0}\triangleq \sup_{D_0}\abs{v(t,x)}$ is sufficiently small.
  Moreover, for any $n\ge 1$, we can even sequently apply the interior Schauder estimate (see \cite[Theorem 3.5]{Friedman}) for \eqref{E:diff-eqns-v-stable} on the domain $D_n:=[\frac{nT}{2},\frac{nT}{2}+T]\times S^1$ to obtain that $\abs{v_{xx}(t,x)}\ll 1$ for any $(t,x)\in [\frac{nT}{2}+\frac{T}{2},\frac{nT}{2}+T]\times S^1$, because \eqref{E:solu-suffi-close} guarantees that $\abs{v}_{0,D_n}\triangleq \sup_{D_n}\abs{v(t,x)}$ is sufficiently small uniformly for all $n\ge 1$. As a consequence, it follows that
  \begin{equation*}\label{E:v-2nd-order-gradiant}
  \abs{v_{xx}(t,x)}\ll 1,\,\text{ for all }t\ge T/2\text{ and }x\in S^1,
  \end{equation*}
  which implies that $|a_x(t,x)-a^0_x(t,x)|$ is sufficiently small for all $t\ge T/2$ and $x\in S^1$.

 As for the estimate of $|a_t(t,x)-a^0_t(t,x)|$, by differentiating \eqref{E:a(t,x)-es} and \eqref{E:a0-(t,x)-es} with respect to $x$ twice, it suffices to estimate the third-order derivative   $\abs{v_{xxx}(t,x)}$. So, together with the fact that the nonlinearity $g\in C^{3}$, one can again apply  the interior Schauder estimate (see \cite[Theorem 3.10]{Friedman}) and use the similar arguments above to obtain that $\abs{v_{xxx}(t,x)}\ll 1,\,\text{ for all }t\ge T/2\text{ and }x\in S^1$; and hence,  $|a_t(t,x)-a^0_t(t,x)|$ is sufficiently small for all $t\ge T/2$ and $x\in S^1$. Thus, we have proved the claim.


Based on the claim, for any $\epsilon>0$ and $T>0$, one can choose some $u_1\in X$ with $\|u_1-u_0\|$ being sufficiently small and $u_1-u_0\in {\rm Int}X^+$ such that \eqref{b-b-bar} holds.
 As a consequence, by the comparison principle for parabolic equations, the corresponding solutions of \eqref{no-difference-gradient2} and \eqref{no-gradient2} satisfy
 \begin{equation}\label{E:comparison-prin}
- Ce^{-\epsilon t}\bar v(t,\cdot)\le \bar v^0(t,\cdot)\le  Ce^{\epsilon t}\bar v(t,\cdot)\quad \forall t\ge T/2,
 \end{equation}
 where
  $C>0$ is such that the corresponding initial value satisfies $$-C \bar v(0,\cdot)\le\bar  v^0(0,\cdot)\le C \bar v(0,\cdot),$$ where $\bar v^0(0,\cdot)\in \mathrm{Int} X^+$ with $\|\bar v^0(0,\cdot)\|=1$. Observe also that  $\norm{\bar v(t,\cdot)}$ is bounded for all $t\ge 0$. Then, \eqref{E:comparison-prin} implies that $\sup_{t\ge T/2,x\in S^1}e^{-\varepsilon t}\abs{\bar v^0(t,\cdot)}$ is bounded. By applying the  Schauder estimate again as above, we obtain that
 $e^{-\varepsilon t}\|\bar v^0(t,\cdot)\|$ is bounded for all $t\ge T$. It then entails that
 $$
  \limsup_{t\to\infty}\frac{\ln \|\bar v^0(t,\cdot)\|}{t}\le \epsilon.
 $$
 It then follows from Lemma \ref{upper-lya} that
 \begin{equation}\label{E:Lypunov-control}
 \limsup_{t\to\infty}\frac{\ln \|\bar\Phi(t;u_0,g)\|}{t}\le \epsilon,
  \end{equation}
  where $\bar \Phi(t;u_0,g)$ is the solution operator of the transformed equation \eqref{no-gradient2} (from linearized equation \eqref{linear-variation-M} associated with $\varphi(t,x;u_0,g)$).
   By arbitrariness of $\epsilon$ and arbitrariness of $(u_0,g)\in M$, we have $\bar\lambda_M\le 0$, where $\bar\lambda_M$ is the upper Lyapunov exponent of \eqref{no-gradient2} on $M$. Note also that,
as in Remark \ref{no-chang-spetr}, the transformation form \eqref{E:transform-congr-3} also preserves the Lypapunov exponents of \eqref{no-gradient2} from those of \eqref{linear-variation-M}. Then the upper Lyapunov exponent of \eqref{linear-variation-M} on $M$ is also non-positive, which entails that $M$ is linearly-stable.
 The proof of this lemma is completed.

  \end{proof}


\begin{proof}[Proof of Theorem \ref{main-result-thm}]
Due to Lemma \ref{stable-imply-linear-stable}, we only need to consider the case that $M$ is linearly stable. Fix some $x_0 \in S^1$, we define the following mapping:
\begin{equation}\label{embeding-map}
\chi:M (\subset X\times H(f))\longrightarrow \mathbb{R}^2\times H(f);(v,g)\longmapsto (v(x_0),v_x(x_0),g).
\end{equation} Clearly, $\chi$ is continuous and onto $\chi(M)\subset \mathbb{R}^2\times H(f)$.
  Moreover, we can obtain that $\chi|_{_{M\cap p^{-1}(Y_0)}}$ is injective, where $Y_0\subset H(f)$ is defined in Lemma \ref{zero-number ncover}. In fact, choose any $g\in Y_0$ and two distinct elements $(v,g),(w,g)\in M\cap p^{-1}(g)$. It then follows from Lemma \ref{zero-number ncover} that
$z(\varphi(t,\cdot;v,g)-\varphi(t,\cdot;w,g))\equiv$ constant for all $t\in \mathbb{R}$. So, $v-w\in X$ possesses only simple zeros. Therefore, $(v(x_{0}),v_{x}(x_{0}))\neq(w(x_{0}),w_{x}(x_{0}))$, which implies that $\chi(v,g)\neq\chi(w,g)$, and hence, $\chi|_{_{M\cap p^{-1}(Y_0)}}$ is injective.

Let $\hat{M}\triangleq \{\chi(v,g)\in \chi(M):g\in Y_0\}\subset \mathbb{R}^2\times Y_0$. Since $\chi|_{_{M\cap p^{-1}(Y_0)}}$ is injective and onto $\hat{M}$, $\Pi^t$ naturally induces a (skew-product) flow $\hat{\Pi}^t$ on $\hat{M}$ as:
\begin{equation}\label{topological-conju}
\hat{\Pi}^t(\chi(v,g))\triangleq\chi(\varphi(t,\cdot;v,g),g\cdot t)\,\,\,\text{ for any }\chi(v,g)\in \hat{M}.
\end{equation}
We will show that the map $(\chi|_{_{M\cap p^{-1}(Y_0)}})^{-1}$ is also continuous from $\hat{M}$ to $M\cap p^{-1}(Y_0)$. Indeed, let $\chi(v^n,g^n)\to \chi(v,g)$ in $\hat{M}$ (that is, $(v^n(x_0),v^n_x(x_0),g^n)\to (v(x_0),v_x(x_0),g)$ with $g^n\to g$ in $Y_0$). By the compactness of $M$,
one may assume without loss of generality that $(v^n,g^n)\to (w,g)\in M$. This then implies that
$(v(x_0),v_x(x_0))=(w(x_0),w_x(x_0))$. Recall that $(v,g),(w,g)\in M$ with $g\in Y_0$. Suppose that $v\ne w$. Then Lemma \ref{zero-number ncover} implies that $v-w$ possesses only simple zeros, a contradiction. Consequently, $v=w$, and hence,  $(v^n,g^n)\to (v,g)\in M$. Thus, we have proved  $(\chi|_{_{M\cap p^{-1}(Y_0)}})^{-1}$ is continuous from $\hat{M}$ to $M\cap p^{-1}(Y_0)$. By virtue of \eqref{topological-conju}, $(M\cap p^{-1}(Y_0),\Pi^t)$ is topologically conjugate to the flow $(\hat{M},\hat{\Pi}^t)$ on $\mathbb{R}^2\times Y_0$.
\end{proof}

\begin{proof}[Proof of Theorem \ref{main-corllary}]
Since $(H(f),\mathbb{R})$ is minimal and distal, the $\omega$-limit set $\omega(u,g)$ of the semi-orbit $\mathcal{O}^+(u,g)$ for the skew-product semiflow $\Pi^t$ is a minimal set which admits a distal flow extension (See \cite[Theorem II.2.8]{Shen1998}). It then follows from Lemma \ref{epimorphism-thm} that the residual set $Y_0=H(f)$. So, it follows from Theorem \ref{main-result-thm} that $\omega(u,g)$  is topologically conjugate to a skew-product flow on some $\tilde{M}\subset \mathbb{R}^2\times H(f)$.
\end{proof}

\section*{Acknowledgements}
The authors are greatly indebted to the anonymous referee for very careful reading and providing lots of very inspiring and helpful comments and suggestions which led to much improvement of two earlier versions of this paper.

\end{document}